\documentclass[letterpaper,10pt,reqno,onefignum,onetabnum]{amsart}
\usepackage[english]{babel}
\usepackage{amsmath}
\usepackage{amsthm}
\usepackage{verbatim}
\usepackage{mathrsfs}
\usepackage{bm}
\usepackage[foot]{amsaddr}
\usepackage[dvipsnames]{xcolor}
\usepackage{hyperref}
\hypersetup{
	colorlinks = true,
	linkcolor = OliveGreen,
	anchorcolor = OliveGreen,
	citecolor = OliveGreen,
	filecolor = OliveGreen,
	urlcolor = OliveGreen
}
\usepackage{algorithm}
\usepackage{algorithmic}
\usepackage{float}
\usepackage{lipsum}
\usepackage{amsfonts}
\usepackage{amssymb}
\usepackage{graphicx}
\usepackage{epstopdf}
\usepackage{cases}
\usepackage{multirow}

\ifpdf
\DeclareGraphicsExtensions{.eps,.pdf,.png,.jpg}
\else
\DeclareGraphicsExtensions{.eps}
\fi
\usepackage{verbatim}
\usepackage{mathrsfs}
\usepackage{bm}

\newtheorem{teo}{Theorem}[section]
\newtheorem{prop}[teo]{Proposition}
\newtheorem{lem}[teo]{Lemma}

\newtheorem{pro}[teo]{Problem}
\newtheorem{algo}[teo]{Algorithm}
\newtheorem{asume}[teo]{Assumption}

\newtheorem{rem}[teo]{Remark}

\theoremstyle{definition}
\newtheorem{ini}[teo]{Initialization}
\newtheorem{claim}[teo]{Claim}

\usepackage{lipsum}
\usepackage{amsfonts}
\usepackage{amssymb}
\usepackage{graphicx}
\usepackage{epstopdf}

\usepackage{cases}
\usepackage{multirow}
\usepackage{tikz}
\usepackage{booktabs}%
\usepackage{hhline}
\usetikzlibrary{matrix}
\usetikzlibrary{arrows}
\usetikzlibrary{calc}
\ifpdf
\DeclareGraphicsExtensions{.eps,.pdf,.png,.jpg}
\else
\DeclareGraphicsExtensions{.eps}
\fi
\usepackage{verbatim}
\usepackage{mathrsfs}
\usepackage{bm}
\usepackage{color}
\usepackage{xcolor}
\usepackage{caption}
\usepackage{subcaption}
\usepackage{enumitem}
\usepackage{pgfplots}
\pgfplotsset{compat=1.18}

\newcommand{\N}{\mathbb N}

\newcommand{\R}{\mathbb R}

\renewcommand{\H}{\mathcal{H}}
\newcommand{\G}{\mathcal G}

\newcommand{\HH}{{\bm{\mathcal{H}}}}

\newcommand{\Id}{{\bf Id}}
\newcommand{\id}{\textnormal{Id}}
\newcommand{\x}{\bm x}
\newcommand{\y}{\bm y}
\newcommand{\w}{\bm w}
\newcommand{\z}{\bm z}
\newcommand{\s}{\sigma}
\newcommand{\T}{\tau}
\newcommand{\weak}{\rightharpoonup}

\newcommand{\zer}{\textnormal{zer}}

\newcommand{\gra}{\textnormal{gra}\,}

\newcommand{\argm}[1]{\underset{#1}{\argmin\, }}
\newcommand{\cblue}[1]{%
	{\color{black}%
		\ifmmode
		#1
		\else
		\textnormal{#1}
		\fi
}}

\newcommand{\scal}[2]{{\left\langle{{#1}\mid{#2}}\right\rangle}}

\newcommand{\menge}[2]{\big\{{#1}~\big |~{#2}\big\}}
\newcommand{\pinf}{\ensuremath{{+\infty}}}

\newcommand{\RR}{\ensuremath{\mathbb{R}}}
\newcommand{\RP}{\ensuremath{\left[0,+\infty\right[}}

\newcommand{\RPP}{\ensuremath{\left]0,+\infty\right[}}

\newcommand{\RX}{\ensuremath{\left]-\infty,+\infty\right]}}

\newcommand{\prox}{\ensuremath{\text{\rm prox}\,}}

\newcommand{\weakly}{\ensuremath{\:\rightharpoonup\:}}

\usepackage{geometry}
\numberwithin{equation}{section}

\numberwithin{equation}{section}

\floatname{algorithm}{Line Search}

\DeclareFontEncoding{FMS}{}{}
\DeclareFontSubstitution{FMS}{futm}{m}{n}
\DeclareFontEncoding{FMX}{}{}
\DeclareFontSubstitution{FMX}{futm}{m}{n}
\DeclareSymbolFont{fouriersymbols}{FMS}{futm}{m}{n}
\DeclareSymbolFont{fourierlargesymbols}{FMX}{futm}{m}{n}
\DeclareMathDelimiter{\nr}{\mathord}{fouriersymbols}{152}{fourierlargesymbols}{147}

\DeclareMathOperator*{\argmin}{arg\,min}
\DeclareMathDelimiter{\nr}{\mathord}{fouriersymbols}{152}{fourierlargesymbols}{147}
\DeclareMathAlphabet{\mathpzc}{OT1}{pzc}{m}{it}

\usepackage{lineno} 

\title[Relaxed and Inertial Nonlinear Forward-Backward Algorithm]{Relaxed and Inertial Nonlinear Forward-Backward Algorithm}

\author{Juan José Maulén$^{1,2}$}
\author{Fernando Rold\'an$^3$}
\author{Cristian Vega$^4$}
\address{$^1$ Instituto de Ciencias de la Ingeniería, Universidad de O'Higgins, Rancagua, Chile. {\it 
		E-mail address:} 
	{\sf{juan.maulen@postdoc.uoh.cl}}.}
\address{$^2$ Centro de Modelamiento Matemático (CMM), Universidad de Chile, Chile.}
\address{$^3$ Departamento de Ingeniería Matemática, Universidad de Concepción, Concepción, Chile. {\it 
		E-mail address:} 
	{\sf{fernandoroldan@udec.cl}}. }
\address{$4$ Instituto de Alta investigaci\'on (IAI), Universidad de Tarapac\'a,
	Arica, Chile. {\it 
		E-mail address: } 
	{\sf{cristianvegacereno6@gmail.com}}. }
\begin{document}
	\begin{abstract}
		The Nonlinear Forward-Backward (NFB) algorithm, also known as \textit{warped resolvent} iterations, is a splitting method for finding zeros of sums of monotone operators. In particular cases, NFB reduces to well-known algorithms such as Forward-Backward, Forward-Backward-Forward, Chambolle--Pock, and Condat--V\~u. Therefore, NFB can be used to solve monotone inclusions involving sums of maximally monotone, cocoercive, monotone and Lipschitz operators as well as linear compositions terms. In this article, we study the \cblue{weak and strong (linear)} convergence of NFB with inertial and relaxation steps. Our results recover known convergence guarantees for the aforementioned methods when extended with inertial and relaxation terms. Additionally, we establish the convergence of inertial and relaxed variants of the Forward-Backward-Half-Forward and Forward-Primal--Dual-Half-Forward algorithms, which, to the best of our knowledge, are new contributions. We consider both nondecreasing and decreasing sequences of inertial parameters, the latter being a novel approach in the context of inertial algorithms. To evaluate the performance of these strategies, we present numerical experiments on optimization problems with affine constraints and on image restoration tasks. Our results show that decreasing inertial sequences can accelerate the numerical convergence of the algorithms.
		\par
		\bigskip
		
		\noindent \textbf{Keywords.} {\it Operator splitting, monotone operators, monotone inclusion, inertial methods, convex optimization}
		\par
		\bigskip \noindent
		2020 {\it Mathematics Subject Classification.}{ 47H05, 65K05,
			65K15, 90C25}
		
	\end{abstract}
	
	\maketitle
	
	\section{Introduction}
	In this article, we study an inertial and relaxed version of the Nonlinear Forward-Backward algorithm (NFB) \cite{BuiCombettesWarped2020,Giselsson2021NFBS} (see also \cite{Combettes2024AN}). This method can be applied to numerically solve a broad class of monotone inclusion problems, such as the following:
	\begin{pro}\label{pro:intro}
		Let $\HH$ be a real Hilbert space, let $\bm{A} \colon \HH \to 2^{\HH}$ be a maximally monotone operator, and let $\bm{C} \colon \HH \to \HH$ be a $\beta$-cocoercive operator with $\beta \in \RPP$. The problem is to 
		\begin{equation}
			\text{find } \x \in \HH \text{ such that } 0 \in \bm{A}\x+\bm{C}\x,
		\end{equation}
		under the assumption that its solution set is nonempty.
	\end{pro}
	This problem encompasses several applications such as mechanical problems \cite{Gabay1983,Glowinsky1975,Goldstein1964}, differential inclusions \cite{AubinHelene2009,Showalter1997}, convex programming \cite{Combettes2018MP}, game theory \cite{Nash13,BricenoLopez2019}, data science \cite{CombettesPesquet2021strategies}, image processing \cite{BotHendrich2014TV,Briceno2011ImRe}, traffic theory \cite{Nets1,Fuku96,GafniBert84}, among other disciplines.  
	
	In the context of Problem~\ref{pro:intro}, given $\z_0 \in \HH$, NFB iterates as follows:
	\begin{equation}\label{eq:NFBgis}
		(\forall n\in\N)\quad 
		\begin{array}{l}
			\left\lfloor
			\begin{array}{l}
				\x_n=( \bm{M}_n+\bm{A})^{-1}(\bm{M}_n-\bm{C})\z_n,\\
				\z_{n+1}=\z_n - \theta_n\mu_n\bm{S}^{-1}(\bm{M}_n \z_n -\bm{M}_n\x_n),
			\end{array}
			\right.
		\end{array}
	\end{equation}
	where $\bm{S} \colon \HH \to \HH$ is a linear bounded definite positive operator, for every $n \in \N$, $\bm{M}_n \colon \HH \to \HH$ is a monotone operator, $(\theta_n) \in~ ]0,2[$, and $(\mu_n)_{n \in \N}$ is an adequate sequence in $\RPP$ (see \eqref{eq:mun}). For guaranteeing the convergence of NFB to a solution to Problem~\ref{pro:intro}, the authors in \cite{Giselsson2021NFBS} assume that, for every $n \in \N$, $\bm{M}_n \colon \HH \to \HH$ is Lipschitz continuous and 1-strongly monotone with respect to a linear, bounded and definite positive operator $\bm{P}\colon \HH \to \HH$. In particular instances, the iterations defined by  \eqref{eq:NFBgis} reduce to well known optimization methods. 
	
	\begin{itemize}
		\item \underline{\textbf{Forward-Backward}}: If $\bm{S} = \Id$ and, for every $n \in \N$, $\theta_n\mu_n = \gamma$, $\bm M_n = \Id / \gamma$, for $\gamma \in \RPP$, NFB reduces to the classic Forward-Backward algorithm (FB) \cite{Goldstein1964,Lions1979SIAM,passty1979JMAA}. Moreover, since the Chambolle--Pock algorithm \cite{ChambollePock2011} and the Condat--V\~u algorithm \cite{Condat13,Vu13}  are preconditioned versions of FB applied to primal-dual inclusions, NFB also covers these algorithms.
		
		\item \underline{\textbf{Forward-Backward-Forward}}:  Given a maximally monotone operator $A \colon \HH \to 2^\HH$ and a Lipschitz operator $B \colon \HH \to \HH$, by defining $\bm{A}=A+B$, $\bm{S}=\Id$, $\theta_n\mu_n = \gamma$, for every $n \in \N$, $\bm M_n = \Id / \gamma-B$ and assuming that $\bm{C}=0$, the Problem \eqref{pro:intro} reduces to 	
		\begin{equation}
			\text{find } \x \in \HH \text{ such that } 0 \in A\x + B\x,
		\end{equation}
		and NFB reduces to the Forward-Backward-Forward method (FBF) \cite{Tseng2000SIAM}. 
		\cblue{Given $z_0 \in \HH$,  FBF iterates as follows:				
			\begin{equation}\label{eq:algoFBFintro}
				(\forall n\in\N)\quad 
				\begin{array}{l}
					\left\lfloor
					\begin{array}{l}
						x_{n} = J_{\tau A} (z_n-\tau Dz_{n}) 
						) \\
						z_{n+1} = x_{n}-\tau(Dx_{n}-Dz_{n}).
					\end{array}
					\right.
				\end{array}
		\end{equation}}
		
		\item \underline{\textbf{Forward-Backward-Half-Forward}}: considering the same setting as before but with $\bm{C}\neq 0$, the Forward-Backward-Half-Forward algorithm (FBHF) \cite{BricenoDavis2018} is recovered. Note that, although the structure of FBHF can be recovered naturally, this is not the case for the condition on the steps guaranteeing convergence for FBHF, which deduction is laborious (see \cite[Appendix~A.3]{Giselsson2021NFBS}). \cblue{Given $z_0 \in \HH$,  FBHF iterates as follows:				
			\begin{equation}\label{eq:algoFBHFintro}
				(\forall n\in\N)\quad 
				\begin{array}{l}
					\left\lfloor
					\begin{array}{l}
						x_{n} = J_{\tau A} (z_n-\tau (Dz_{n}+\bm{C} 
						z_n) 
						) \\
						z_{n+1} = x_{n}-\tau(Dx_{n}-Dz_{n}).
					\end{array}
					\right.
				\end{array}
		\end{equation}}
	\end{itemize}	
	These specific instances highlight the versatility of NFB, as it generalizes several methods from the literature and enables the solution of problems involving Lipschitz operators, primal-dual inclusions, linear compositions, among others. In \cite{MorinBanertGiselsson2022}, the authors proposed a nonlinear forward-backward with momentum correction term (NFBM), also used for solving Problem~\ref{pro:intro}. Moreover, NFBM also allows recovering and extend existing methods in the literature, such as, forward-half-reflected-backward (FHRB) \cite{Malitsky2020SIAMJO}, Chambolle--Pock, Condat--V\~u, Douglas--Rachford \cite{Eckstein1992,Lions1979SIAM}, among others. \cblue{However, NFBM cannot recover Tseng-type algorithms such as FBF, FBHF, and forward-primal-dual-half-forward (FPDHF)~\cite{Roldan20254op}}. Note that NFBM also considers the {\it warped resolvent} $(\bm{M}_n+\bm{A})^{-1}$ on its iterates. The authors in \cite{MorinBanertGiselsson2022}, derived the convergence of NFBM under the assumption that, for every $n \in \N$, $\gamma_n\bm{M}_n-\bm{S}$ is $\zeta_n$-Lipschitz with respect to $\bm{S}$, for $\zeta_n \in ]0,1[$ and $\gamma_n \in \RPP$. In particular, if this condition holds, it is possible to choose $\bm{P}$ so that $\bm{M}_n$ is Lipschitz continuous and $1$-strongly monotone with respect to $\bm{P}$ \cite[Proposition~2.1]{MorinBanertGiselsson2022}, thereby recovering the assumptions on $\bm{M}_n$ for NFB made in \cite{Giselsson2021NFBS}.
	

	Several works have incorporated inertial and relaxation steps into splitting algorithms for solving monotone inclusions and it has been shown that inertial and relaxation steps can improve the performance of the algorithm. For instance, inertial proximal algorithms were introduced in \cite{Alvarez2001} and further developed in works such as \cite{AttouchCabot19,AttouchPeypouquet2019,Moudafi2003,Thang2022,PesquetPustelnik2012}. The inertial forward-backward algorithm was first studied in \cite{Alvarez2004} and later has been studied, for example, in \cite{AttouchCabot19,Attouch2014,AttouchPeypouquet2016,BeckTeboulle2009,ChambolleDossal2015,LorenzInFB2015,Bot2016NA,Thang2025}. Inertial and relaxed variants of the forward-backward-forward algorithm have been analyzed in \cite{Bot2016NA,Bot2016Tsengine,BotrelaxFBF2023,TranDinh2025}.
	Moreover, inertial versions of the Douglas--Rachford and Chambolle--Pock algorithms have been studied, for example, in \cite{Alves2020,Bot2015AMC} and \cite{He2022inertial,Valkonen2020}, respectively. The Krasnosel'ski\u{\i}–Mann algorithm, which generalize some methods mentioned above, has been extended to incorporate inertial and relaxation steps in \cite{CombettesGlaudin2017,CortildPeypouquet2024,Dong2022new,MaulenFierroPeypouquet2023}. \cblue{The inertial version of the reflected-forward-backward \cite{Cevher2020SVVA} has been studied in \cite{Bot2025RFB}.} In the recent paper \cite{RoldanVega2025}, the authors presented an inertial and relaxed version of NFBM, and consequently, of all the methods generalized by NFBM. 
	
	In this article, we analyze the convergence of NFB, including inertial and relaxation steps, under the assumption that $\gamma_n \bm{M}_n - \bm{S}$ is $\zeta_n$-Lipschitz with respect to $\bm{S}$. This assumption simplifies the derivation of convergence conditions on the step-size, relaxation, and inertial parameters, making it easier to recover the corresponding conditions for particular cases of NFB (see, for instance, Remark~\ref{rem:FPDHF}\ref{rem:FPDHF1} and \cite[Appendix~A.3]{Giselsson2021NFBS}). \cblue{We also derive the linear convergence of NFB with inertial and relaxation steps in the case where $\bm{A}$ is strongly monotone}. As a result, we recover and extend the relaxed and inertial versions of FB, FBF, FBHF, Chambolle--Pock, Condat--V\~u, and FPDHF. In particular, we consider the case where the sequence of inertial parameters is decreasing, which, to the best of our knowledge, is the first approach that considers this type of inertial strategy, \cblue{whereas they are usually assumed to be nondecreasing, for instance in \cite{Bot2016MTA,Bot2016NA,Bot2015AMC,BotrelaxFBF2023,LorenzInFB2015,MaulenFierroPeypouquet2023}}. To facilitate implementation and the choice of step-sizes, relaxation parameters, and inertial parameters, we propose an initialization procedure for running these algorithms. We also discuss how the original assumptions made in~\cite{Giselsson2021NFBS} for including inertial and relaxation steps in NFB can be incorporated into our analysis. Finally, we present numerical experiments testing the benefits of inertial and relaxation steps in NFB, applied to optimization problems with affine constraints and image restoration tasks, \cblue{both in the monotone and in the strongly monotone case}.
	
	The main contributions of this article are summarized as follows:
	\begin{itemize}
		\item We incorporate inertial and relaxation steps into NFB.
		\item We recover convergence results for inertial and relaxed versions of FB, FBF, Chambolle--Pock, and Condat--V\~u, and extend them to the case of decreasing inertial parameters.
		\item We propose relaxed and inertial versions of FBHF and FPDHF. In particular, we show that FPDHF correspond to a NFB iteration.
		\item We provide numerical evidence showing that decreasing inertial parameters can accelerate convergence.
	\end{itemize}
	The article is structured as follows. In Section~\ref{sec:pre}, we introduce our notation and preliminaries. Section~\ref{sec:conv} presents our main convergence results for inertial and relaxed NFB. This section also discusses how the assumptions in~\cite{Giselsson2021NFBS} relate to our framework and analyzes the conditions on the inertial and relaxation parameters to ensure convergence. In Section~\ref{sec:particularcases}, we derive relaxed and inertial versions of FB, FBF, FBHF, Chambolle--Pock, Condat--V\~u, and FPDHF as particular instances of NFB. Section~\ref{se:NE} presents numerical experiments in optimization with affine constraints and image restoration. Finally, conclusions are given in Section~\ref{sec:conclu}.
		\section{Preliminaries and Notation}\label{sec:pre}
		Let $\H$ be a real Hilbert space with inner product $\scal{\cdot}{\cdot}$, we denote the induced norm by the inner product by $\|\cdot\|$. The symbols $\weakly$ and $\to$ \cblue{denote} the weak and strong 
		convergence, respectively.
		Let $S \colon \H \to \H$ be a self adjoint positive definite linear bounded operator, we denote $\scal{\cdot}{\cdot}_S = \scal{\cdot}{S(\cdot)}$ and $\|\cdot\|_S=\sqrt{\scal{\cdot}{\cdot}_S}$. Note that $(\H,\scal{\cdot}{\cdot}_S)$ is a real Hilbert space \cblue{and we have
			\begin{equation}\label{eq:normSmS}
				(\forall x \in \H )  \quad \|S^{-1}x\|_{S} = \|x\|_{S^{-1}}
			\end{equation}
			and
			\begin{equation}\label{eq:scalSmS}
				(\forall x \in \H )  (\forall y \in \H ) \quad 2\scal{x}{y} = 2\scal{x}{S^{-1}y}_S = \|Sx-y\|^2_{S^{-1}}-\|x\|^2_S -\|y\|^2_{S^{-1}}.
			\end{equation}
			Moreover, it follows from \cite[Corollary~2.15]{bauschkebook2017} that
			\begin{equation}\label{eq:normaalpha}
				(\forall x \in \H) (\forall y \in \H) (\forall \alpha \in \R) \quad \|\alpha x +(1-\alpha)y\|^2_S = \alpha\|x\|^2_S+(1-\alpha)\|y\|^2_S-\alpha(1-\alpha)\|x-y\|^2_S.
		\end{equation}}
		Let $T\colon \H \rightarrow 
		\H$, and $\beta \in \left]0,+\infty\right[$. The operator $T$ is 
		$\beta$-cocoercive with respect to $S$ if  
		\begin{equation}\label{def:cocowrtS}
			(\forall (x,y) \in \H^2) \quad \scal{x-y}{Tx-Ty}
			\geq \beta \|Tx - Ty 
			\|^2_{S^{-1}}.
		\end{equation}
		The operator $T$ is $\beta$-Lipschitzian with respect to $S$ if
		\begin{equation}\label{def:lipwrtS}
			\|Tx-Ty\|_{S^{-1}} \leq  
			\beta\|x - y \|_{S}.
		\end{equation}
		Let $A\colon \H \to 2^\H$ be a set-valued operator, we define the graph of $A$ and the set of zeros of $A$, respectively, by
		\begin{align}\label{eq:defgraph}
			&\gra A = \menge{(x,u) \in \H\times \H}{ u \in Ax},\\
			&\zer A = \menge{x \in \H }{0 \in Ax}.
		\end{align}
		The inverse of $A$ is $A^{-1}\colon \H  \to 2^\H \colon u 
		\mapsto 
		\menge{x \in \H}{u \in Ax}$ and the resolvent of $A$ is 
		$J_A=(\id+A)^{-1}$. \cblue{Let $s \in \R$}, the operator $A$ is \cblue{s}-monotone \cblue{with respect to $S$} if 
		\begin{equation}\label{def:monotone}
			(\forall (x,u) \in \gra A) (\forall (y,v) \in \gra A)\quad \scal{x-y}{u-v} \geq \cblue{s\|u-v\|^2_S}.
		\end{equation}
		\cblue{If $s=0$, $A$ is called monotone and, if $s>0$, $A$ is called $s$-strongly monotone. Moreover, $A$} is maximally monotone if it is monotone and there exists no 
		monotone operator $B :\H\to  2^{\H}$ such that $\gra B$ properly 
		contains $\gra A$. If $A$ is maximally monotone, then $J_A$ is single valued and has full domain. We say that $A$ is $\beta$-strongly monotone with respect to $S$, for $\beta >0$, if 
		\begin{equation}\label{def:Smonotone}
			(\forall (x,u) \in \gra A) (\forall (y,v) \in \gra A)\quad \scal{x-y}{u-v} \geq \beta \|x-y\|^2_S.
		\end{equation}
		We denote by $\Gamma_0(\H)$ the class of proper lower 
		semicontinuous convex functions $f\colon\H\to\RX$. Let 
		$f\in\Gamma_0(\H)$.
		The Fenchel conjugate of $f$ is given by 
		\begin{equation}\label{eq:deff*}
			f^*\colon u\mapsto \sup_{x\in\H}(\scal{x}{u}-f(x))
		\end{equation}
		and we have $ f^* \in\Gamma_0(\H)$. The subdifferential of $f$ is the set-valued operator defined by
		$$\partial f\colon x\mapsto \menge{u\in\H}{(\forall y\in\H)\:\: 
			f(x)+\scal{y-x}{u}\le f(y)}.$$
		It follows that $\partial f$ is a maximally monotone operator, that $(\partial f)^{-1}=\partial f^*$, and that
		$\zer\partial f$ is the set of 
		minimizers of $f$, which is denoted by $\arg\min_{x\in \H}f$. 
		The proximity operator of $f$ is defined by
		\begin{equation}
			\label{e:prox}
			\prox_{f}\colon 
			x\mapsto\argm{y\in\H}\big(f(y)+\frac{1}{2}\|x-y\|^2\big).
		\end{equation} 
		Note that $\prox_f=J_{\partial f}$.
		Given a nonempty closed convex set $C\subset\H$, we denote by 
		$P_C$ the projection onto $C$ and by
		$\iota_C\in\Gamma_0(\H)$ the indicator function of $C$, which 
		takes the value $0$ in $C$ and $\pinf$ otherwise. We denote by $\mathcal{N}_C = \partial \iota_C$ the normal cone to $C$.
		For further properties of monotone operators,
		nonexpansive mappings, and convex analysis, the 
		reader is referred to \cite{bauschkebook2017}.
		
		\cblue{ We denote by $\ell_+^1(\N)$ the set of nonnegative summable sequences.
			The following lemmas will be used throughout the article.
			\begin{lem}[Lemma~11 \cite{BotrelaxFBF2023}]\label{lemmabot}
				Let $(a_n)_{n \in \N}$ be a sequence in $\RP$, let  $(b_n)_{n \in \N}$ be a sequence in $\ell_+^1(\N)$, and let $(\alpha_n)_{n \in \N}$ be a sequence in $[0,\alpha]$ for
				$\alpha \in [0,1[$. Suppose that
				\begin{equation*}
					(\forall n \geq 1) \quad	a_{n+1} \leq a_n + \alpha_n(a_n-a_{n-1}) + b_n.
				\end{equation*}
				Then $(a_n)_{n \in \N}$ converges.	
			\end{lem}
			\begin{lem}[Lemma~5.31 \cite{bauschkebook2017}]\label{lemmaquasifej}
				Let $(a_n)_{n \in \N}$ and $(b_n)_{n \in \N}$ be sequences in $\RP$, and let $(t_n)_{n \in \N}$ and  $(r_n)_{n \in \N}$  be sequences in $\ell_+^1(\N)$ such that
				\begin{equation*}
					(\forall n \in \N) \quad a_{n+1} \leq (1+t_n) a_n - b_n + r_n.
				\end{equation*}
				Then $(a_n)_{n \in \N}$ converges and $\sum_{n \in \N} b_n < +\infty$.	
			\end{lem}
			\begin{lem}[Lemma~3.80 \cite{ChouzenouxPesquetRoldan2023}]\label{lemmaPesquet}
				Let $(a_n)_{n \in \N}$ be a sequence in $\RP$, let $(\alpha_n)_{n \in \N}$ be a sequence in $\RP$ such that $\alpha_n \to 0$, and let $r \in~]0,1[$. Suppose that
				\begin{equation*}
					(\forall n \in \N)	\quad a_{n+1} \leq (1-r+\alpha_n)a_n.
				\end{equation*}
				Then, $(a_n)_{n \in \N}$ converges linearly to $0$.
		\end{lem}}
		
		\section{Inertial Nonlinear Forward-Backward}\label{sec:conv}
		In this section we consider the following problem.
		\begin{pro}\label{pro:main}
			Let $\HH$ be a real Hilbert space, let $\bm{S} \colon \HH \to \HH$ be a self adjoint strongly monotone linear bounded operator, let $\bm{A} \colon \HH \to 2^{\HH}$ be a maximally monotone operator, let $\bm{C} \colon \HH \to \HH$ be a $\beta$-cocoercive operator with respect to $\bm{S}$ for $\beta \in \RPP$. The problem is to 
			\begin{equation*}
				\text{find } \x \in \HH \text{ such that } 0 \in \bm{A}\x+\bm{C}\x,
			\end{equation*}
			under the assumption that its solution set is nonempty.
		\end{pro}
		The following assumption plays a central role in the convergence analysis of the inertial-relaxed nonlinear forward-backward algorithm studied in this article.
		\begin{asume}\label{asume:1}
			In the context of Problem~\ref{pro:main}, let $\gamma \in \RPP$ and let $(\gamma_n)_{n \in \N}$ be a sequence in $[\gamma,+\infty[$. Moreover, let $\bm{M}_n \colon \HH \to \HH$ be a singled valued operator such that $\gamma_n \bm{M}_n-\bm{S}$ is a $\zeta_n$-Lipschitz operator with respect to $\bm{S}$ for $\zeta_n \in [\cblue{\underline{\zeta},\zeta}]$ where \cblue{$\underline{\zeta} \in [0,\zeta]$} and $\zeta \in~]0,1[$.
		\end{asume}
		The following algorithm incorporates inertial and relaxation steps into NFB.
		\begin{algo}\label{algo:algoInertialRelaxed}
			In the context of Problem~\ref{pro:main} and Assumption~\ref{asume:1}, let $(\z_0,\z_{-1}) \in \HH^{2}$, let $(\alpha_n)_{n \in \N}$ be a nonnegative sequence such that $\alpha_n \to \alpha \in [0,1[$, let $(\lambda_n)_{n \in \N}$ be a sequence in $[\underline{\lambda} ,\overline{\lambda}[$ for $(\underline{\lambda},\overline{\lambda}) \in~]0,2[^2$, and consider the recurrence defined by 
			\begin{subequations}\label{eq:algoInertialRelaxed}
				\begin{align}
					&\y_n = \z_n + \alpha_n (\z_n - \z_{n-1}), \label{eq:algoInertialRelaxed_a}\\
					&\x_n = (\bm{M}_n+\bm{A})^{-1}(\bm{M}_n-\bm{C})\y_n, \label{eq:algoInertialRelaxed_b}\\
					&\w_{n+1} = \y_n - \gamma_n \bm{S}^{-1}(\bm{M}_n \y_n - \bm{M}_n\x_n), \label{eq:algoInertialRelaxed_c}\\
					&\z_{n+1} = \lambda_n \w_{n+1} + (1-\lambda_n)\y_n. \label{eq:algoInertialRelaxed_d}
				\end{align}
			\end{subequations}
		\end{algo}
		\begin{rem}
			Note that, Algorithm~\ref{algo:algoInertialRelaxed} is well defined in view of Assumption~\ref{asume:1} \cite[Proposition~3.1]{MorinBanertGiselsson2022}.
		\end{rem}
		The following proposition provides essential estimates for establishing the convergence of Algorithm~\ref{algo:algoInertialRelaxed}.
		\begin{prop}\label{prop:fejersequence}
			In the context of Problem~\ref{pro:main} and Assumption~\ref{asume:1}, let $(\z_0,\z_1)\in \HH^2$ and consider the sequence $(\z_n)_{n \in \N}$ generated by Algorithm~\ref{algo:algoInertialRelaxed}. \cblue{Suppose that $\bm{A}$ is $s$-monotone with respect to $\bm{S}$ for $s\geq 0$.} \cblue{Let $\varepsilon >0$,} \cblue{ let $\theta \in \RPP$ be such that $s\theta <1$, and,} for every $n \in \N$, define
			\begin{align}
				\cblue{r}&=\cblue{\dfrac{s\underline{\lambda}}{2}\min\{\left(1-\zeta^2-\varepsilon\right){\theta},2\gamma\},\label{eq:defr}}\\
				\nu_n &= \begin{cases}
					0, \text{ if } -(\gamma_n\bm{M}_n-\bm{S}) \text { is monotone,}\\
					2\zeta_n, \text{ otherwise.}
				\end{cases} 	\label{eq:defnun}\\
				\psi_n	&=\frac{2-\varepsilon+\nu_n-\cblue{s\theta(1-\zeta_n^2-\varepsilon)}}{1+\zeta_n^2+\nu_n},\label{eq:defpsin}\\
				\rho_n	&= \frac{\psi_n}{\lambda_n}-1,\label{eq:defrhon}\\
				\delta_n & = \cblue{(1-r)}(1-\alpha_{n})\rho_{n}-\alpha_{n+1}(1-\alpha_{n+1})\rho_{n+1}-\cblue{(1-r)}\alpha_{n+1}(1+\alpha_{n+1}).\label{eq:defdeltan}
			\end{align}
			Moreover, suppose that, for every $n \geq N \in  \N$, $\rho_n \geq 0$ and let $\varepsilon > 0$ be such that
			\begin{equation}\label{eq:cond_epsilon_prop}
				(\forall n \geq N) \quad	\varepsilon < 1 -\zeta_n^2.
			\end{equation}
			Therefore, for every \cblue{$\z \in \zer(\bm{A}+\bm{C})$ and every} $n \geq N$,
			\begin{align}
				&\|\z_{n+1}-\z\|^2_{\bm{S}}-\alpha_n\|\z_{n}-\z\|^2_{\bm{S}}+\cblue{(1-\alpha_{n})\rho_{n}}\|\z_{n+1}-\z_n\|^2_{\bm{S}}\nonumber\\
				&\leq \cblue{(1-r)}(\|\z_{n}-\z\|^2_{\bm{S}}-\alpha_n\|\z_{n-1}-\z\|^2_{\bm{S}}+\cblue{(1-\alpha_{n-1})\rho_{n-1}}\|\z_{n}-\z_{n-1}\|^2_{\bm{S}})\nonumber\\
				&\hspace{4cm}
				-\delta_{\cblue{n-1}}\|\z_{n}-\z_{\cblue{n-1}}\|^2_{\bm{S}} -\frac{\lambda_n\gamma_n}{\varepsilon}\left(2\beta\varepsilon-\gamma_n\right)\|\bm{C}\y_n-\bm{C}\z\|^2_{\bm{S}^{-1}}.\label{eq:desHn1}
			\end{align}
		\end{prop}
		\begin{proof}
			Fix $n \in \N$. Let $\z \in \zer(\bm{A}+\bm{C})$, thus, $-\bm{C}\z \in \bm{A}\z$. It follows from \cblue{\eqref{eq:algoInertialRelaxed_b}} that
			\begin{equation*}
				\bm{M}_n\y_n - \bm{M}_n\x_n- \bm{C}\y_n  \in \bm{A} \x_n. 
			\end{equation*}
			Hence, by the \cblue{$s$-strong monotonicity} of $\bm{A}$ we have
			\begin{equation}\label{eq:desmon}
				2\gamma_n\scal{\bm{M}_n\y_n - \bm{M}_n\x_n}{\x_n-\z} -  2\gamma_n\scal{ \bm{C}\y_n -\bm{C} \z }{\x_n-\z} \geq \cblue{2s \gamma_n \|\x_n-\z\|_{\bm{S}}^2}.
			\end{equation}
			Furthermore, by the \cblue{$\beta$}-cocoercivity of $\bm{C}$ with respect to $\bm{S}$ \cblue{and Young’s inequality}, we have
			\begin{align}\label{eq:descoco}
				-2\gamma_n\scal{\bm{C}\y_n - \bm{C} \z}{\x_n-\z}
				=& -2\gamma_n\scal{\bm{C}\y_n - \bm{C} \z}{\y_n-\z}-2\gamma_n\scal{\bm{C}\y_n - \bm{C} \z}{\x_n-\y_n}\nonumber\\ 
				\leq& -\frac{\gamma_n}{\varepsilon}\left(2\beta\varepsilon-\gamma_n\right)\|\bm{C}\y_n-\bm{C}\z\|^2_{\bm{S}^{-1}}+\varepsilon\|\y_n-\x_n\|^2_{\bm{S}}.
			\end{align}
			Moreover, it follows from \cblue{\eqref{eq:algoInertialRelaxed_c}} that
			\begin{align}\label{eq:desalgo1}
				\|\w_{n+1}-\z\|_{\bm{S}}^2 =&  \|\y_n - \gamma_n\bm{S}^{-1}(\bm{M}_n \y_n -\bm{M}_n\x_n)-\z\|_{\bm{S}}^2\nonumber\\
				=& \|\y_n -\z\|^2_{\bm{S}}- 2\gamma_n\scal{\y_n -\z}{\bm{M}_n \y_n -\bm{M}_n\x_n} +\gamma_n^2\|\bm{S}^{-1}(\bm{M}_n \y_n -\bm{M}_n\x_n)\|_{\bm{S}}^2.
			\end{align}	
			Note that, by summing the last two terms in \eqref{eq:desalgo1} and the first term in \eqref{eq:desmon}, by the Lipschitzian property of $\gamma_n \bm{M}_n-\bm{S}$, we obtain
			\begin{align}\label{eq:dessuma1}
				2\gamma_n&\scal{\bm{M}_n\y_n - \bm{M}_n\x_n}{\x_n-\z} - 2\gamma_n\scal{\y_n -\z}{\bm{M}_n \y_n -\bm{M}_n\x_n} +\gamma_n^2\|\bm{S}^{-1}(\bm{M}_n \y_n -\bm{M}_n\x_n)\|_{\bm{S}}^2\nonumber\\
				=&	2\gamma_n\scal{\bm{M}_n\y_n - \bm{M}_n\x_n}{\x_n-\y_n} +\gamma_n^2\|\bm{S}^{-1}(\bm{M}_n \y_n -\bm{M}_n\x_n)\|_{\bm{S}}^2\nonumber\\
				\cblue{\overset{\eqref{eq:scalSmS}}{=}}
				&	-\|\x_{n}-\y_{n}\|_{\bm{S}}^{2}+\|(\gamma_n\bm{M}_n-\bm{S})\y_n-(\gamma_n\bm{M}_n-\bm{S})\x_n)\|_{\bm{S}^{-1}}^{2} \nonumber\\
				\leq&	-(1-\zeta_n^2)\|\x_{n}-\y_{n}\|_{\bm{S}}^{2}.
			\end{align}
			By combining \eqref{eq:desmon}-\eqref{eq:dessuma1}, we deduce
			\begin{align}\label{eq:proof1}
				\|\w_{n+1}-\z\|_{\bm{S}}^2 \leq \|\y_n-\z\|_{\bm{S}}^2-&\left(1-\zeta_n^2-\varepsilon\right)\|\y_n-\x_n\|_{\bm{S}}^2-\frac{\gamma_n}{\varepsilon}\left(2\beta\varepsilon-\gamma_n\right)\|\bm{C}\y_n-\bm{C}\z\|^2_{\bm{S}^{-1}}\nonumber\\
				& -\cblue{2s\gamma_n\|\x_n-\z\|^2_{\bm{S}}}.
			\end{align}
			Now, it follows from  {\eqref{eq:algoInertialRelaxed_d}} and \eqref{eq:proof1} that
			\begin{align}\label{eq:proof2}
				\|\z_{n+1}-\z\|^2_{\bm{S}} &= \|\lambda_n \w_{n+1}+(1-\lambda_n)\y_n-\z\|^2_{\bm{S}}\nonumber\\
				&= \|(1-\lambda_n) (\y_n-\z)+\lambda_n(\w_{n+1}-\z)\|^2_{\bm{S}}\nonumber\\
				& {\overset{\eqref{eq:normaalpha}}{=}} (1-\lambda_n)\|\y_n-\z\|_{\bm{S}}^2+\lambda_n\|\w_{n+1}-\z\|^2_{\bm{S}}-\lambda_n(1-\lambda_n)\|\w_{n+1}-\y_n\|^2_{\bm{S}}\nonumber\\
				&\leq (1-\lambda_n)\|\y_n-\z\|_{\bm{S}}^2+\lambda_n\|\y_n-\z\|_{\bm{S}}^2-\lambda_n\left(1-\zeta_n^2-\varepsilon\right)\|\y_n-\x_n\|_{\bm{S}}^2\nonumber\\
				&\quad-\frac{\lambda_n\gamma_n}{\varepsilon}\left(2\beta\varepsilon-\gamma_n\right)\|\bm{C}\y_n-\bm{C}\z\|^2_{\bm{S}^{-1}} -\cblue{2s\lambda_n\gamma_n\|\x_n-\z\|_{\bm{S}}^2}- \lambda_n(1-\lambda_n)\|\w_{n+1}-\y_n\|^2_{\bm{S}}\nonumber\\
				&= \|\y_n-\z\|_{\bm{S}}^2-\lambda_n\left(1-\zeta_n^2-\varepsilon\right)\|\y_n-\x_n\|_{\bm{S}}^2-\frac{\lambda_n\gamma_n}{\varepsilon}\left(2\beta\varepsilon-\gamma_n\right)\|\bm{C}\y_n-\bm{C}\z\|^2_{\bm{S}^{-1}}\nonumber\\
				&\quad   -\cblue{2s\lambda_n\gamma_n\|\x_n-\z\|_{\bm{S}}^2}-\frac{1-\lambda_n}{\lambda_n}\|\z_{n+1}-\y_n\|^2_{\bm{S}}\nonumber\\
				&= \|\y_n-\z\|_{\bm{S}}^2-\lambda_n\cblue{(1-s\theta)}\left(1-\zeta_n^2-\varepsilon\right)\|\y_n-\x_n\|_{\bm{S}}^2-\frac{\lambda_n\gamma_n}{\varepsilon}\left(2\beta\varepsilon-\gamma_n\right)\|\bm{C}\y_n-\bm{C}\z\|^2_{\bm{S}^{-1}}\nonumber\\
				&\quad   -\lambda_n\cblue{s\theta}\left(1-\zeta_n^2-\varepsilon\right)\|\y_n-\x_n\|_{\bm{S}}^2-\cblue{2s\lambda_n\gamma_n\|\x_n-\z\|_{\bm{S}}^2}-\frac{1-\lambda_n}{\lambda_n}\|\z_{n+1}-\y_n\|^2_{\bm{S}}\nonumber\\
				&\leq \|\y_n-\z\|_{\bm{S}}^2-\lambda_n\cblue{(1-s\theta)}\left(1-\zeta_n^2-\varepsilon\right)\|\y_n-\x_n\|_{\bm{S}}^2-\frac{\lambda_n\gamma_n}{\varepsilon}\left(2\beta\varepsilon-\gamma_n\right)\|\bm{C}\y_n-\bm{C}\z\|^2_{\bm{S}^{-1}}\nonumber\\
				&\quad  -\cblue{\frac{s \lambda_n }{2}\min\{\cblue{\theta}\left(1-\zeta_n^2-\varepsilon\right),2\gamma_n\}\|\y_n-\z\|_{\bm{S}}^2}-\frac{1-\lambda_n}{\lambda_n}\|\z_{n+1}-\y_n\|^2_{\bm{S}}\nonumber\\
				&\leq \cblue{(1-r)}\|\y_n-\z\|_{\bm{S}}^2-\lambda_n\cblue{(1-s\theta)}\left(1-\zeta_n^2-\varepsilon\right)\|\y_n-\x_n\|_{\bm{S}}^2-\frac{\lambda_n\gamma_n}{\varepsilon}\left(2\beta\varepsilon-\gamma_n\right)\|\bm{C}\y_n-\bm{C}\z\|^2_{\bm{S}^{-1}}\nonumber\\
				&\quad  -\frac{1-\lambda_n}{\lambda_n}\|\z_{n+1}-\y_n\|^2_{\bm{S}},
			\end{align}
			\cblue{where the last inequality follows by noticing that $r\leq \dfrac{s \lambda_n }{2}\min\{\cblue{\theta}\left(1-\zeta_n^2-\varepsilon\right),2\gamma_n\}$.}
			\cblue{Note that if $-(\gamma_n\bm{M_n}-\bm{S})$ is monotone then 
				\begin{equation*}
					2\scal{(\gamma_n\bm{M_n}-\bm{S})\y_n-(\gamma_n\bm{M}_n-\bm{S})\x_n}{\y_n-\x_n}\leq 0.
				\end{equation*}
				Otherwise, by the Cauchy–Schwarz inequality and the $\zeta_n$-Lipschitz property of $\gamma_n \bm{M}_n - \bm{S}$
				\begin{align*}
					2\scal{(\gamma_n\bm{M_n}-\bm{S})\y_n-(\gamma_n\bm{M}_n-\bm{S})\x_n}{\y_n-\x_n}
					&\leq 2\zeta_n \|\y_n-\x_n\|^2_{\bm{S}}.
				\end{align*}	
				Hence, from \eqref{eq:defnun} we conclude that 
				\begin{equation}\label{eq:MSmono}
					2\scal{(\gamma_n\bm{M_n}-\bm{S})\y_n-(\gamma_n\bm{M}_n-\bm{S})\x_n}{\y_n-\x_n}\leq \nu_n\|\y_n-\x_n\|^2_{\bm{S}}.
			\end{equation}}
			Moreover, from \cblue{\eqref{eq:algoInertialRelaxed_d}, \eqref{eq:algoInertialRelaxed_c}, the $\zeta_n$-Lipschitz property of $\gamma_n \bm{M}_n - \bm{S}$, and \eqref{eq:MSmono}} we have
			\begin{align}\label{eq:proof3}
				\frac{1}{\lambda^2_n}\|\z_{n+1}-\y_n\|^2_{\bm{S}} & = \|\w_{n+1}-\y_n\|^2_{\bm{S}}\nonumber\\
				& = \|\gamma_n \bm{S}^{-1}(\bm{M_n}\y_n-\bm{M}_n\x_n)\|^2_{\bm{S}}\nonumber\\
				& = \| \bm{S}^{-1}((\gamma_n\bm{M_n}-\bm{S})\y_n-(\gamma_n\bm{M}_n-\bm{S})\x_n)+(\y_n-\x_n)\|^2_{\bm{S}}\nonumber\\
				& = \| (\gamma_n\bm{M_n}-\bm{S})\y_n-(\gamma_n\bm{M}_n-\bm{S})\x_n\|_{\bm{S}^{-1}}^2+\|\y_n-\x_n\|^2_{\bm{S}}\nonumber\\
				&\quad +2\scal{(\gamma_n\bm{M_n}-\bm{S})\y_n-(\gamma_n\bm{M}_n-\bm{S})\x_n}{\y_n-\x_n}\nonumber\\
				&\leq (1+\zeta_n^2+\nu_n)\|\y_n-\x_n\|^2_{\bm{S}}.
			\end{align}
			Hence, by combining \eqref{eq:proof2} with \eqref{eq:proof3}, by condition \eqref{eq:cond_epsilon_prop}, and by noticing that
			\begin{equation*}
				\frac{\cblue{(1-s\theta)}(1-\zeta_n^2-\varepsilon)}{\lambda_n(1+\zeta_n^2+\nu_n)}+\frac{(1-\lambda_n)}{\lambda_n}=\frac{\psi_n}{\lambda_n}-1=\rho_n,
			\end{equation*}
			we deduce
			\begin{align}\label{eq:proof4}
				\|\z_{n+1}-\z\|^2_{\bm{S}} &\leq \cblue{(1-r)} \|\y_n-\z\|_{\bm{S}}^2-\rho_n\|\z_{n+1}-\y_n\|_{\bm{S}}^2 -\frac{\lambda_n\gamma_n}{\varepsilon}\left(2\beta\varepsilon-\gamma_n\right)\|\bm{C}\y_n-\bm{C}\z\|^2_{\bm{S}^{-1}}.
			\end{align}
			Now, \cblue{\eqref{eq:algoInertialRelaxed_a}} yields
			\begin{align}\label{eq:proof5}
				\|\y_n-\z\|^2_{\bm{S}} &= \|\z_n+\alpha_n(\z_n-\z_{n-1})-\z\|^2_{\bm{S}}\nonumber\\
				&= \|(1+\alpha_n)(\z_n-\z)-\alpha_n(\z_{n-1}-\z)\|^2_{\bm{S}}\nonumber\\
				&\cblue{\overset{\eqref{eq:normaalpha}}{=}}(1+\alpha_n)\|\z_n-\z\|_{\bm{S}}^2-\alpha_n\|\z_{n-1}-\z\|_{\bm{S}}^2+\alpha_n(1+\alpha_n)\|\z_n-\z_{n-1}\|_{\bm{S}}^2.
			\end{align}
			Moreover,
			\begin{align}\label{eq:proof6}
				-\|\z_{n+1}-\y_n\|^2_{\bm{S}} &= -\|\z_{n+1}-\z_{n}-\alpha_n(\z_n-\z_{n-1})\|_{\bm{S}}^2\nonumber\\
				&= -\|(1-\alpha_{n})(\z_{n+1}-\z_{n})+\alpha_n(\z_{n+1}-2\z_{n}+\z_{n-1})\|_{\bm{S}}^2\nonumber\\
				& \cblue{\overset{\eqref{eq:normaalpha}}{=}}-(1-\alpha_n)\|\z_{n+1}-\z_n\|^2_{\bm{S}}+\alpha_n(1-\alpha_n)\|\z_n-\z_{n-1}\|^2_{\bm{S}}-\alpha_n\|\z_{n+1}-2\z_n+\z_{n-1}\|^2_{\bm{S}}\nonumber\\
				&\leq -(1-\alpha_n)\|\z_{n+1}-\z_n\|^2_{\bm{S}}+\alpha_n(1-\alpha_n)\|\z_n-\z_{n-1}\|^2_{\bm{S}}.
			\end{align}
			Therefore, in view of \eqref{eq:proof4}-\eqref{eq:proof6}, we deduce that
			\begin{align}
				&\|\z_{n+1}-\z\|^2_{\bm{S}}-\cblue{(1-r)}\alpha_n\|\z_{n}-\z\|^2_{\bm{S}}+(1-\alpha_{n})\rho_{n}\|\z_{n+1}-\z_n\|^2_{\bm{S}}\nonumber\\
				&\leq \cblue{(1-r)}(\|\z_{n}-\z\|^2_{\bm{S}}-\alpha_n\|\z_{n-1}-\z\|^2_{\bm{S}}+(1-\alpha_{n-1})\rho_{n-1}\|\z_{n}-\z_{n-1}\|^2_{\bm{S}})\nonumber\\
				&\quad 
				-\cblue{((1-r)(1-\alpha_{n-1})\rho_{n-1}-\alpha_n(1-\alpha_n)\rho_n-(1-r)\alpha_n(1+\alpha_n))\|\z_{n}-\z_{n-1}\|^2_{\bm{S}}}\nonumber\\
				&\quad -\frac{\lambda_n\gamma_n}{\varepsilon}\left(2\beta\varepsilon-\gamma_n\right)\|\bm{C}\y_n-\bm{C}\z\|^2_{\bm{S}^{-1}}\nonumber
			\end{align}
			and the \cblue{result} follows. 
		\end{proof}	
		\cblue{\begin{rem}\label{rem:rhopositivo}
				Let $(\rho_n)_{n \in \N}$ and $(\delta_n)_{n \in \N}$ be the sequences defined in \eqref{eq:defrhon}, \eqref{eq:defdeltan}, respectively. 
				We will derive the weak convergence of Algorithm~\ref{algo:algoInertialRelaxed} under the assumption $\liminf \delta_n = \delta >0$.  Note that this assumption ensures the hypothesis on $(\rho_n)_{n \in \N}$ stated in Proposition~\ref{prop:fejersequence}, that is, there exists $N \in \N$ such that, $\rho_n \geq 0$,  for $n \geq N$. Indeed, since $\alpha_n \to \alpha$, we have
				\begin{align*}
					(1-r)(1-\alpha) \liminf \rho_n = \liminf (1-r)(1-\alpha_n)\rho_n
					\geq \liminf \delta_n \geq \delta >0.
				\end{align*}
		\end{rem}}
		The following theorem proves the convergence of Algorithm~\ref{algo:algoInertialRelaxed} under nondecreasing inertial sequences.
		\begin{teo}\label{teo:WRL}
			In the context of Problem~\ref{pro:main} and Assumption~\ref{asume:1}, consider the sequence $(\z_{n})_{n\in\N}$ defined recursively by Algorithm~\ref{algo:algoInertialRelaxed} with initialization points $(\z_{0},\z_{-1})\in\H^{2}$.
			Let $(\nu_{n})_{n\in\N}$, $(\rho_{n})_{n\in\N}$, and $(\delta_{n})_{n\in\N}$ be the sequences defined in \eqref{eq:defnun}, and \eqref{eq:defrhon}, \eqref{eq:defdeltan}, respectively, \cblue{for $r=0$}. Suppose that, there exist $(\varepsilon,\delta) \in \RPP^2$ and $N \in \N$ such that 
			\begin{equation}\label{eq:stepsize}
				(\forall n \geq N) \quad 1-\zeta_n^2-\varepsilon >\delta \textnormal{ and }2\beta\varepsilon-\gamma_n \geq 0.
			\end{equation}
			Additionally, suppose that $(\alpha_n)_{n \geq N}$ is nondecreasing and that \begin{align}\label{eq:stepsize1}
				\liminf \delta_{n}=\delta>0.
			\end{align}
			Then, the following assertions hold:
			\begin{enumerate}
				\item\label{teo:WRL11} $\sum_{n \in \N} \|\z_{n+1}-\z_n\|^2_{\bm{S}}<+\infty$ and, for every $\z \in \zer(\bm{A}+\bm{C})$, $(\|\z_n-\z\|_{\bm{S}})_{n \in \N}$ is convergent.
				\item\label{teo:WRL12} $(\z_n)_{n \in \N}$ converges weakly to some solution to Problem~\ref{pro:main}.
			\end{enumerate}
			
		\end{teo}
		\begin{proof}	
			\begin{enumerate} \item Define, for each $n \in \N$, 
				\begin{equation}\label{eq:defHn}
					\bm{H}_n = \|\z_{n}-\z\|^2_{\bm{S}}-\alpha_{n-1}\|\z_{n-1}-\z\|^2_{\bm{S}}+\cblue{\alpha_{n-1}(1-\alpha_{n-1})}\|\z_{n}-\z_{n-1}\|^2_{\bm{S}}.
				\end{equation}
				Hence, it follows from Proposition~\ref{prop:fejersequence}, \cblue{Remark~\ref{rem:rhopositivo}}, and the nondecreasing property of $(\alpha_n)_{n \in \N}$,  that for every $n \geq N$, we have  
				\begin{equation}\label{eq:Hnonincreasing}
					\bm{H}_{n+1}\leq \bm{H}_n	-\delta_n\|\z_{n+1}-\z_{n}\|^2_{\bm{S}} -\frac{\lambda_n\gamma_n}{\varepsilon}\left(2\beta\varepsilon-\gamma_n\right)\|\bm{C}\y_n-\bm{C}\z\|^2_{\bm{S}^{-1}}.
				\end{equation}
				Therefore, in view of \eqref{eq:stepsize} and \eqref{eq:stepsize1}, we conclude that $(\bm{H}_{n})_{n \geq N}$ is nonincreasing. 
				Hence, \cblue{since $(\alpha_n)_{n \in \N}$ is nondecreasing, we deduce}
				\begin{equation}
					(\forall n \geq N) \quad \bm{H}_N \geq \bm{H}_n \geq \|\z_n-\z\|^2_{\bm{S}} -\alpha_n\|\z_{n-1}-\z\|^2_{\bm{S}} \geq \|\z_n-\z\|^2_{\bm{S}} -\alpha\|\z_{n-1}-\z\|^2_{\bm{S}}.
				\end{equation}
				Therefore, by an inductive procedure, we deduce
				\begin{align*}
					\|\z_n-\z\|^2_{\bm{S}} &\leq \alpha\|\z_{n-1}-\z\|^2_{\bm{S}} + 	\bm{H}_{N} \\
					&\leq  \alpha^{n-N}\|\z_{N}-\z\|^2_{\bm{S}} + \bm{H}_{N} \sum_{k=0}^{n-{N}-1} \alpha^k\\
					&= \alpha^{n-N}\|\z_{N}-\z\|^2_{\bm{S}} + \bm{H}_{N} \frac{1-\alpha^{n-{N}}}{1-\alpha}.\\
				\end{align*}	
				Hence, $(\z_n)_{n \in \N}$ and $(\bm{H}_{n})_{n \in \N}$ are bounded sequences. Moreover, since $(\bm{H}_{n})_{n \geq N}$ is nonincreasing, it is convergent. Now, by telescoping \eqref{eq:Hnonincreasing} we have,
				\begin{equation}
					(\forall n \geq N )	\quad \sum_{k = N}^n \delta_k \| \z_{k+1}-\z_k\|^2_{\bm{S}} \leq \bm{H}_{N}-\bm{H}_n.
				\end{equation}
				Then, by taking $n \to \infty$, we conclude that $\sum_{n \in \N}  \delta_n \| \z_{n+1}-\z_n\|^2_{\bm{S}}< + \infty$. Moreover, since $\liminf \delta_n = \delta > 0$, we conclude that \cblue{ $\sum_{n \in \N}  \| \z_{n+1}-\z_n\|^2_{\bm{S}}< + \infty$}. \cblue{By defining, for each $n \in \N$, $a_n = \|\z_n-\z\|^2_{\bm{S}}$ and $b_n = (\epsilon_{n+1}-\delta_n) \| \z_{n+1}-\z_n\|^2_{\bm{S}}+\epsilon_n\|\z_{n}-\z_{n-1}\|^2_{\bm{S}}$, since $(\epsilon_n)_{n \in \N}$ is bounded, it follows from Proposition~\ref{prop:fejersequence} and Lemma~\ref{lemmabot}} that, for every $\z \in \zer(\bm{A}+\bm{C})$, $(\|\z_n-\z\|_{\bm{S}
				})_{n \in \N}$ is convergent.
				\item \cblue{Note that, since $\gamma_n \geq \gamma$,} by \cblue{\eqref{eq:algoInertialRelaxed_c}}
				\begin{align*}
					\cblue{\gamma \|\bm{M}_{n} \y_{n}  - \bm{M}_{n} \x_{n}\|_{\bm{S}^{-1}}}	&\leq \|\cblue{\gamma_n}\bm{M}_{n} \y_{n}  - \cblue{\gamma_n} \bm{M}_{n} \x_{n}\|_{\bm{S}^{-1}} \\
					&= \|\w_{n+1}-\y_n\|_{\bm{S}}\\
					&= \frac{1}{\lambda_n}\|\z_{n+1}-\y_n\|_{\bm{S}}\\
					&\leq \frac{1}{\underline{\lambda}} (\|\z_{n+1}-\z_n\|_{\bm{S}} + \alpha_n \| \z_{n}-\z_{n-1}\|_{\bm{S}}) \to 0,
				\end{align*}
				thus, we conclude that
				\begin{equation}\label{eq:Myxto0}
					\|\bm{M}_{n} \y_{n}  - \bm{M}_{n} \x_{n}\|_{\cblue{\bm{S^{-1}}}} \to 0 \cblue{ \text{ and }  \|\gamma_n\bm{M}_{n} \y_{n}  -\gamma_n\bm{M}_{n} \x_{n}\|_{\bm{S}^{-1}} \to 0}.
				\end{equation}	
				On the other hand,
				\begin{align*}
					\| \y_{n}  -\x_{n}\|_{\bm{S}} &\leq  \| (\gamma_n\bm{M}_n-\bm{S})\y_{n}  -(\gamma_n\bm{M}_n-\bm{S})\x_{n}\|_{\bm{S}^{-1}} + \|\gamma_n\bm{M}_n\y_{n}  -\gamma_n\bm{M}_n\x_{n}\|_{\bm{S}^{-1}}\\
					&\leq \zeta_n \|\y_n-\x_n\|_{\bm{S}}+\|\gamma_n\bm{M}_n\y_{n}  -\gamma_n\bm{M}_n\x_{n}\|_{\bm{S}^{-1}}.
				\end{align*}
				Therefore, \cblue{since $\zeta_n < \zeta < 1$},
				\begin{equation}\label{eq:yxto0}
					(1-\zeta)	\| \y_{n}  -\x_{n}\|_{\bm{S}} \leq \|\gamma_n\bm{M}_n\y_{n}  -\gamma_n\bm{M}_n\x_{n}\|_{\bm{S}^{-1}}\to 0.
				\end{equation}
				Furthermore, by the $\beta$-cocoercivity of $\bm{C}$, we conclude
				\begin{equation}\label{eq:Cyxto0}
					\|\bm{C}_{n} \y_{n}  - \bm{C}_{n} \x_{n}\|_{S^{-1}} \leq \frac{1}{\beta}\| \y_{n}  -\x_{n}\|_{\bm{S}}\to 0.
				\end{equation}	
				Now, let $\z^*$ be a weak cluster point of $(\z_n)_{n \in \N}$ and let $(\z_{n_k})_{k \in \N}$ be a subsquence such that $\z_{n_k} \weak \z^*$. We have, $\y_{n_k} = \z_{n_{k}} + \alpha_{n_{k}} (\z_{n_k}-\z_{n_{k}+1} ) \weak \z^*$, thus, by \eqref{eq:yxto0}, $\x_{n_k} \cblue{\weak} \z^*$.
				Therefore, it follows from \cblue{\eqref{eq:algoInertialRelaxed_b}} that 
				\begin{equation}\label{eq:contA2}
					(\bm{M}_{n_k} \y_{n_k}  - \bm{M}_{n_k} \x_{n_k} ) - (\bm{C} \y_{n_k}-\bm {C}\x_{n_k}) \in (\bm{A} + \bm {C})\x_{n_k}.
				\end{equation}
				Since $\bm{C}$ is cocoercive, $\bm{A}+\bm{C}$ 
				is a maximally monotone operator \cite[Corollary~25.5]{bauschkebook2017}. Hence, from the weak-strong closedness of $\gra(\bm{A}+\bm{C})$  \cite[Proposition~20.38]{bauschkebook2017}, \eqref{eq:Myxto0}, \eqref {eq:Cyxto0}, and \eqref{eq:contA2}, we conclude that every weak cluster point of $\left(\z_{n}\right)_{n\in \N}$ belongs to $\zer (\bm{A}+\bm{C})$. The \cblue{result} follows by applying \cite[Lemma 2.47]{bauschkebook2017} in the space $(\HH,\scal{\cdot}{\cdot}_{\bm{S}})$.
			\end{enumerate}
		\end{proof}
		\begin{rem}\label{rem:MS}\begin{enumerate}
				\item By consider $\bm{M}_n = \frac{1}{\gamma_n}\bm{S}$, according to \eqref{eq:defnun}, $\zeta_n\equiv\nu_n\equiv0$ and Algorithm~\ref{algo:algoInertialRelaxed} reduces to the inertial and relaxed version of the classical forward-backward algorithm. If $\gamma_n \equiv \gamma \in \RPP$,  $\lambda_n \to \lambda \in ]0,2[$, \cblue{$r=0$}, and we set $\varepsilon = \gamma/2\beta$, \eqref{eq:stepsize1} is equivalent to
				\begin{equation*}
					\frac{(1-\alpha)^2}{\lambda}\left(2-\lambda-\frac{\gamma}{2\cblue{\beta}}\right)-\alpha(1+\alpha) >0
				\end{equation*}
				which corresponds with the condition in \cite[Corollary~3.12]{AttouchCabot19}. Furthermore, if $\lambda = 1$, it reduces to
				\begin{equation}\label{eq:desaFB}
					1-3\alpha-\frac{\gamma(1-\alpha)^2}{2\cblue{\beta}} >0,
				\end{equation}
				which is the condition proposed in \cite[Remark~3]{LorenzInFB2015}.
				\item The possibility of adding inertia in NFB in the case where $\bm{C} = 0$ and $\bm{S} = \id$ was discussed in \cite[Remark 4.4(ii)]{BuiCombettesWarped2020}. However, the conditions that guarantee convergence were not derived.
			\end{enumerate}
		\end{rem}
		\cblue{The nondecreasing property of $(\alpha_n)_{n \in \N}$ is crucial for proving that $(\bm{H}_n)_{n \in \N}$, defined in \eqref{eq:defHn}, is nonincreasing, which in turn allows us to establish the convergence of Algorithm~\ref{algo:algoInertialRelaxed}.}
		
		The next theorem ensures the convergence of Algorithm~\ref{algo:algoInertialRelaxed} for a decreasing sequence of inertial parameters $(\alpha_n)_{n \in \N}$ satisfying $\sum_{n \in \N} (\alpha_n - \alpha) < +\infty$\cblue{, for an adequate $\alpha \in ]0,1[$. Hence, the inertial step $\y_n$ can be interpreted as an $\ell^1$-type approximation of the inertial step generated with an admissible parameter $\alpha$. In this case, the term $\bm{K}_n$, defined in \eqref{eq:defKn}, is not Fejér monotone but quasi Fejér monotone, i.e., admits an $\ell^1$-type error (see \cite[Section~5.4]{bauschkebook2017}).} Although this result may appear weaker than Theorem~\ref{teo:WRL}, it allows a practical acceleration of convergence (see Section~\ref{se:NE}), as it permits larger inertial parameters during the initial iterations. To the best of our knowledge, this is the first approach that considers a decreasing sequence of inertial parameters.
		\begin{teo}\label{teo:WRL2}
			In the context of Problem~\ref{pro:main} and Assumption~\ref{asume:1}, consider the sequence $(\z_{n})_{n\in\N}$ defined recursively by Algorithm~\ref{algo:algoInertialRelaxed} with initialization points $(\z_{0},\z_{-1})\in\H^{2}$.
			Let $(\nu_{n})_{n\in\N}$, $(\rho_{n})_{n\in\N}$, and $(\delta_{n})_{n\in\N}$ be the sequences defined in \eqref{eq:defnun}, \eqref{eq:defrhon}, and \eqref{eq:defdeltan}, respectively, \cblue{for $r=0$}. Suppose that, there exist $(\varepsilon,\delta) \in \RPP^2$ and $N \in \N$ such that 
			\begin{equation}\label{eq:stepsize2}
				(\forall n \geq N) \quad 1-\zeta_n^2-\varepsilon >\delta \textnormal{ and }2\beta\varepsilon-\gamma_n \geq 0.
			\end{equation}
			Additionally, suppose that $(\alpha_n)_{n \geq N}$ is decreasing, that $\sum_{n \in\N} (\alpha_n-\alpha) < +\infty$,  and that \begin{align}\label{eq:stepsize12}
				\liminf \delta_{n}=\delta>0.
			\end{align}
			Then, the following assertions hold:
			\begin{enumerate}
				\item $\sum_{n \in \N} \|\z_{n+1}-\z_n\|^2_{\bm{S}}<+\infty$ and, for every $\z \in \zer(\bm{A}+\bm{C})$, $(\|\z_n-\z\|_{\bm{S}})_{n \in \N}$ is convergent.
				\item $(\z_n)_{n \in \N}$ converges weakly to some solution to Problem~\ref{pro:main}.
			\end{enumerate}
		\end{teo}
		\begin{proof}
			\begin{enumerate}
				\item For every $n \in \N$, define 
				\begin{equation}\label{eq:defKn}
					\bm{K}_{n} = \|\z_{n}-\z\|^2_{\bm{S}}-\alpha\|\z_{n-1}-\z\|_{\bm{S}}^2+\cblue{(1-\alpha_{n-1})\rho_{n-1}}\|\z_{n}-\z_{n-1}\|^2_{\bm{S}}.
				\end{equation} 
				First, suppose that $\alpha = 0$\cblue{, thus $\sum_{n \in \N} \alpha_n < + \infty$ and $\inf_{n \in \N} \bm{K}_n \geq 0$}. Then Proposition~\ref{prop:fejersequence}, \cblue{Remark~\ref{rem:rhopositivo}}, and \eqref{eq:stepsize2} yield
				\begin{align*}
					(\forall n \geq N) \quad	\bm{K}_{n+1} &\leq \bm{K}_n + \alpha_n \|\z_n-\z\|^2_{\bm{S}}-\alpha_n\|\z_{n-1}-\z\|^2_{\bm{S}}- \delta_{n-1} \|\z_{n}-\z_{n-1}\|^2_{\bm{S}}\\
					&\leq (1+\alpha_n)\bm{K}_n - \delta_{n-1} \|\z_{n}-\z_{n-1}\|^2_{\bm{S}}.
				\end{align*}
				Hence, in view of \cblue{Lemma~\ref{lemmaquasifej} applied to $a_n = \bm{K}_n$, $t_n = \alpha_n$, and $b_n = \delta_{n-1}\|\z_{n}-\z_{n-1}\|^2_{\bm{S}}$}, we conclude that $(\bm{K}_n)_{n \in \N}$ is convergent and $\sum_{n \in \N} \delta_n \|\z_{n+1}-\z_n\|^2_{\bm{S}}<+\infty$. Since $\liminf \delta_n = \delta >0$, we deduce that $\sum_{n \in \N}\|\z_{n+1}-\z_n\|^2_{\bm{S}} <+\infty$. Then, \cblue{since $((1-\alpha_n)\rho_n)_{n \in \N}$ is bounded and $\alpha = 0$}, by the convergence of $(\bm{K}_n)_{n \in \N}$ and \eqref{eq:defKn}, we have that $(\|\z_n-\z\|_{\bm{S}})_{n \in \N}$ is convergent for every $\z \in \zer (\bm{A}+\bm{C})$.
				
				Now, suppose that $\alpha > 0$,. \cblue{Since $\alpha <1$}, it follows from Proposition~\ref{prop:fejersequence} and \eqref{eq:stepsize2} that, for every $n \geq N$,
				\begin{align}\label{eq:desKn}
					\bm{K}_{n+1} &\leq \bm{K}_n + (\alpha_n-\alpha) \|\z_n-\z\|^2_{\bm{S}}-(\alpha_n-\alpha)\|\z_{n-1}-\z\|^2_{\bm{S}}- \delta_{n-1} \|\z_{n}-\z_{n-1}\|^2_{\bm{S}}\nonumber\\
					&\leq \bm{K}_n + \frac{(\alpha_n-\alpha)}{\alpha} \|\z_n-\z\|^2_{\bm{S}}-\alpha\frac{(\alpha_n-\alpha)}{\alpha}\|\z_{n-1}-\z\|^2_{\bm{S}}- \delta_{n-1} \|\z_{n}-\z_{n-1}\|^2_{\bm{S}}\nonumber\\
					&\leq \left(1+\frac{\alpha_n-\alpha}{\alpha}\right)\bm{K}_n  - \delta_{n-1} \|\z_{n}-\z_{n-1}\|^2_{\bm{S}}.
				\end{align}
				If there exist $n_0\geq N$ such that $\bm{K}_{n_0} \leq 0$, it follows from \eqref{eq:desKn} that $\bm{K}_n \leq 0$, for every $n \geq n_0$. Therefore, we have
				\begin{align*}
					(\forall n \geq n_0) \quad \bm{K}_n \leq 0 \Leftrightarrow (\forall n \geq n_0) \quad \|\z_{n+1}-\z\|^2_{\bm{S}}&\leq \alpha\|\z_n-\z\|_{\bm{S}}^2-\cblue{(1-\alpha_{n-1})\rho_{n-1}}\|\z_{n}-\z_{n-1}\|^2_{\bm{S}}\\
					&\leq\|\z_n-\z\|_{\bm{S}}^2-\cblue{(1-\alpha_{n-1})\rho_{n-1}}\|\z_{n}-\z_{n-1}\|^2_{\bm{S}}.
				\end{align*}
				Hence, in view of \cblue{Lemma~\ref{lemmaquasifej} applied to $a_n = \|\z_n-\z\|_{\bm{S}}^2$, $t_n \equiv 0$, and $b_n = (1-\alpha_{n-1})\rho_{n-1}\|\z_{n}-\z_{n-1}\|^2_{\bm{S}}$}, we conclude that, for every $\z \in \zer (\bm{A}+\bm{C})$, $(\|\z_n-\z\|_{\bm{S}})_{n \in \N}$ is convergent and $\sum_{n \in \N} (1-\alpha_{n})\rho_{n} \|\z_{n+1}-\z_n\|^2_{\bm{S}}<+\infty$. Since $\liminf (1-\alpha_{n})\rho_{n} \geq \delta >0$, we deduce that $\sum_{n \in \N}\|\z_{n+1}-\z_n\|^2_{\bm{S}}<+\infty$.
				
				On the other hand, if $\bm{K}_n \geq 0$, for every $n \geq N$,
				by \eqref{eq:desKn} and invoking again
				\cblue{Lemma~\ref{lemmaquasifej} applied  to $a_n = \bm{K}_n$, $t_n = (\alpha_n-\alpha)\alpha^{-1}$, and $b_n = \delta_{n-1}\|\z_n-\z_{n-1}\|^2_{\bm{S}}$}, we conclude that $(\bm{K}_n)_{n \in \N}$ is convergent and $\sum_{n \in \N} \delta_n \|\z_{n+1}-\z_n\|^2_{\bm{S}}<+\infty$. Since $\liminf \delta_n = \delta >0$, we deduce that $\sum_{n \in \N} \|\z_{n+1}-\z_n\|^2_{\bm{S}}<+\infty$. 
				By defining $b_n =  \frac{\alpha_n-\alpha}{\alpha} \bm{K}_n + \cblue{(1-\alpha_{n-1})\rho_{n-1}}\|\z_{n}-\z_{n-1}\|^2_{\bm{S}}$, it follows from \eqref{eq:defKn} and \eqref{eq:desKn} that, for every $n \geq N$,
				\begin{align}\label{eq:desalpham}
					\|\z_{n+1}-\z\|^2_{\bm{S}}&\leq \cblue{\|\z_n-\z\|^2_{\bm{S}}} + \alpha(\|\z_{n}-\z\|^2_{\bm{S}}- \|\z_{n-1}-\z\|_{\bm{S}}^2)+\cblue{((1-\alpha_{n-1})\rho_{n-1}-\delta_{n-1})}\|\z_{n}-\z_{n-1}\|^2_{\bm{S}}\nonumber\\
					&\quad +\frac{\alpha_n-\alpha}{\alpha} \bm{K}_n- \cblue{(1-\alpha_{n})\rho_{n}} \|\z_{n+1}-\z_n\|^2_{\bm{S}}\nonumber\\
					&\leq\cblue{\|\z_n-\z\|^2_{\bm{S}}}+ \alpha(\|\z_{n}-\z\|^2_{\bm{S}}- \|\z_{n-1}-\z\|_{\bm{S}}^2)+b_n.
				\end{align}
				Now, since  $\sum_{n \in\N} (\alpha_n-\alpha) < +\infty$, $\sum_{n \in\N} \|\z_{n+1}-\z_{n}\|^2_{\bm{S}} < +\infty$,  $(\bm{K}_n)_{n \in \N}$ \cblue{ is convergent (therefore bounded),  and $((1-\alpha_{n})\rho_{n})_{n \in \N}$ is bounded}, we deduce $\sum_{n \in \N} b_n< + \infty$. 
				Therefore, by \cblue{Lemma~\ref{lemmabot} applied to $a_n = \|\z_n-\z\|_{\bm{S}}^2$}, we conclude that $ (\|\z_{n}-\z\|^2_{\bm{S}})_{n \in \N}$ is convergent for every $\z \in \zer (\bm{A}+\bm{C})$.		
				\item By the convergence of $(\|\z_n-\z\|_{\bm{S}})_{n \in \N}$, for every $\z \in \zer (\bm{A}+\bm{C})$, and since $\sum_{n \in \N}\|\z_{n+1}-\z_n\|^2_{\bm{S}}<+\infty$, the proof of this assertion is analogous to the proof of Theorem~\ref{teo:WRL}\eqref{teo:WRL12}.
			\end{enumerate}
		\end{proof}
		
		\subsection{Linear Convergence}
		In this section, we prove the linear convergence of Algorithm~\ref{algo:algoInertialRelaxed} in the case where $\bm{A}$ is $s$ strongly monotone for $s > 0$. In this analysis, for simplicity, we suppose that the sequence $(\rho_n)_{n \in \N}$ converges to some $\rho \in \RPP$. The following proposition yields a fundamental estimate for proving the strong convergence of our algorithm.
		\begin{prop}\label{prop:Hboundedbelow}
			In the context of Problem~\ref{pro:main} and Assumption~\ref{asume:1}, consider the sequence $(\z_{n})_{n\in\N}$ defined recursively by Algorithm~\ref{algo:algoInertialRelaxed} with initialization points $(\z_{0},\z_{-1})\in\H^{2}$.
			Let $(\nu_{n})_{n\in\N}$, $(\rho_{n})_{n\in\N}$, and $(\delta_{n})_{n\in\N}$ be the sequences defined in \eqref{eq:defnun}, \eqref{eq:defrhon}, and \eqref{eq:defdeltan}, respectively. Suppose that $\rho_n \to \rho \in \RPP$, that $\alpha_n \to \alpha \in \RPP$, and that
			\begin{align}\label{eq:stepsizelimite0}
				\delta:=(1-r)(1-\alpha)\rho-\alpha(1-\alpha)\rho-\alpha(1+\alpha) >0.
			\end{align}
			Let $(\bm{H}_n)_{n \in \N}$ and $(\bm{K}_n)_{n \in \N}$ be the sequences defined in \eqref{eq:defHn} and \eqref{eq:defKn}, respectively. Then, there exists $N \in \N$ such that, for every $n \geq N$, the following statements hold:
			\begin{enumerate}
				\item \label{prop:Hboundedbelow1} There exists $\epsilon_1 \in \RPP$ such that
				\begin{equation*}
					\delta + r(1-\alpha)\rho+\alpha(1-\alpha)\rho+\alpha^2+\epsilon_1 \geq (1-\alpha_n)\rho_n-\alpha \geq \delta+ r(1-\alpha)\rho +\alpha(1-\alpha)\rho+\alpha^2-\epsilon_1>0.
				\end{equation*}     
				\item \label{prop:Hboundedbelow2}There exists $\epsilon_2 \in \RPP$ such that
				\begin{equation*}
					\rho_n(1-\alpha_n)(1-\alpha)-\alpha \geq  \delta + \alpha^2-\epsilon_2   >0.
				\end{equation*} 
				\item \label{prop:Hboundedbelow3} $\liminf \delta_n \geq \delta$.
				\item \label{prop:Hboundedbelow4}  	 $\bm{K}_{n+1}
				\geq \left( \dfrac{\delta+\alpha^2-\epsilon_2}{\delta+ r(1-\alpha)\rho+\alpha(1-\alpha)\rho+\alpha^2+\epsilon_1}\right)\|\z_{n+1}-\z\|^2_{\bm{S}}.$
				\item \label{prop:Hboundedbelow5} $
				\bm{H}_{n+1}
				\geq \left( \dfrac{\delta+\alpha^2-\epsilon_2}{\delta+ r(1-\alpha)\rho+\alpha(1-\alpha)\rho+\alpha^2+\epsilon_1}\right)\|\z_{n+1}-\z\|^2_{\bm{S}}.$
			\end{enumerate}
		\end{prop}
		\begin{proof}
			\begin{enumerate}
				\item  Since $\rho_n \to \rho$ and $\alpha_n \to \alpha$, the result follows by noticing that
				\begin{align*}
					(1-\alpha_n)\rho_n - \alpha \to (1-\alpha)\rho-\alpha = \delta+ r(1-\alpha)\rho +\alpha(1-\alpha)\rho+\alpha^2>0.
				\end{align*}   
				\item Analogously, we have 
				\begin{align*}
					\rho_n(1-\alpha_n)(1-\alpha) - \alpha \to \rho(1-\alpha)^2-\alpha &=(1-\alpha) \rho - \alpha(1-\alpha)\rho-\alpha\\
					&\geq (1-r)(1-\alpha) \rho - \alpha(1-\alpha)\rho-\alpha \\
					&= \delta + \alpha^2.
				\end{align*}
				The result follows.
				\item It follows directly from
				\begin{align*}
					\liminf \delta_n = (1-r)(1-\alpha)\rho-\alpha(1-\alpha)\rho-(1-r)\alpha(1+\alpha) \geq \delta.
				\end{align*}
				\item 	
				Fix $n \in \N$, let $\kappa_n \in \RPP$. By \eqref{eq:defKn} and Young's inequality, we have
				\begin{align*}
					\bm{K}_{n+1}  &=\|\z_{n+1}-\z\|^2_{\bm{S}}-\alpha\|\z_{n}-\z\|^2_{\bm{S}}+(1-\alpha_{n})\rho_{n}\|\z_{n+1}-\z_n\|^2_{\bm{S}}\\
					&\geq \|\z_{n+1}-\z\|^2_{\bm{S}}-\alpha\left(1+\frac{1}{\kappa_n}\right)\|\z_{n+1}-\z\|^2_{\bm{S}}-\alpha(1+\kappa_n)\|\z_{n+1}-\z_n\|^2_{\bm{S}}+(1-\alpha_{n})\rho_{n}\|\z_{n+1}-\z_n\|^2_{\bm{S}}\\
					&\geq \left(1-\alpha\left(1+\frac{1}{\kappa_n}\right)\right)\|\z_{n+1}-\z\|^2_{\bm{S}}+((1-\alpha_{n})\rho_{n}-\alpha(1+\kappa_n))\|\z_{n+1}-\z_n\|^2_{\bm{S}}.
				\end{align*}
				Choosing $\kappa_n$ such that $((1-\alpha_{n})\rho_{n}-\alpha(1+\kappa_n)) = 0$, we obtain $\kappa_n = \frac{(1-\alpha_n)\rho_n-\alpha}{\alpha}$. Note that, in view of \ref{prop:Hboundedbelow1}, for $n \geq N$, we have $\kappa_n \geq \frac{\delta+ r(1-\alpha)\rho+\alpha(1-\alpha)\rho+\alpha^2-\epsilon_1}{\alpha}>0$, thus, it is well defined. Hence, by \ref{prop:Hboundedbelow1} and \ref{prop:Hboundedbelow2}
				\begin{align}
					\bm{K}_{n+1}  
					&\geq  \left(1-\alpha\left(1+\frac{1}{\kappa_n}\right)\right)\|\z_{n+1}-\z\|^2_{\bm{S}}\nonumber\\
					&=\left(\frac{\rho_n(1-\alpha_n)(1-\alpha)-\alpha}{\rho_n(1-\alpha_n)-\alpha}\right)\|\z_{n+1}-\z\|^2_{\bm{S}}\label{eq:remarkliminf}\\
					& \geq \left( \frac{\delta+\alpha^2-\epsilon_2}{\delta+ r(1-\alpha)\rho+\alpha(1-\alpha)\rho+\alpha^2+\epsilon_1}\right)\|\z_{n+1}-\z\|^2_{\bm{S}}.\nonumber
				\end{align}
				\item By the nondecreasing property of $(\alpha_n)_{n \in \N}$, for every $n \in \N$, we have $\alpha_n \leq \alpha$. Hence, the result follows directly from \ref{prop:Hboundedbelow4} noticing that $\bm{H}_n \geq \bm{K}_{n}$.
			\end{enumerate}
		\end{proof}
		\begin{teo}\label{teo:WRL3}
			In the context of Problem~\ref{pro:main} and Assumption~\ref{asume:1}, consider the sequence $(\z_{n})_{n\in\N}$ defined recursively by Algorithm~\ref{algo:algoInertialRelaxed} with initialization points $(\z_{0},\z_{-1})\in\H^{2}$.
			Let $(\nu_{n})_{n\in\N}$, $(\rho_{n})_{n\in\N}$, and $(\delta_{n})_{n\in\N}$ be the sequences defined in \eqref{eq:defnun}, \eqref{eq:defrhon}, and \eqref{eq:defdeltan}, respectively. Suppose that, there exists $(\varepsilon,\delta) \in \RPP^2$ and $N \in \N$ such that 
			\begin{equation}\label{eq:stepsize3}
				(\forall n \geq N) \quad 1-\zeta_n^2-\varepsilon >\delta \textnormal{ and }2\beta\varepsilon-\gamma_n \geq 0.
			\end{equation}
			Additionally, suppose that $\rho_n \to \rho \in \RPP$, that $\alpha_n \to \alpha \in \RPP$,  that
			\begin{align}\label{eq:stepsizelimite}
				\delta:=(1-r)(1-\alpha)\rho-\alpha(1-\alpha)\rho-\alpha(1+\alpha) >0,
			\end{align}   
			and that one of the following assertions holds
			\begin{enumerate}
				\item\label{teo:WRL31} $(\alpha_n)_{n \in \N}$ is nondecreasing.
				\item\label{teo:WRL32} $(\alpha_n)_{n \in \N}$ is decreasing and $\sum_{n \in \N} (\alpha_n-\alpha)<+\infty$.
			\end{enumerate}
			Then, if $\bm{A}$ is $s$-strongly monotone with respect to $\bm{S}$, for some $s \in \RPP$, $(\z_n)_{n \in \N}$ converges strongly, at linear rate, to the unique solution of Problem~\ref{pro:main}.
		\end{teo}
		\begin{proof}
			Let us assume that $\bm{A}$ is $s$-strongly monotone with respect to $\bm{S}$ for $s \in \RPP$.
			\begin{enumerate}
				\item First, suppose that $(\alpha_n)_{n \in \N}$ is nondecreasing. It follows from Proposition~\ref{prop:fejersequence} that for $(\bm{H}_n)_{n \in \N}$ defined in \eqref{eq:defHn} and all $n \geq N$, we have
				\begin{equation}
					\bm{H}_{n+1} \leq (1-r)\bm{H}_{n},
				\end{equation}
				where $r>0$ is defined in \eqref{eq:defr}. Therefore, $(\bm{H}_{n})_{n \in \N}$ converges to $0$ at linear rate.
				If $\alpha = 0$, since $(\alpha_n)_{n \in \N}$ is nondecreasing, we have $\alpha_n \equiv 0$. Hence, $\bm{H}_n \geq \|\z_{n}-\z\|^2_{\bm{S}}$ and we conclude that $(\|\z_n-\z\|_{\bm{S}})_
				{n \in \N}$ converges to $0$ at linear rate. Otherwise, if $\alpha >0$,  $(\|\z_n-\z\|_{\bm{S}})_
				{n \in \N}$ converges to $0$ at linear rate in view of Proposition~\ref{prop:Hboundedbelow}\eqref{prop:Hboundedbelow5}.
				\item Suppose now that  $(\alpha_n)_{n \in \N}$ is decreasing and $\sum_{n \in \N} (\alpha_n-\alpha)<+\infty$. If $\alpha =0$, from Proposition~\ref{prop:fejersequence} it follows that, for $(\bm{K}_n)_{n \in \N}$ defined in \eqref{eq:defKn} and every $n \geq N$,
				\begin{align*}
					\bm{K}_{n+1} &\leq (1-r)\bm{K}_n + \alpha_n \|\z_n-\z\|^2_{\bm{S}}-(1-r)\alpha_n\|\z_{n-1}-\z\|^2_{\bm{S}}-\delta_{n-1}\|\z_{n}-\z_{n-1}\|^2_{\bm{S}}\\
					&\leq (1-r+\alpha_n)\bm{K}_n.
				\end{align*}
				By Lemma~\ref{lemmaPesquet} we conclude that $(\bm{K}_n)_{n \in \N}$ converges linearly to $0$ and the result follows by noticing that $\bm{K}_n \geq \|\z_n-\z\|^2_{\bm{S}}$. On the other hand, if $\alpha > 0$, by Proposition~\ref{prop:fejersequence}, we deduce
				\begin{align*}
					\bm{K}_{n+1} &\leq (1-r)\bm{K}_n + (\alpha_n-\alpha) \|\z_n-\z\|^2_{\bm{S}}-(1-r)(\alpha_n-\alpha)\|\z_{n-1}-\z\|^2_{\bm{S}}-\delta_{n-1}\|\z_{n}-\z_{n-1}\|^2_{\bm{S}}\nonumber\\
					&\leq (1-r)\bm{K}_n + \frac{(\alpha_n-\alpha)}{\alpha} \|\z_n-\z\|^2_{\bm{S}}-(\alpha_n-\alpha)\|\z_{n-1}-\z\|^2_{\bm{S}}\\
					&\leq \left(1-r+\frac{\alpha_n-\alpha}{\alpha}\right)\bm{K}_n.
				\end{align*}
				The result follows by combining  Proposition~\ref{prop:Hboundedbelow}\eqref{prop:Hboundedbelow3} and Lemma~\ref{lemmaPesquet}.
			\end{enumerate}
		\end{proof}
		\begin{rem}
			For simplicity, we assume condition \eqref{eq:stepsizelimite} on the parameters $(\alpha_n)_{n \in \N}$ and $(\rho_n)_{n \in \N}$ where $\rho_n \to \rho$. However, in view of \eqref{eq:remarkliminf}, this requirement can be relaxed to $\liminf \delta_n = \delta$ and \begin{equation*}
				\liminf \left(\frac{\rho_n(1-\alpha_n)(1-\alpha)-\alpha}{\rho_n(1-\alpha_n)-\alpha}\right) >0.
			\end{equation*}
		\end{rem}
		
		\subsection{Analyzing the inertial and relaxation parameters}
		
		In this section, we provide conditions in order to the sequences of inertial and relaxation parameters to fulfill the hypotheses required to ensure the convergence of Algorithm~\ref{algo:algoInertialRelaxed}. \cblue{First}, let us recall that the sequences $(\alpha_n)_{n \in \N}$, $(\lambda_n)_{n \in \N}$, and $(\gamma_n)_{n \in \N}$ are design parameters of the algorithm and can therefore be tuned by the user. In contrast, $(\rho_n)_{n \in \N}$ and $(\delta_n)_{n \in \N}$ depend on $(\zeta_n)_{n \in \N}$, which are the Lipschitz constants of the operators  $\gamma_n \bm{M}_n-\bm{S}$ \cblue{ and on $(\alpha_n)_{n \in \N}$ and $(\lambda_n)_{n \in \N}$.
			The key parameter conditions ensuring convergence of the algorithm, as stated in Proposition~\ref{prop:fejersequence}, Theorem~\ref{teo:WRL}, and Theorem~\ref{teo:WRL2}, are}
		\begin{equation}\label{eq:condpsi-delta}
			\cblue{(\exists N \in \N)(\forall n \geq N) \quad \rho_n \geq 0 \quad \textnormal{ and } \quad 	\liminf_{n \to +\infty}\delta_{n}=\delta>0.}
		\end{equation}
		Now we will determine suitable $(\alpha_n)_{n \in \N}$ and $(\lambda_n)_{n \in \N}$ to ensure that these requirements are satisfied. We split our analysis in two cases: when $\bm{A}$ is merely monotone $(s =0)$ and when it is strongly monotone $(s>0)$.
		
		\subsubsection{$\bm{A}$ is a monotone operator.}
		
		\begin{claim}\label{claim:first_r0} Condition $\rho_n \geq 0$ is satisfied,  if $\varepsilon < 1-\zeta_n^2$ and $\lambda_n \in \left[\underline{\lambda},\psi_n\right]$, where $\psi_n \in ]1,2[$ for all $n \in \N$.
		\end{claim}
		\begin{proof} According to \eqref{eq:defrhon}, $\rho_n = {\psi_n}/{\lambda_n}-1$.  Since $\lambda_n$ is a positive number, $\rho_n \geq 0$ is fulfilled, if and only if, $\psi_n \geq \lambda_n$. Note that \eqref{eq:cond_epsilon_prop} imposes $\varepsilon < 1-\zeta_n^2$, thus, since $\nu_n \geq 0$ (see \eqref{eq:defnun}) and by \eqref{eq:defpsin}, we deduce
			\begin{equation}\label{eq:psimen1}
				\varepsilon < 1 - \zeta_n^2 \Longleftrightarrow \psi_n	= \frac{2-\varepsilon+\nu_n}{1+\zeta_n^2+\nu_n}>1  .
			\end{equation}
			According to Algorithm~\ref{algo:algoInertialRelaxed}, there exists $\underline{\lambda} \in~]0,2[$ such that $\lambda_n \geq  \underline{\lambda}$. Hence, 
			\begin{equation}\label{eq:intervalolamn}
				\lambda_n \in \left[ \underline{\lambda}, \psi_n \right] = \left[ \underline{\lambda}, \dfrac{2-\varepsilon+\nu_n}{1+\zeta_n^2+\nu_n} \right].
			\end{equation}
			In view of \eqref{eq:defnun} $\nu_n=0$, or $\nu_n = 2\zeta_n$. When $\nu_n =0$, as $\zeta_n \in [0,1[$, we have
			\begin{equation}\label{eq:psimay21}
				2 - \frac{\varepsilon}{2}> \frac{2}{1+\zeta_n^2} -\frac{\varepsilon}{1+\zeta_n^2} =  \psi_n.
			\end{equation}
			On the other hand, if $\nu_n = 2\zeta_n$, we deduce  
			\begin{equation}\label{eq:psimay22}	2 - \frac{\varepsilon}{4} > \frac{2(1+\zeta_n)}{(1+\zeta_n)^2} -\frac{\varepsilon}{(1+\zeta_n)^2} =   \psi_n.
			\end{equation}
			Thus, \eqref{eq:psimen1}, \eqref{eq:psimay21}, and \eqref{eq:psimay21} gives $\psi_n \in~ ]1,2[$. 
		\end{proof}
		\begin{rem} Notice that the case $\nu_n = 0$ includes,  the choice $\bm{M}_n = \frac{1}{\gamma_n}\bm{S}$ presented in Remark \ref{rem:MS}. Additional examples where $\nu_n = 0$ will be  presented in Section~\ref{sec:partcases}.
		\end{rem}
		We study the following condition in a simplified framework. In the context of Problem~\ref{pro:main}, Assumption~\ref{asume:1}, and Algorithm~\ref{algo:algoInertialRelaxed}, suppose that $\lambda_n \to \lambda \in ]0,2[$, $\zeta_n \to \zeta \in [0,1[$, and $\nu_n \to \nu \in \{0,2\zeta\}$. Then, in this case, $\rho_n \to \psi/\lambda -1$ where 
		\begin{equation}\label{eq:defpsi}
			\psi = \frac{2-\varepsilon+\nu}{1+\zeta^2+\nu}.
		\end{equation}
		\begin{claim}\label{claim:deltan} Assume that $1-\zeta^2-\varepsilon>0$  and $2\beta\varepsilon \geq \gamma$. Condition $\liminf \delta_{n} >0$ is satisfied in each of the following cases:
			\begin{enumerate}
				\item\label{claim:deltan1} For $\alpha \in [0,1[$, we have
				\begin{equation}\label{eq:lambda}
					\lambda \in \left]0, \psi \cdot\phi(\alpha) \right[, \quad \text{where} \quad \phi(\alpha) =  \dfrac{(1-\alpha)^2}{2\alpha^2-\alpha+1}.
				\end{equation}
				In particular, $\phi (\alpha) \in [0,1[$ and $\lambda \in~]0,2[$.
				\item\label{claim:deltan2} For $\lambda \in~]0,\psi[$, we have
				\begin{equation}\label{eq:alphaintervalo}
					\alpha \in [0,\overline{\alpha}[, \quad \textnormal{ where } \quad \overline{\alpha} = \frac{2\left(\frac{\psi}{\lambda}-1\right)}{\left(\frac{2\psi}{\lambda}-1\right) + \sqrt{\frac{8\psi}{\lambda}-7} }. 
				\end{equation}
			\end{enumerate} 
		\end{claim}
		\begin{proof}     \begin{enumerate}
				\item  In our setting, condition $\liminf \delta_{n} >0$ reduces to
				\begin{align}
					\liminf \delta_{n} = 	(1-\alpha)^2\rho-\alpha(1+\alpha) > 0 &\Leftrightarrow (1-\alpha)^2\frac{\psi}{\lambda}-1+\alpha-2\alpha^2>0\label{eq:desalphalam0}\\
					&\Leftrightarrow 
					\frac{(1-\alpha)^2}{2\alpha^2-\alpha+1} \cdot \psi>\lambda \label{eq:desalphalam1} \\
					&\Leftrightarrow 
					\phi(\alpha) \cdot \psi>\lambda. \nonumber
				\end{align}
				It is easy to check, that $\phi(0) = 1$, $\lim_{\alpha \nearrow 1} \phi(\alpha) = 0$, and $\phi$ is a decreasing function over $[0,1[$, as seen in Figure \ref{fig:alpha_1}. Using the same arguments as in Claim \ref{claim:first_r0}, $\psi \in ]0,2[$ and hence, $\lambda \in ]0,2[$. 
				\item On the other hand, consider $\lambda \in~]0,\psi[$, and let us define $a = \left(\frac{\psi}{\lambda} - 2\right)$. Then, \eqref{eq:desalphalam0} is equivalent to
				\begin{equation}\label{eq:alphacuadratic}
					a\alpha^2 - (2a+3)\alpha + a+1> 0.
				\end{equation}	
				Since $\lambda \in   ]0,\psi[$, we have that $a>-1$. Notice that in the case $a=0$ \eqref{eq:alphacuadratic} reduces to $\alpha < \frac{1}{3}$. Now, if $a \neq 0$, the solutions of the quadratic equation corresponding to \eqref{eq:alphacuadratic}, are given by
				
				\begin{equation}\label{eq:solalpha_1}
					\alpha_1 = \frac{2a+3 - \sqrt{8a+9}}{2a}=\frac{2(a+1)}{(2a+3 + \sqrt{8a+9})}
				\end{equation}
				and
				\begin{equation}\label{eq:solalpha_2}
					\alpha_2 = \frac{2a+3 + \sqrt{8a+9}}{2a}=\frac{2(a+1)}{(2a+3 - \sqrt{8a+9})}.
				\end{equation}
				If $-1<a < 0$, then $\alpha_1 > 0 $ and, since $2a+3+\sqrt{8a+9}>0$, $\alpha_2<0$.  Therefore, we conclude that  \eqref{eq:desalphalam0} holds is $\alpha \in \left[0, \alpha_1\right[$. Conversely, if $a > 0$, we have $\frac{1}{3}< \alpha_1 < 1$ and $\alpha_2>1$, then \eqref{eq:desalphalam0} holds if $\alpha \in \left[0, \alpha_1\right[$ and the conclusion follows noticing that $\alpha_1=\overline{\alpha}$. 
			\end{enumerate}
		\end{proof}
		\begin{rem}	\cblue{The following particular cases are deduced from Claim~\ref{claim:deltan}.
				\begin{enumerate}
					\item If $\lambda \to \psi$, then $\overline{\alpha} \to 0$.
					\item If $\lambda = \psi/2$, then  $\overline{\alpha} = 1/3$, which is usually a bound for inertial parameters (see, for instance, \cite{Alvarez2001,BotrelaxFBF2023,LorenzInFB2015,MaulenFierroPeypouquet2023}).
					\item If $\lambda=1$, from \eqref{eq:alphaintervalo} and \eqref{eq:defpsi} we obtain
					\begin{align*}
						\alpha \in  \left[0,\frac{2\left(\frac{2-\varepsilon+\nu(\zeta)}{1+\zeta^2+\nu(\zeta)}-1\right)}{\left(2\frac{2-\varepsilon+\nu(\zeta)}{1+\zeta^2+\nu(\zeta)}-1\right) + \sqrt{8\frac{2-\varepsilon+\nu(\zeta)}{1+\zeta^2+\nu(\zeta)}-7} }\right[.
					\end{align*}
			\end{enumerate}}
		\end{rem}
		\begin{center}
			\begin{figure}
				\begin{tikzpicture}
					\begin{axis}[
						xmin=-0.0, xmax=1.0,
						ymin=0., ymax=1,
						samples=400,
						xlabel={$\alpha$},
						axis lines=middle,
						grid=both,
						width=12cm,
						height=8cm,
						grid style={line width=.1pt, draw=gray!10},
						major grid style={line width=.2pt, draw=gray!50},
						minor tick num=5,
						enlargelimits={abs=0.1},
						axis line style={latex-latex},
						ticklabel style={font=\tiny, fill=white},
						xlabel style={at={(ticklabel* cs:1)}, anchor=south west},
						ylabel style={at={(ticklabel* cs:1)}, anchor=south west}
						]
						\addplot[red, thick, domain=-0.00:1] {(1-x)^2 / (2*x^2 - x +1)};
						\addlegendentry{$\phi(\alpha)$}
						\addlegendentry{$\frac{1}{3}$}
					\end{axis}
				\end{tikzpicture}\caption{Plot of the function $\alpha \mapsto \phi(\alpha) =  \dfrac{(1-\alpha)^2}{2\alpha^2-\alpha+1}
					$.}\label{fig:alpha_1}
			\end{figure}
		\end{center}
		\subsubsection{$\bm{A}$ is $s$-strongly monotone.} \hfill 
		\bigskip \par
		Let us consider the case $\lambda_n \to \lambda \in~ ]0,2[$, $\zeta_n \to \zeta \in [0,1[$, and $\nu_n \to \nu \in \{0,2\zeta\}$ and assume that $1-\zeta^2-\varepsilon>0$ and $2\beta\varepsilon \geq \gamma$. Under this setting, the condition guaranteeing convergence in Theorem~\ref{teo:WRL3} is:
		\begin{align}\label{eq:deltar}
			\delta:=(1-r)(1-\alpha)\rho-\alpha(1-\alpha)\rho-\alpha(1+\alpha) >0.
		\end{align}
		The following claim shows that, for $\theta$ sufficiently small, the conditions in \eqref{eq:intervalolamn} and \eqref{eq:desalphalam0} studied for the case $r=0$ also guarantee that $\rho>0$ and that \eqref{eq:deltar} holds when $r>0$.
		\begin{claim}
			\begin{enumerate}
				\item Suppose that (see \eqref{eq:intervalolamn})
				$$\lambda \in \left[ \underline{\lambda}, \dfrac{2-\varepsilon+\nu}{1+\zeta^2+\nu} \right[.$$
				Then there exists $\theta \in \RPP$ sufficiently small such that $\rho >0$.
				\item Suppose that (see \eqref{eq:desalphalam0})
				\begin{equation*}
					(1-\alpha)^2\rho-\alpha(1+\alpha) > 0.
				\end{equation*}
				Then, there exists $\theta \in \RPP$ sufficiently small such that $\delta > 0$.
			\end{enumerate}
		\end{claim}   
		\begin{proof}
			\begin{enumerate}
				\item   Consider $\lambda \in \left[ \underline{\lambda}, \dfrac{2-\varepsilon+\nu}{1+\zeta^2+\nu} \right[$. Let $$\delta_1 := \frac{2-\varepsilon+\nu}{\lambda(1+\zeta+\nu)} -1  > 0.$$
				By picking $\theta > 0$ such that $  \theta \leq \delta_1\lambda (1+\zeta^2+\nu)/(2s(1-\zeta^2-\varepsilon))$, then, by \eqref{eq:defrhon},
				\begin{align*}
					\rho &= 	\frac{2-\varepsilon+\nu-s\theta(1-\zeta^2-\varepsilon)}{\lambda(1+\zeta^2+\nu)}-1\\
					&= \delta_1 - \frac{s\theta(1-\zeta^2-\varepsilon)}{\lambda(1+\zeta^2+\nu)}\\
					&\geq \delta_1 - \frac{\delta_1}{2} > 0.
				\end{align*}
				\item Let $
				\delta_2:=(1-\alpha)^2\rho-\alpha(1+\alpha)  > 0$.	Hence, choosing $\theta$ such that $r \leq \delta_2/(2\rho(1-\alpha))$ (see \eqref{eq:defr}), we have
				\begin{align*}
					(1-r)(1-\alpha)\rho-\alpha(1-\alpha)\rho -\alpha(1+\alpha)
					&= (1-\alpha)^2\rho-\alpha(1+\alpha)-r\rho(1-\alpha)\\
					&\geq \delta_2 -\frac{\delta_2}{2}=\frac{\delta_2}{2},
				\end{align*}
				which shows that \eqref{eq:deltar} holds.
			\end{enumerate}
		\end{proof}        
		The next Claim provides conditions guaranteeing \eqref{eq:deltar} in the cases where $\lambda=1$ and $\alpha=0$.
		\begin{claim}
			\begin{enumerate}
				\item Suppose that $\lambda = 1$, then $\rho >0$ and \eqref{eq:deltar} holds if
				\begin{align*}
					\alpha \in \left[0, \frac{(2-r)\rho+1 - \sqrt{((2-r)\rho+1)^2-4(1-r)\rho(\rho-1)}}{2(\rho-1)}\right[.
				\end{align*}
				\item Suppose that $\alpha =0$, then \eqref{eq:deltar} holds if
				\begin{align*}
					\frac{2-\varepsilon+\nu-s\theta(1-\zeta^2-\varepsilon)}{1+\zeta^2+\nu}> \lambda,
				\end{align*}
				which is always true if $\lambda<1$.
			\end{enumerate}
		\end{claim}
		\begin{proof}
			\begin{enumerate}
				\item Suppose that $\lambda = 1$, we have
				\begin{align*}
					\rho = \frac{2-\varepsilon+\nu -s\theta (1-\zeta^2-\varepsilon)}{1+\zeta^2+\nu}-1 = \frac{(1-s\theta)(1-\zeta^2-\varepsilon)}{1+\zeta^2+\nu}.
				\end{align*}
				Hence, since $\theta \in~]0,1/r[$ and $1-\zeta^2-\varepsilon > 0$, we have $\rho > 0$. Moreover, it is easy to verify that $\rho-1 < 0$ and $1-r > 0$, thus,
				\begin{align*}
					&(1-r)(1-\alpha)\rho - \alpha (1-\alpha)\rho -\alpha (1+\alpha) = (1-r)\rho-((2-r)\rho+1)\alpha +(\rho-1)\alpha^2> 0\\
					&\Leftrightarrow \alpha \in \left[0, \frac{(2-r)\rho+1 - \sqrt{((2-r)\rho+1)^2-4(1-r)\rho(\rho-1)}}{2(\rho-1)}\right[.
				\end{align*}
				The result follows.
				\item When $\alpha =0$, we have $\delta = (1-r)\rho$. Hence, since $r<1$, \eqref{eq:deltar} reduces to
				\begin{align*}
					\rho & = \frac{2-\varepsilon+\nu-s\theta(1-\zeta^2-\varepsilon)}{\lambda(1+\zeta^2+\nu)}-1> 0 \Leftrightarrow \frac{2-\varepsilon+\nu-s\theta(1-\zeta^2-\varepsilon)}{1+\zeta^2+\nu}> \lambda.
				\end{align*}
			\end{enumerate}
		\end{proof}
		
		\subsection{Relaxing the hypothesis on \texorpdfstring{$\gamma_n \bm{M}_n-\bm{S}$}{x}}
		In this section we discuss how to consider the assumptions made in \cite{Giselsson2021NFBS} for the convergence of Algorithm~\ref{algo:algoInertialRelaxed}. Consider the following assumption.
		\begin{asume}{\cite[Assumptions~3.1~\&~3.2]{Giselsson2021NFBS}}\label{asume:1.2}
			Let $\bm{P}$ be a self adjoint strongly monotone linear operator, let 	$\bm{A}\colon \HH \to 2^\HH$ be a maximally monotone operator, let $\bm{C} \colon \HH \to \HH$ be $\beta$-cocoercive with respect to $\bm{P}$ for $\beta \in ]0,1/4[$, and let, for every $n \in \N$, $\bm{M}_n \colon \HH \to \HH$ be a 1-strongly monotone with respect to $\bm{P}$ and $\zeta$-Lipschitz operator for $\zeta >0$.
		\end{asume}
		Weak convergence of \eqref{eq:NFBgis} under Assumption~\ref{asume:1.2} was established in \cite[Theorem 5.1]{Giselsson2021NFBS} for $(\theta_n)_{n \in \N}$ in $]0,2[$ and $(\mu_n)_{n \in \N}$ satisfying
		\begin{equation}\label{eq:mun}
			(\forall n \in \N)\quad	\mu_n \leq \frac{\scal{\bm{M}_n\z_n-\bm{M}_n\x_n}{\z_n-\x_n}-\frac{1}{4\beta}\|\z_n-\x_n\|^2_{\bm{P}}}{\|\bm{M}_n\z_n-\bm{M}_n\x_n\|_{\bm{S}^{-1}}^2}.
		\end{equation} 
			In view of \cite[Proposition~2.1]{MorinBanertGiselsson2022}, with an appropriate choice of $\bm{P}$, Assumption~\ref{asume:1.2} holds in the context of Problem~\ref{pro:main} and Assumption~\ref{asume:1}. Therefore, Assumption~\ref{asume:1} can be seen as a particular case of  Assumption~\ref{asume:1.2}.
			
			For simplicity, our previous convergence analysis was restricted to the setting of Assumption~\ref{asume:1}. Nonetheless, Assumption~\ref{asume:1} allowed us to establish the convergence of the proposed algorithm under simpler step-size conditions, thereby recovering and extending previous results in the literature in a more direct manner (see Section~\ref{sec:partcases}). In particular, all the examples considered in that section satisfy Assumption~\ref{asume:1}.
			
			Even so, our analysis can be adapted to incorporate Assumption~\ref{asume:1.2}. Indeed, consider the sequence $(\z_n)_{n \in \N}$ generated by Algorithm~\ref{algo:algoInertialRelaxed} under Assumption~\ref{asume:1.2}, and let $(\mu_n)_{n \in \N}$ be such that
			\begin{equation}\label{eq:mun2}
				(\forall n \in \N)\quad \mu_n \leq \frac{\scal{\bm{M}_n\y_n-\bm{M}_n\x_n}{\y_n-\x_n}-\frac{1}{4\beta}\|\y_n-\x_n\|^2_{\bm{P}}}{\|\bm{M}_n\y_n-\bm{M}_n\x_n\|_{\bm{S}^{-1}}^2}.
			\end{equation}
			Note that, when $\gamma_n \equiv \theta_n\mu_n$, Algorithm~\ref{algo:algoInertialRelaxed} yields a relaxed and inertial version of the recurrence in~\eqref{eq:NFBgis}.
			
			We now proceed to modify the proof of Proposition~\ref{prop:fejersequence} accordingly. In particular, equations~\eqref{eq:desmon} and~\eqref{eq:dessuma1} remain valid.
			On the other hand, by the cocoercivity of $\bm{C}$ with respect to $\bm{P}$, \eqref{eq:desalgo1} is modified to
			\begin{align}\label{eq:descoco.2}
				-2\gamma_n\scal{\bm{C}\y_n - \bm{C} \z}{\x_n-\z}
				=& -2\gamma_n\scal{\bm{C}\y_n - \bm{C} \z}{\y_n-\z}-2\gamma_n\scal{\bm{C}\y_n - \bm{C} \z}{\x_n-\y_n}\nonumber\\ 
				\leq& -\frac{\gamma_n}{\varepsilon}\left(2\beta\varepsilon-\gamma_n\right)\|\bm{C}\y_n-\bm{C}\z\|^2_{\bm{P}^{-1}}+\varepsilon\|\y_n-\x_n\|^2_{\bm{P}}.
			\end{align}
			Moreover, in view of \eqref{eq:mun2}, \eqref{eq:dessuma1} is modified to
			\begin{align}\label{eq:dessuma1.2}
				2\gamma_n&\scal{\bm{M}_n\y_n - \bm{M}_n\x_n}{\x_n-\z} - 2\gamma_n\scal{\y_n -\z}{\bm{M}_n \y_n -\bm{M}_n\x_n} +\gamma_n^2\|\bm{S}^{-1}(\bm{M}_n \y_n -\bm{M}_n\x_n)\|_{\bm{S}}^2\nonumber\\
				&=	2\gamma_n\scal{\bm{M}_n\y_n - \bm{M}_n\x_n}{\x_n-\y_n} +\gamma_n^2\|\bm{S}^{-1}(\bm{M}_n \y_n -\bm{M}_n\x_n)\|_{\bm{S}}^2\nonumber\\
				&\leq \left(\frac{\gamma_n^2}{\mu_n}-2\gamma_n\right)\scal{\bm{M}_n\y_n - \bm{M}_n\x_n}{\y_n-\x_n} -\frac{\gamma_n^2}
				{4\beta\mu_n}\|\y_n-\x_n\|^2_{\bm{P}}.
			\end{align}
			Then, by combining \eqref{eq:desmon}, \eqref{eq:dessuma1}, \eqref{eq:descoco.2}, and \eqref{eq:dessuma1.2},  we deduce
			\begin{align}\label{eq:proof1.2}
				\|\w_{n+1}-\z\|_{\bm{S}}^2 \leq \|\y_n-\z\|_{\bm{S}}^2-\left(\frac{\gamma_n^2}
				{4\beta\mu_n}-\varepsilon\right)\|\y_n-\x_n\|_{\bm{P}}^2-\frac{\gamma_n}{\varepsilon}\left(2\beta\varepsilon-\gamma_n\right)\|\bm{C}\y_n-\bm{C}\z\|^2_{\bm{P}^{-1}}\nonumber\\
				+ \left(\frac{\gamma_n^2}{\mu_n}-2\gamma_n\right)\scal{\bm{M}_n\y_n - \bm{M}_n\x_n}{\y_n-\x_n}.
			\end{align}
			Furthermore, if 
			{
				\begin{equation}\label{eq:gammamu}
					\frac{\gamma_n^2}{\mu_n}-2\gamma_n \leq 0,
				\end{equation}
			}
			it follows from the $1$-strongly monotonicity of $\bm{M}_n$ that
			\begin{align}\label{eq:proofalt1}
				\|\w_{n+1}-\z\|_{\bm{S}}^2 \leq \|\y_n-\z\|_{\bm{S}}^2-&\left(\frac{\gamma_n^2}
				{4\beta\mu_n}-\varepsilon-\frac{\gamma_n^2}{\mu_n}+2\gamma_n\right)\|\y_n-\x_n\|_{\bm{P}}^2\nonumber\\
				&-\frac{\gamma_n}{\varepsilon}\left(2\beta\varepsilon-\gamma_n\right)\|\bm{C}\y_n-\bm{C}\z\|^2_{\bm{P}^{-1}}.
			\end{align}
			Hence, in view of \eqref{eq:proofalt1}, if
			\begin{equation}
				(\forall n \in \N) \quad \left(\frac{\gamma_n^2}
				{4\beta\mu_n}-\varepsilon-\frac{\gamma_n^2}{\mu_n}+2\gamma_n\right) \geq \epsilon  \quad \textnormal{ and } \quad 2\beta\varepsilon-\gamma_n \geq 0,
			\end{equation}
			for some $\epsilon> 0$, it is possible to derive a similar result to the one presented in Proposition~\ref{prop:fejersequence} and consequently establish the convergence of the Algorithm~\ref{algo:algoInertialRelaxed} under Assumption~\ref{asume:1.2} for adequate relaxation and inertial parameters. Note that \eqref{eq:gammamu} is equivalent to $\gamma_n\leq 2 \mu_n$ and it holds when $\gamma_n = \theta_n \mu_n$ and $\theta_n \in~]0,2[$, which are the conditions guaranteeing convergence in \cite{Giselsson2021NFBS}.
			%
			%
			
			\section{Particular cases of the inertial and relaxed Nonlinear-Forward-Backward}\label{sec:particularcases}
			In this section, we derive several algorithms with relaxation and inertial steps as specific instances of Algorithm~\ref{algo:algoInertialRelaxed}. In particular, we introduce a new version of FPDHF incorporating both features, which is a novel contribution to the literature. As a consequence, we derive the convergence of relaxed and inertial version of FB, FBF, and FBHF. To the best of the authors' knowledge, there are no existing inertial and relaxed versions of FBHF, being also a new contribution. To this end, we introduce the following problem.
			\begin{pro}\label{pro:mainPD}
				Let $(\H,\scal{\cdot}{\cdot})$ and $(\G,\scal{\cdot}{\cdot})$ 
				be real Hilbert spaces, let $A: \H \to 2^\H$ and $B:\G \to 
				2^\G$ be maximally 
				monotone operators,  let $L\colon \H \to \G$ be 
				a linear bounded operator, 
				let 
				$D:\H 
				\to \H$ be a $\zeta$-Lipschitzian and monotone operator for $\zeta \in 
				\RPP$, and let 
				$C:\H \to \H$ be a 
				$\beta$-cocoercive operator for $\beta \in \RPP$. The problem is to
				\begin{equation} 
					\text{find }  (x,u) \in \H \times \G \text{ such that } 
					\begin{cases}
						0 &\in (A+C+D)x+L^*u\\
						0 &\in B^{-1}u-Lx,
					\end{cases} 
				\end{equation}
				under the hypothesis that its solution set, denoted by 
				$\bm{Z}$,  is nonempty.
			\end{pro}
			Note that, given $(\hat{x},\hat{u}) \in \bm{Z}$, $\hat{x}$ is a 
			solution to the primal monotone inclusion
			\begin{equation}\label{eq:primalinclu}
				\text{find }  x \in \H \text{ such that } 
				0 \in (A+L^*BL+C+D)x
			\end{equation}
			and $\hat{u}$ is a solution to the dual monotone inclusion
			\begin{equation}\label{eq:dualinclu}
				\text{find }  u \in \G \text{ such that } 
				0 \in  B^{-1}u - L(A+C+D)^{-1}(-L^*u).
			\end{equation}
			This inclusion problem has also been studied in multivariate settings and in the context of parallel sums, for example, in \cite{AttouchBricenoCombettes2010,CombettesMinh2022,Comb13,CombettesEckstein2018MP}. We propose the following inertial and relaxed version of FPDHF, presented in  \cite[Algorithm~2.3]{Roldan20254op} for solving Problem~\ref{pro:mainPD}.
			\begin{algo}\label{algo:algo1PD}
				In the context of Problem~\ref{pro:mainPD}, let $(z_0,u_0)\in 
				\H\times \G$, let $(z_{-1},u_{-1})\in 
				\H\times \G$, let $(\sigma,\tau)\in \RPP^2$, let $(\alpha_n)_{n \in \N}$ be a nonnegative sequence such that $\alpha_n \to \alpha \in [0,1[$, let $\lambda_n \in \RPP$, and consider the 
				iteration:
				\begin{equation}\label{eq:algo1PD}
					(\forall n\in\N)\quad 
					\begin{array}{l}
						\left\lfloor
						\begin{array}{l}
							(p_{n},q_n) = (z_n,u_n)+\alpha_n(z_n-z_{n-1},u_n-u_{n-1})\\
							x_{n} = J_{\tau A} (p_n-\tau (L^* q_n+Dp_{n}+C 
							p_n) 
							) \\
							w_{n+1} = x_{n}-\tau(Dx_{n}-Dp_{n})\\
							v_{n+1} =  J_{\sigma
								B^{-1}}\left(q_{n}+\sigma L( 
							x_{n}+w_{n+1}-p_{n})\right)\\
							(z_{n+1},u_{n+1}) = \lambda _n(w_{n+1},v_{n+1})+(1-\lambda_n)(p_n,q_n).
							
						\end{array}
						\right.
					\end{array}
				\end{equation}
			\end{algo}
			\begin{teo}\label{teo:convPD}
				In the context of Problem~\ref{pro:mainPD}, let $(\sigma,\tau)\in \RPP^2$, let $\varepsilon 
				\in 
				\,]0,1[$, let $(z_0,u_0)\in 
				\H\times \G$, let $(z_{-1},u_{-1})\in 
				\H\times \G$, let $(\alpha_n)_{n \in \N}$ be a nonnegative sequence such that $\alpha_n \to \alpha \in [0,1[$, let $(\lambda_n)_{n \in \N}$ be a nonnegative sequence, and let 
				$(z_n)_{n \in 
					\N}$ and $(u_n)_{n \in \N}$ be the sequences 
				generated by Algorithm~\ref{algo:algo1PD}.
				Assume that one of the following assertions hold
				\begin{enumerate}
					\item $(\alpha_n)_{n \in \N}$ is nondecreasing.
					\item $(\alpha_n)_{n \in \N}$ is decreasing and $\sum_{n \in \N} (\alpha_n-\alpha)<+\infty$.
				\end{enumerate}
				Define
				\begin{align}
					\widetilde{\beta} & = \beta(1-\sigma\tau\|L\|^2)\label{eq:wbeta},\\
					\widetilde{\zeta} &= \tau\zeta/(\sqrt{1-\sigma\tau\|L\|^2})\label{eq:wzeta},\\
					\nu &= \begin{cases}
						0, \text{ if } T \colon \H\times \G \to \H \times \G \colon (x,u) \mapsto \tau(Dx, -\tau LDx) \text { is monotone,}\\
						2\widetilde{\zeta}, \text{ otherwise,}
					\end{cases}\label{eq:nu}\\
					\psi	&= \frac{2-\varepsilon+\nu}{1+(\widetilde{\zeta})^2+\nu},\label{eq:defpsiPD}\\
					(\forall n \in \N) \quad \rho_n	&= \frac{\psi}{\lambda_n}-1,\label{eq:defrho}\\
					(\forall n \in \N) \quad \delta_n & = (1-\alpha_{n})\rho_n-\alpha_{n+1}(1-\alpha_{n+1})\rho_{n+1}-\alpha_{n+1}(1+\alpha_{n+1}).\label{eq:defdelta}
				\end{align}
				Assume that $1-\sigma\tau\|L\|^2>0$, $\widetilde{\zeta}<1$,
				\begin{equation}\label{eq:steps0}
					1-(\widetilde{\zeta})^2-\varepsilon>0,
				\end{equation}	
				\begin{equation}\label{eq:steps1}
					\tau \leq 2\widetilde{\beta}\varepsilon,
				\end{equation}	
				for every $n \in \N$, $\rho_n \geq 0$, and that
				\begin{equation}\label{eq:steps2}
					\liminf \delta_n = \delta >0.
				\end{equation}
				Then, there exists $(x,u) \in \bm{Z}$ such that $(z_n,u_n)\weak (x,u)$.	
			\end{teo}
			\begin{proof}
				Let $\HH=\H\times\G$ and consider the operators:
				\begin{equation}\label{eq:defAPD}
					\begin{aligned}
						&\bm{A} \colon \HH \to 2^\HH \colon (x,u) \mapsto (Ax+Dx+L^*u)\times (B^{-1}u-Lx)\\
						&\bm{M} \colon \HH \to \HH \colon (x,u) \mapsto (x/\tau-Dx-L^*u,-Lx+\tau LDx+u/\sigma)\\
						&\bm{C} \colon \HH \to \HH \colon (x,u) \mapsto (Cx,0)\\
						&\bm{S} \colon \HH \to \HH \colon (x,u) \mapsto (x-\tau L^*u,-\tau L x+\tau u /\sigma)\\
						&\bm{T} \colon \HH \to \HH \colon (x,u) \mapsto (x/\tau - D x,u /\tau)
					\end{aligned}
				\end{equation}
				For every $n \in \N$, let
				\begin{equation}
					\bm{M}_n = \bm{M}, \quad \gamma_n = \tau, \quad \y_n=(p_n,q_n),  \quad\x_n = (x_{n},v_{n+1}), \quad \w_n=(w_n,v_n), \text{ and }  \z_n = (z_n,u_n).
				\end{equation}	
				Note that, $\bm{M} = \bm{S}\circ \bm{T}$, thus $\bm{S}^{-1}\circ\bm{M}=\bm{T}$. Moreover, for every $(x,u) \in \HH$ we have,
				\begin{equation}\label{eq:M-C}
					(\bm{M}_n - \bm{C})(x,u)= (x/\tau-Dx-Cx-L^*u,\tau LDx-Lx+u/\sigma),
				\end{equation}
				and
				\begin{equation}\label{eq:M+A}
					(\bm{M}_n + \bm{A})(x,u) = (Ax+x/\tau)\times (B^{-1}u-2Lx+\tau LDx+u/\sigma).
				\end{equation}
				Now, set $(r_n,s_n)=\bm{r}_n = (\bm{M}_n-\bm{C})\y_n$, then, in view of \eqref{eq:M-C}
				\begin{align}\label{eq:proofcor1}
					(\forall n \in \N) \quad \bm{r}_n = (\bm{M}_n - \bm{C})\y_n \Leftrightarrow
					\begin{cases}
						r_n = p_n/\tau-Dp_n-Cp_n-L^*q_n,\\
						s_n = \tau LDp_n-Lp_n+q_n/\sigma.
					\end{cases}
				\end{align}
				Moreover, it follows from \eqref{eq:algo1PD} that, for every $n \in \N$,
				\begin{align}\label{eq:proofcor2}
					\w_{n+1}=(w_{n+1},v_{n+1})&=(x_{n}-\tau(Dx_{n}-Dp_n),v_{n+1})\nonumber\\
					&=(p_n,q_n)-(p_n-\tau Dp_n,q_n)+(x_{n}-\tau Dx_{n},v_{n+1})\nonumber\\
					&=\y_n - \tau  (\bm{T}\y_n -\bm{T}\x_n)\nonumber\\
					&=\y_n - \tau  (\bm{S}^{-1}\bm{M}_n \y_n -\bm{S}^{-1}\bm{M}_n\x_n)\nonumber\\
					&=\y_n - \gamma_n \bm{S}^{-1}(\bm{M}_n \y_n -\bm{M}_n\x_n).
				\end{align}
				Hence, by \eqref{eq:algo1PD}, \eqref{eq:proofcor1}, \eqref{eq:proofcor2}, and \eqref{eq:M+A}, we deduce
				\begin{align*}
					&\begin{cases}
						x_{n} = J_{\tau A} ( p_n-\tau (L^*q_n+Dp_n+Cp_n) )\\
						v_{n+1}  =  J_{\sigma
							B^{-1}}\left(q_{n}+\sigma L( 
						x_{n}+w_{n+1}-p_{n})\right)
					\end{cases} \\
					&\hspace{4cm}\Leftrightarrow	\begin{cases}
						x_{n} = J_{\tau A} ( p_n-\tau (L^*q_n+Dp_n+Cp_n) )\\
						v_{n+1} = J_{\sigma B^{-1}} (q_n+\sigma L(2x_{n}-\tau Dx_n+\tau Dp_{n}-p_n))
					\end{cases}	\\
					&\hspace{4cm}\Leftrightarrow
					\begin{cases}
						x_{n} = J_{\tau A} (\tau r_n)\\
						v_{n+1}= J_{\sigma B^{-1}} (\sigma (s_n+2Lx_{n}-\tau LDx_{n}))
					\end{cases}  \\	 
					&\hspace{4cm}\Leftrightarrow
					\begin{cases}
						r_n \in Ax_{n}+x_{n}/\tau\\
						s_n \in B^{-1}v_{n+1}-2Lx_{n}+\tau LDx_{n}+v_{n+1}/\sigma
					\end{cases}\\
					&\hspace{4cm}\Leftrightarrow \bm{r}_n \in (\bm{M}_n+\bm{A}) \x_n\\
					&\hspace{4cm}\Leftrightarrow
					\x_n=(\bm{M}_n+\bm{A})^{-1}\bm{r}_n.
				\end{align*}  
				Therefore, we deduce that, \eqref{eq:algo1PD}, can be written as
				\begin{align*}
					&\y_n = \z_n + \alpha_n (\z_n - \z_{n-1}),\\
					&\x_n = (\bm{M}_n+\bm{A})^{-1}(\bm{M}_n-\bm{C})\y_n\\
					&\w_{n+1} = \y_n - \gamma_n \bm{S}^{-1}(\bm{M}_n \y_n - \bm{M}_n\x_n), \\
					&\z_{n+1} = \lambda_n \w_{n+1} + (1-\lambda_n)\y_n.
				\end{align*}
				Thus, \eqref{eq:algo1PD} is a particular instance of Algorithm~\ref{algo:algoInertialRelaxed}.	Now, note that, in view of the Lipschitzian property of $D$ and \cite[Lemma~6.1]{MorinBanertGiselsson2022},  we have for every $\z = (x,u) \in \HH$ and every $\x = (y,v) \in \HH$
				\begin{align*}
					\|(\gamma_n\bm{M}_n-\bm{S})\z-(\gamma_n\bm{M}_n-\bm{S})\x\|^2_{\bm{S}^{-1}}&=\|\bm{S^{-1}}(\gamma_n\bm{M}_n-\bm{S})\z-\bm{S^{-1}}(\gamma_n\bm{M}_n-\bm{S})\x\|^2_{\bm{S}} \\
					&= \|(\gamma_n\bm{T}-\id)\z-(\gamma_n\bm{T}-\id)\x\|^2_{\bm{S}} \\
					&= \scal{\tau(Dx-Dy,0)}{\tau\bm{S}(Dx-Dy,0)}\\
					& = \scal{\tau(Dx-Dy,0)}{\tau(Dx-Dy,\tau L(Dy-Dx))}\\
					& =\tau^2\|Dx-Dy\|^2\\
					&\leq \tau^2\zeta^2\|x-y\|^2\\ 
					&\leq \frac{\tau^2\zeta^2}{1-\tau\sigma\|L\|^2}\|(x-y,u- v))\|^2_{\bm{S}}\\
					&=\frac{\tau^2\zeta^2}{1-\tau\sigma\|L\|^2}\|\z-\x\|^2_{\bm{S}}.
				\end{align*}	
				Then, it follows that $\gamma_n \bm{M}_n-\bm{S}$ is $\widetilde{\zeta}$-Lipschitz with respect to $\bm{S}$. Moreover, $\bm{C}$ is $\widetilde{\beta}$-cocoercive with respect to $\bm{S}$ \cite[Corollary~6.1]{MorinBanertGiselsson2022}. Note that, in this case $-(\gamma_n\bm{M}_n-\bm{S})\z=\tau(Dx, -\tau LDx)=T(x,u)$, then $\nu_n \equiv \nu$.
				Hence, $\delta_n$ defined in \eqref{eq:defdelta} coincide with \eqref{eq:defdeltan}. 
				Altogether,  the result follows by invoking Theorem~\ref{teo:WRL}.
			\end{proof}  
			\begin{rem}\label{rem:FPDHF}
				\begin{enumerate}
					\item\label{rem:FPDHF0} \cblue{In view of Theorem~\ref{teo:WRL3}, if the operator $\bm{A}$, defined in~\eqref{eq:defAPD}, is strongly monotone, Algorithm~\ref{algo:algo1PD} converges linearly to the unique solution of Problem~\ref{pro:mainPD}. This occurs, for instance, when $A$ is strongly monotone and $B$ is cocoercive.}
					\item \label{rem:FPDHF1}		In the case where $\alpha_n\equiv0$ and $\lambda_n\equiv1$, Algorithm~\ref{algo:algo1PD} reduces to
								%
					the FPDHF algorithm \cite[Algorithm~2.3]{Roldan20254op}. In this case, \eqref{eq:steps2} reduces to
					\begin{equation*}
						\delta = \rho = \psi-1 = \dfrac{2-\varepsilon+\nu}{1+\nu+(\widetilde{\zeta})^2}-1> 0 \Leftrightarrow 1-\varepsilon - (\widetilde{\zeta})^2> 0.
					\end{equation*}	
					Moreover,  by considering $\varepsilon = \dfrac{\epsilon}{1-\sigma\tau\|L\|^2}$, with $\epsilon>0$, we have
					\begin{equation}\label{eq:stepPD1}
						1-\varepsilon - (\widetilde{\zeta})^2> 0 \Leftrightarrow 1-\epsilon > \sigma\tau\|L\|^2+\tau^2\zeta^2,
					\end{equation}
					and \eqref{eq:steps1}  reduces to
					\begin{equation}\label{eq:stepPD2}
						\tau \leq 2\widetilde{\beta}\varepsilon \Leftrightarrow \tau \leq 2\beta\epsilon.
					\end{equation}
					Note that \eqref{eq:stepPD1} and \eqref{eq:stepPD2} are the conditions on the step-size for guaranteeing the convergence of FPDHF in \cite[Theorem~4.2]{Roldan20254op}. This shows that FPDHF is a particular case of the Nonlinear Forward-Backward algorithm, which has not been previously established.
					\item \label{rem:FPDHF2} Let $f \in \Gamma_0(\H)$, $g\in \Gamma_0(\G)$, $L\colon \H \to \G$ be a linear bounded operator,  $h\colon \H \to \R$ be 
					a differentiable 
					convex function with a $\zeta$-Lipschitz 
					gradient for some $\zeta\in \RPP$, and let 
					$d\colon \G \to \R$ be a differentiable convex function
					with a $1/\beta$-Lipschitz gradient for some 
					$\beta \in \RPP$. Consider the optimization problem,
					\begin{equation}\label{eq:probopti}
						\min_{x \in \H} f(x)+g(Lx)+h(x)+d(x),
					\end{equation}
					under the assumption that its solution set is nonempty. Under standard qualification conditions (see \cite[Proposition 6.19]{bauschkebook2017}) by considering the Fermat's rule \cite[Theorem~16.3]{bauschkebook2017} and subdifferential calculus rules \cite[Theorem~16.47]{bauschkebook2017},
					and setting  $A = \partial f$, $B = \partial g$, $C = \nabla d$, and $D = \nabla h$, the problem in \eqref{eq:probopti} can be solved by Algorithm~\ref{algo:algo1PD}, which, in this setting, iterates as follows:
					\begin{equation}\label{eq:algo1opti}
						(\forall n\in\N)\quad 
						\begin{array}{l}
							\left\lfloor
							\begin{array}{l}
								(p_{n},q_n) = (z_n,u_n)+\alpha_n(z_n-z_{n-1},u_n-u_{n-1})\\
								x_{n} = \prox_{\tau f} (p_n-\tau (L^* q_n+\nabla h (p_{n})+\nabla d 
								(p_n)) 
								) \\
								w_{n+1} = x_{n}-\tau(\nabla h (x_{n})- \nabla h (p_{n}))\\
								v_{n+1} =  \prox_{\sigma
									g^*}\left(q_{n}+\sigma L( 
								x_{n}+w_{n+1}-p_{n})\right)\\
								(z_{n+1},u_{n+1}) = \lambda _n(w_{n+1},v_{n+1})+(1-\lambda_n)(p_n,q_n).
							\end{array}
							\right.
						\end{array}
					\end{equation}
				\end{enumerate}	
			\end{rem}
			
			To initialize Algorithm~\ref{algo:algo1PD}, the parameters $\varepsilon$, $\sigma$, and $\tau$ and the sequences $(\alpha_n)_{n \in \N}$ and $(\lambda_n)_{n \in \N}$ must be chosen to satisfy conditions \eqref{eq:steps0}–\eqref{eq:steps2}. Since this selection may not be straightforward, we propose an initialization strategy. This scheme depends on $\overline{\varepsilon} \in \RPP$ and $\chi \in \RPP$, which are specified in \eqref{eq:defepsilon} and \eqref{eq:rtau2}, respectively. To implement the algorithm, the user only needs to choose $(t,\kappa_1,\kappa_2) \in~]0,1[^3$. The constant $t$ provides some flexibility in the step-size and the relaxation-inertial parameters. The constants $\kappa_1$ and $\kappa_2$ regulate the step-sizes $\tau$ and $\sigma$, respectively.  A detailed explanation of each step of the initialization, can be found below. 	
			
			\begin{ini}\label{ini:1} In the context of Problem~\ref{pro:mainPD}, let $(\alpha_n)_{n \in \N}$ and $(\lambda_n)_{n \in \N}$ be nonnegative sequences. Suppose that $\alpha_n \to \alpha$ and $\lambda_n\to\lambda$ and consider the following initialization procedure for Algorithm~\ref{algo:algo1PD}.
				\begin{enumerate}
					\item\label{ini:11}  
					Choose $t \in ]0,1]$, and set $\varepsilon = t\cdot \overline{\varepsilon}$, where 
					\begin{equation} \label{eq:defepsilon}
						\overline{\varepsilon} = \frac{1-\sqrt{1+16\beta^2\zeta^2}}{8\beta^2\zeta^2} = \frac{2}{1+\sqrt{1+16\beta^2\zeta^2}}.
					\end{equation}	 Compute $\chi \in \RPP$ as 
					\[\chi=\frac{4\beta}{1+\sqrt{1+16\beta^2\zeta^2}}.\]

					\item\label{ini:12}  Choose $\kappa_1 \in ~]0,1[$ and set $\tau = \kappa_1 \chi$. 
					\item\label{ini:13}  Choose $\kappa_2 \in ~ ]0,1[$ and set $\sigma =  \frac{\kappa_2}{\tau\|L\|^2}\left(1-\frac{\tau}{\chi}\right)$.  
					\item\label{ini:14}  Let $\widetilde{\zeta}$, $\nu$, $\psi$ be defined as in \eqref{eq:wzeta}, \eqref{eq:nu}, and \eqref{eq:defpsiPD}, respectively.
					
					\item\label{ini:15}  Select one of the following scenarios: 
					\begin{itemize}
						\item\label{ini:151}  Choose $\alpha \in [0,1[$ and $\lambda$ according \eqref{eq:lambda}, thus
						\begin{equation}\label{eq:lambdaini}
							\lambda \in \left]0, \frac{(1-\alpha)^2}{2\alpha^2-\alpha+1}\cdot \psi \right[.
						\end{equation} 
						\item\label{ini:152} 	Choose $\lambda \in \left]0, \psi\right[$ and $\alpha$ according \eqref{eq:alphaintervalo}, thus
						\begin{equation}\label{eq:alphaini}
							\alpha \in \left[0, \overline{\alpha}(\widetilde{\zeta},\varepsilon,\lambda)\right[.
						\end{equation}
					\end{itemize}
				\end{enumerate}	
			\end{ini}	
			
			Now, we provide explicit values for $\overline{\varepsilon}$ and $\chi$ to be considered in step \eqref{ini:11} in Initialization~\ref{ini:1} and an explanation of each step on it. 
			
			First, note that \eqref{eq:steps0} is equivalent to
			\begin{align}\label{eq:rsigma1}
				1-\varepsilon-\frac{\tau^2\zeta^2}{1-\s\T\|L\|^2}>0 \Leftrightarrow \frac{1-\varepsilon-\tau^2\zeta^2}{(1-\varepsilon)\tau\|L\|^2} > \sigma
			\end{align}
			and \eqref{eq:steps1} is equivalent to
			\begin{align}\label{eq:rsigma2}
				\tau \leq 2\beta (1-\s\T\|L\|^2)\varepsilon \Leftrightarrow \frac{1}{\tau\|L\|^2} \left(1-\frac{\tau}{2\beta \varepsilon}\right) \geq \sigma.
			\end{align}
			Hence, in order to guarantee the existence of $\sigma > 0$ satisfying both \eqref{eq:rsigma1} and \eqref{eq:rsigma2}, it is necessary that
			\begin{equation}\label{eq:rtau1}
				1-\varepsilon-\tau^2\zeta^2 > 0 \textnormal{ and } 1-\tau/(2\beta\varepsilon) >0 \Leftrightarrow \frac{\sqrt{1-\varepsilon}}{\zeta} \geq \tau \textnormal{ and }
				2\beta\varepsilon>\tau.
			\end{equation}
			To allow larger values of $\tau$, following a similar approach to that in \cite[Proposition~2.1.3]{BricenoDavis2018}, we choose $\varepsilon$ such that $\sqrt{1-\varepsilon}/{\zeta} = 2\beta\varepsilon$ which yields \eqref{eq:defepsilon}.
			Therefore,  \eqref{eq:rtau1} holds if
			\begin{equation}\label{eq:rtau2}
				\tau \in ~ ]0, \chi[, \quad \textnormal{ where } \quad \chi =\frac{\sqrt{1-\overline{\varepsilon}}}{\zeta}= 2\beta \overline{\varepsilon}= \frac{4\beta}{1+\sqrt{1+16\beta^2\zeta^2}},
			\end{equation}
			which yield the constants on \eqref{ini:11} in Initialization~\ref{ini:1}. 
			Hence, we have $\tau \in ~ ]0, \chi[$, which yields the step \eqref{ini:12} in Initialization~\ref{ini:1}.
			In order to chose $\sigma$ such that \eqref{eq:rsigma1} and \eqref{eq:rsigma2} hold, note that $1-\overline{\varepsilon} = \chi^2\zeta^2$, thus, \eqref{eq:rsigma1} is equivalent to $(\chi^2-\tau^2)(\chi^2 \tau \|L\|^2)^{-1}>\sigma$. Furthermore, since $2\beta\overline{\varepsilon} = \chi$, \eqref{eq:rsigma2} is equivalent to $(1-\tau/\chi)(\tau\|L\|^2)^{-1} \geq \sigma$. Hence, since $\T < \chi $,  we have 
			\begin{equation}
				\frac{\chi^2-\tau^2}{\chi^2}\frac{1}{\tau \|L\|^2} >\left(1- \frac{\tau}{\chi}\right)\frac{1}{\tau \|L\|^2}.
			\end{equation}
			Then, both \eqref{eq:rsigma1} and \eqref{eq:rsigma2} hold if 
			\begin{equation}\label{eq:rsigma3}
				\s \in \left]0,\frac{1}{\tau\|L\|^2}\left(1-\frac{\tau}{\chi}\right) \right[,
			\end{equation}
			which yields the step \eqref{ini:13} in Initialization~\ref{ini:1}. Thus, we have determined $\overline{\varepsilon}$ and $\chi$ giving the largest interval for $\tau$ and $\sigma$ such that \eqref{eq:steps0} and \eqref{eq:steps1} are feasible. Step \eqref{ini:14} and \eqref{ini:15} in Initialization~\ref{ini:1} guarantee that \eqref{eq:steps2} holds in view of \eqref{eq:lambda} and \eqref{eq:alphaintervalo}.
			\subsection{Particular instances of Algorithm~\ref{algo:algo1PD}}\label{sec:partcases}
			In this section, we present particular instances of Algorithm~\ref{algo:algo1PD} that are of interest in their own right, as they recover and extend inertial and relaxed versions of relevant methods from the literature. In addition, we provide the values of $\overline{\varepsilon}$ and $\chi$ required to apply Initialization~\ref{ini:1} to these methods. 
			
			In the context of Problem~\ref{pro:mainPD}, let $(\alpha_n)_{n \in \N}$ and $(\lambda_n)_{n \in \N}$ be nonnegative sequences such that $\alpha_n \to \alpha$ and $\lambda_n\to\lambda$. Consider the following cases:		
			\begin{itemize}			
				\item \underline{{\bf Condat--V\~u}}: In the case where $D=0$, Algorithm~\ref{algo:algo1PD} reduces to the relaxed and inertial Condat--V\~u algorithm.
				In this case $\zeta=\nu=\widetilde{\zeta}=0$, $\rho = \frac{2-\varepsilon}{\lambda}-1$. Then, by considering $\varepsilon = \tau/(2\widetilde{\beta})$ we have that \eqref{eq:steps1} holds and \eqref{eq:steps0} reduces to
				\begin{equation}\label{eq:stepsizeCV}
					\s\T \|L\|^2+\frac{\tau}{2\beta}<1,
				\end{equation}
				which corresponds with the condition in \cite[Theorem~3.1(i)]{Condat13}. By taking $\chi = 2\beta$,  $\tau \in ~ ]0,2\beta[$, and $\sigma \in  \left]0,\frac{1}{\tau\|L\|^2}\left(1-\frac{\tau}{2\beta}\right)\right[$, it follows that \eqref{eq:rtau1} holds directly, thus, it can be considered $\overline{\varepsilon}=1$ in Initialization~\ref{ini:1}.
				Moreover, according \eqref{eq:defpsiPD} we have $\psi = 2 - \tau(2\beta(1-\tau\sigma\|L\|^2))^{-1}$. Then, if $\alpha =0$, \eqref{eq:lambdaini} holds if 
				\begin{equation}
					\lambda \in \left]0, 2-\frac{\tau}{2\beta(1-\s\T\|L\|^2)} \right[,
				\end{equation}
				which corresponds to the condition in \cite[Theorem~3.1(ii)]{Condat13}. 

				\item \underline{{\bf Chambolle--Pock}}: If $C = D = 0$, Algorithm~\ref{algo:algo1PD} reduces to the relaxed and inertial Chambolle--Pock algorithm (see \cite[Equation~(38)]{MaulenFierroPeypouquet2023}). In this case, by taking $\zeta = 0$ and $\beta \to +\infty$, and according to~\eqref{eq:rtau1}, we can choose $\overline{\varepsilon} = 0$ in Initialization~\ref{ini:1}. Moreover, $\chi$ can be chosen arbitrarily in $\RPP$, and the only condition to be met is $\s\T\|L\|^2 < 1$. Additionally, from~\eqref{eq:defpsiPD}, we have $\psi \to 2$. Hence,~\eqref{eq:lambdaini} is equivalent to
				\begin{equation}\label{eq:condCP2}
					\frac{2(1 - \alpha)^2}{(1 - \alpha)^2 + \alpha(1 + \alpha)} > \lambda.
				\end{equation}
				Notice that, if $\alpha = 0$, \eqref{eq:condCP2} holds for $\lambda \in\, ]0, 2[$, and if $\lambda = 1$, it holds for $\alpha \in\, ]0, 1/3[$.
				
				\item \underline{{\bf Chambolle--Pock combined with FBF}}: In the case where $C=0$, Algorithm~\ref{algo:algo1PD} is a relaxed and inertial algorithm combining Chambolle--Pock and FBF (see \cite[Eq. (44)]{Roldan20254op}). If we take $\beta \to +\infty$, in view of \eqref{eq:steps1}, we have $\varepsilon \to 0$. Hence,  it can be considered $\overline{\varepsilon}=0$ and $\chi=1/\zeta$ in Initialization~\ref{ini:1} (see \eqref{eq:rtau2}). 
				%
				\item \underline{{\bf Forward-Backward-Half-Forward}}: In the case where $B=0$ and $L=0$ Algorithm~\ref{algo:algo1PD} iterates as follows:
				\begin{equation}\label{eq:algoFBHF}
					(\forall n\in\N)\quad 
					\begin{array}{l}
						\left\lfloor
						\begin{array}{l}
							p_{n} = z_n+\alpha_n(z_n-z_{n-1})\\
							x_{n} = J_{\tau A} (p_n-\tau (Dp_{n}+C 
							p_n) 
							) \\
							w_{n+1} = x_{n}-\tau(Dx_{n}-Dp_{n})\\
							z_{n+1} = \lambda w_{n+1} + (1-\lambda)p_n.
						\end{array}
						\right.
					\end{array}
				\end{equation}
				This recurrence is a relaxed and inertial version of FBHF.		
				Initialization~\ref{ini:1} can be considered in this case with $\sigma=0$ ($\kappa_2=0$) and $\overline{\varepsilon}$ and $\chi$ as in \eqref{eq:defepsilon} and \eqref{eq:rtau2}, respectively. Moreover, by \eqref{eq:rtau2}, the step-size $\tau$ must satisfy the condition
				\begin{equation}
					\tau \in 
					\left]0,\frac{4\beta}{1+\sqrt{1+16\beta^2\zeta^2}}\right[,
				\end{equation}
				which corresponds to the step-size condition for FBHF proposed in	\cite[Theorem~2.3]{BricenoDavis2018}. 
				
				Note that the step-size requirements for FBHF were more easily derived, even with inertial and relaxation steps, than those obtained in \cite[Appendix~A.3]{Giselsson2021NFBS}. This highlights that Assumption~\ref{asume:1} simplifies the process of obtaining step-size constraints and, consequently, the practical implementation of the algorithm.
				
				To our knowledge, no previous work in the existing literature has presented inertial or relaxed versions of FBHF.
				
				Note that, in this setting, the operator $-(\gamma_n\bm{M}_n - \bm{S})=\tau(D,0)$ is monotone, thus, we can consider $\nu=0$. 
				\item \underline{{\bf Forward-Backward-Forward}}: In the case where $B=0$, $L=0$, and $C=0$,  Algorithm~\ref{eq:algo1PD} reduces to the inertial and relaxed FBF with constant step-sizes studied in \cite{BotrelaxFBF2023}. If we take $\beta \to +\infty$, in view of \eqref{eq:defepsilon} and \eqref{eq:rtau2}, we have $\overline{\varepsilon} \to 0$ and $\chi \to 1/\zeta$, which values can be considered in Initialization~\ref{ini:1} together with $\sigma = 0$ ($\kappa_2=0$). In this setting we have $-(\gamma_n\bm{M}_n - \bm{S})=\tau D$, which is a monotone operator. Then $\nu=0$ and $\psi = 2/(1+(\tau\zeta)^2)$ (see \eqref{eq:defpsiPD}). Therefore \eqref{eq:lambdaini} is equivalent to
				\begin{equation}\label{eq:condFBF}
					\frac{2(1-\alpha)^2}{((1-\alpha)^2+\alpha(1+\alpha))(1+\zeta^2\tau^2)} >  \lambda.
				\end{equation}
				Note that this condition is slightly less restrictive than the condition in \cite[Equation~(25)]{BotrelaxFBF2023}, where $\nu$ was taken as $2\zeta\tau$. Indeed, since $\zeta\tau<1$,
				\begin{align*}
					\frac{2(1-\alpha)^2}{((1-\alpha)^2+\alpha(1+\alpha))(1+\zeta^2\tau^2)} &>\frac{2(1-\alpha)^2}{((1-\alpha)^2+\alpha(1+\alpha))(1+\zeta\tau)}.
				\end{align*}
			\end{itemize}

			\section{Numerical simulations.}\label{se:NE}
			In this section, we present three numerical experiments for testing the inclusion of inertial and relaxation steps in NFB\footnote{All numerical experiments were implemented in MATLAB on a desktop computer equipped with an Intel Core i7-14700K processor (3.4/5.6~GHz), 64~GB of RAM, and running Windows~11 Pro 64-bit. The code is available in this  \href{https://github.com/cristianvega1995/RELAXED-AND-INERTIAL-NONLINEAR-FORWARD-BACKWARD-ALGORITHM}{repository}}. In particular, we implement  FBHF and FPDHF for solving optimization problems with affine constraints and image restoration tasks. 

			\subsection{ Optimization with affine constraints.} 
			Let $(N,m,p) \in \N^3$, $M \in \R^{m \times N}$, $b \in \R^m$, $s_1, \ldots, s_p \in \RR^N$, and $\mathcal{S} =\{x\in \R^{p}\mid \left(\forall i \in \{1,\ldots,p\}\right) \quad s_i^\top x \leq 0\}$. Consider the following optimization problem presented in \cite[Section~7.1]{BricenoDavis2018}
			\begin{equation}\label{eq:nonlinear}
				\min_{x\in \mathcal{S}} \iota_{[0,1]^{N}}(x)+\frac{1}{2}\|Mx-b\|^{2}.
			\end{equation}
			This problem is equivalent to (see \cite[Section~6.3]{BricenoDavis2018})
			\begin{equation}\label{eq:oncluenonlinear}
				\textnormal{ find } \quad (x,u) \in \R^{N} \times \R^p,  \quad \textnormal{ such that } (0,0) \in A(x,u) + C(x,u) + D(x,u),
			\end{equation}
			where $A = \mathcal{N}_{[0,1]^N} \times \mathcal{N}_{\RP^p}$, $C(x,u) = (M^\top (Mx-b),0)$, $D(x,u) = (S^\top u,-Sx)$, and $S = [s_1, \ldots, s_p]^\top$.
			Note that $C$ is $(\|M\|^{-2})$-cocoercive and $D$ is $\|S\|$-Lipschitz.  Hence, \eqref{eq:oncluenonlinear} is a particular instance of Problem~\ref{pro:mainPD} by considering $B=0$ and $L=0$ and it can be solved by FBHF. In view of \cite[Proposition~23.18]{bauschkebook2017}, for $(x,u) \in \R^N\times \R^p$, we have $J_{\tau A}(x,u) = ((\min(\max(x_i,0),1))_{1\leq i\leq N},(\max(u_j,0)_{1\leq j\leq p}))$. \cblue{We solve this problem repeatedly using FBHF with and without inertial and relaxation steps. We test different values of the relaxation parameters as well as both constant and decreasing inertial sequences. In particular,} we consider nine instances of $(N,m,p)$, as presented in Table~\ref{Tab:FBHFresult}. The implemented inertial and relaxation parameters are listed in Table~\ref{Tab:decinertial} and Table~\ref{Tab:FBHFresult}. Note that these parameters depend on $\overline{\alpha}=\overline{\alpha}(\tau \|R\|,\varepsilon,\lambda)$ and $\psi=\psi(\tau \|R\|,\varepsilon)$, defined in \eqref{eq:alphaintervalo} and \eqref{eq:defpsi}, respectively.  The step-size is set to $\tau = 2\varepsilon/\|M\|^2$ for $\varepsilon = t(1+\sqrt{1+16\|S\|^2/\|M\|^4})^{-1}$ and $t \in [0,1[$. \cblue{The specific value of $t$ used in each case is given in Table~\ref{Tab:FBHFresult}}.  The stopping criterion is a relative error with tolerance $10^{-6}$ and a maximum of $10^6$ iterations. \cblue{For each of the nine instances of} $(N,m,p)$ described in Table~\ref{Tab:FBHFresult}, we generate 20 random realizations of the matrices $M$ and $S$ using the MATLAB function {\it randn}. The results expressed as the average number of iteration and average CPU time in seconds, \cblue{over each group of 20 realizations}, are reported in Table~\ref{Tab:FBHFresult}. From this table, we observe that, in all \cblue{nine} cases, FBHF outperforms both IFBHF \cblue{(inertial FBHF)} and RIFBHF \cblue{(relaxed and inertial FBHF)} in terms of the average number of iterations (IN) and average computation time (T). This indicates that, in general, selecting larger step-sizes yields better numerical performance than prioritizing larger relaxation parameters at the expense of reducing step-sizes. On the other hand, the best instance of DIFBHF (decreasing inertial FBHF) provides better results in eight out of the nine \cblue{dimension settings}. It is worth to \cblue{mention} that the single \cblue{dimension}, namely $(N,m,k)=(2000,500,100)$, where it does not lead to acceleration, only 2 out of 20 random realizations \cblue{for $M$ and $S$} failed to converge, resulting in a high average value. However, if we consider a sequence of smaller inertial parameters the average results are better \cblue{for this dimension}, as we have shown in Table~\ref{Tab:FBHFresult2}. This highlights the benefit of using decreasing inertial parameters, which allows for larger step-sizes and larger relaxation parameters in the initial iterations. \cblue{In Figure~\ref{fig:resFBHF}, we plot the residual value, given by $\|p_n-x_n-\tau (Dp_n+Cp_n)+\tau( Dx_n + Cx_n)\|$ (see \eqref{eq:algoFBHF}), versus the iteration number for the best cases of IFBHF, DIFBHF, and RIFBHF, considering three different values of $(N,m,p)$. For each $(N,m,p)$ case, the plot corresponds to the first of the 20 random realizations of 
				$M$ and $S$. From this figure, we observe that the residual value decreases faster for the decreasing inertial parameters.} A drawback of this decreasing scheme is the lack of a general rule to select the sequence of inertial parameters. For instance, in five cases the best performance is obtained with the sequence $(\alpha_n^2)_{n}$ while in three cases it is achieved with $(\alpha_n^3)_{n \in \N}$. In addition, the sequence $(\alpha_n^1)_{n}$ consistently yields worse convergence results.
			\begin{table}[htbp]
				\centering
				\renewcommand{\arraystretch}{1.2}
				\setlength{\tabcolsep}{6pt}
				\begin{tabular}{|c|c|}
					\hline 
					\multirow{3}{*}{Decreasing Inertia} 
					& $\alpha_n^1=(1+0.001\cdot n\cdot(\log(n))^{1.001})^{-1}$ \\ 
					& $\alpha_n^2=(3+0.00001\cdot n\cdot(\log(n))^{1.00001})^{-1}$\\  
					& $\alpha_n^3=(9+0.00001\cdot n\cdot(\log(n))^{1.00001})^{-1}$ \\ 
					\hline
				\end{tabular}
				\caption{Decreasing and summable inertial parameters.}\label{Tab:decinertial}
			\end{table}
			\begin{table}[h!]
				\centering
				\begin{tabular}{cccccccccc}
					&  &  &  & \multicolumn{2}{c}{$(N,m,p)$} & \multicolumn{2}{c}{$(N,m,p)$} & \multicolumn{2}{c}{$(N,m,p)$} \\
					\cmidrule(lr){5-10}
					Algorithm & $\alpha_n$ & $\lambda$ & $t$ & IN & T & IN & T & IN & T \\
					\toprule\\[-0.5em]
					& & & & \multicolumn{2}{c}{(2000,1000,100)} & \multicolumn{2}{c}{(2000,1000,500)} & \multicolumn{2}{c}{(2000,1000,800)} \\
					\cmidrule(lr){5-10}
					FBHF & 0 & 1 & 0.999
					& 16955 & 2.95
					& 17975 & 8.02
					& 36452 & 27.25 \\
					\hline
					\multirow{3}{*}{IFBHF} 
					& \multirow{3}{*}{$\epsilon\cdot\overline{\alpha}$}& 1 & 0.8 & 17965 & 3.31 & 19127 & 8.50 & 38634 & 28.87 \\
					&  & 1 & 0.9 & 17237 & 3.16 & 18292 & 8.20 & 37054 & 27.26 \\
					&  & 1 & 0.999 & 16943 & 3.04 & 17961 & 8.08 & 36426 & 26.67 \\
					\hline
					\multirow{3}{*}{DIFBHF} 
					& $\alpha_n^1$ & 1 & 0.999 & 26638 & 4.69 & 30625 & 13.65 & 60074 & 44.93 \\
					& $\alpha_n^2$ & 1 & 0.999 & {\bf 13252} & {\bf 2.34} & {\bf 13680} & {\bf 5.83} & {\bf 29134} & {\bf 21.94} \\
					& $\alpha_n^3$ & 1 & 0.999 & 15667 & 2.76 & 16477 & 7.16 & 33745 & 25.18 \\
					\hline
					\multirow{2}{*}{RIFBHF} 
					& \multirow{2}{*}{$\epsilon\cdot\overline{\alpha}$} &  $0.95\cdot \psi$  & 0.9 & 17132 & 3.06 & 18179 & 7.98 & 36839 & 27.45 \\
					&  & $\epsilon \cdot \psi$ & 0.9 & 17078 & 3.05 & 18116 & 8.18 & 36721 & 27.35 \\
					\toprule\\[-0.5em]
					& & & & \multicolumn{2}{c}{(1500,1000,100)} & \multicolumn{2}{c}{(2500,1000,100)} & \multicolumn{2}{c}{(3500,1000,100)} \\
					\cmidrule(lr){5-10}
					FBHF & 0 & 1& 0.999
					& 10853 & 0.95
					& 1199 & 0.61
					& 701 & 1.11 \\
					\hline
					\multirow{3}{*}{IFBHF} 
					& \multirow{3}{*}{$\epsilon\cdot\overline{\alpha}$} & 1 & 0.8 & 11574 & 1.00 & 1270 & 0.64 & 782 & 1.22 \\
					&   & 1 & 0.9 & 11051 & 0.95 & 1205 & 0.60 & 732 & 1.16 \\
					&   & 1 & 0.999 & 10844 & 0.93 & 1198 & 0.60 & 700 & 1.11 \\
					\hline
					\multirow{3}{*}{DIFBHF} 
					& $\alpha_n^1$ & 1 & 0.999 & 16946 & 1.44 & 44556 & 22.75 & 98885 & 153.93 \\
					& $\alpha_n^2$ & 1 & 0.999 & {\bf 8042} & {\bf 0.69} & 80181 & 40.96 & 100000 & 158.33 \\
					& $\alpha_n^3$ & 1 & 0.999 & 9898 & 0.86 & {\bf 1061} & {\bf 0.55} & {\bf 629} & {\bf 0.98} \\
					\hline
					\multirow{2}{*}{RIFBHF} 
					& \multirow{2}{*}{$\epsilon\cdot\overline{\alpha}$} &  $0.95\cdot \psi$ & 0.9 & 10980 & 0.95 & 1216 & 0.61 & 711 & 1.13 \\
					&  & $\epsilon \cdot \psi$ & 0.9 & 10941 & 0.95 & 1211 & 0.61 & 708 & 1.12\\
					\toprule\\[-0.5em]
					& & & & \multicolumn{2}{c}{(2000,500,100)} & \multicolumn{2}{c}{(2000,800,100)} & \multicolumn{2}{c}{(2000,1500,100)} \\
					\cmidrule(lr){5-10}
					FBHF & 0 & 1 & 0.999
					& {\bf 530} & {\bf 0.11}
					& 1212 & 0.21
					& 16586 & 3.39 \\
					\hline
					\multirow{3}{*}{IFBHF} 
					& \multirow{3}{*}{$\epsilon\cdot\overline{\alpha}$} & 1 & 0.8 & 574 & 0.12 & 1315 & 0.24 & 17679 & 3.65 \\
					&   & 1 & 0.9 & 536 & 0.11 & 1243 & 0.22 & 16887 & 3.49 \\
					&   & 1 & 0.999 & 530 & 0.11 & 1210 & 0.21 & 16573 & 3.40 \\
					\hline
					\multirow{3}{*}{DIFBHF} 
					& $\alpha_n^1$ & 1 & 0.999 & 80474 & 15.98 & 28818 & 4.65 & 24484 & 5.00 \\
					& $\alpha_n^2$ & 1 & 0.999 & 100000 & 22.19 & 80198 & 14.03 & {\bf 12511} & {\bf 2.52} \\
					& $\alpha_n^3$ & 1 & 0.999 & 10421 & 1.68 & {\bf 1094} & {\bf 0.18} & 15168 & 3.10 \\
					\hline
					\multirow{2}{*}{RIFBHF} 
					& \multirow{2}{*}{$\epsilon\cdot\overline{\alpha}$} &  $0.95\cdot \psi$ & 0.9 & 538 & 0.11 & 1229 & 0.21 & 16779 & 3.42 \\
					&  & $\epsilon \cdot \psi$ & 0.9 & 536 & 0.12 & 1224 & 0.21 & 16720 & 3.34\\
					\midrule
				\end{tabular}
				\caption{Comparison of FBHF with and without inertial and relaxation steps. The step-size implemented is  $\tau = 2\varepsilon/\|M\|^2$ for $\varepsilon = t(1+\sqrt{1+16\|R\|^2/\|M\|^4})^{-1}$ and $t \in [0,1[$. We set $\epsilon =0.9999$. The parameters $\overline{\alpha}=\overline{\alpha}(\tau \|R\|,\varepsilon,\lambda)$ and $\psi=\psi(\tau \|R\|,\varepsilon)$ are defined in \eqref{eq:alphaintervalo} and \eqref{eq:defpsi}, respectively. The decreasing inertial parameters $\alpha_n^1$, $\alpha_n^2$, $\alpha_n^3$ are defined in Table~\ref{Tab:decinertial}. The best results obtained in terms of average number of iterations \cblue{(IN)} and average CPU time \cblue{(T)} are highlighted in bold.}\label{Tab:FBHFresult}
			\end{table}
			\begin{table}[htbp]
				\centering
				\captionsetup{width=\linewidth}
				\renewcommand{\arraystretch}{1.2}
				\setlength{\tabcolsep}{6pt}
				\begin{tabular}{|c|c|c|c|cc|}
					\hline
					Algorithm & $\alpha_n$ & $\lambda$ & $t$ 
					& IN & T \\
					\hline
					DIFBHF
					& $(10+0.00001\cdot n\cdot(\log(n))^{1.00001})^{-1}$ & 1 & 0.999 & 477 & 0.08 \\
					\hline
				\end{tabular}
				\caption{Result of DIFBHF for a sequence of smaller inertial parameters and $(N,m,k)=(2000,500,100)$ \cblue{in terms of average number of iterations (IN) and average CPU time (T)}.}\label{Tab:FBHFresult2}
			\end{table}
			\begin{figure}
				\centering
				\subfloat[$(N,m,k)=(2000,1000,500)$]{\label{fig:resFBHF1}\includegraphics[width=0.32\textwidth]{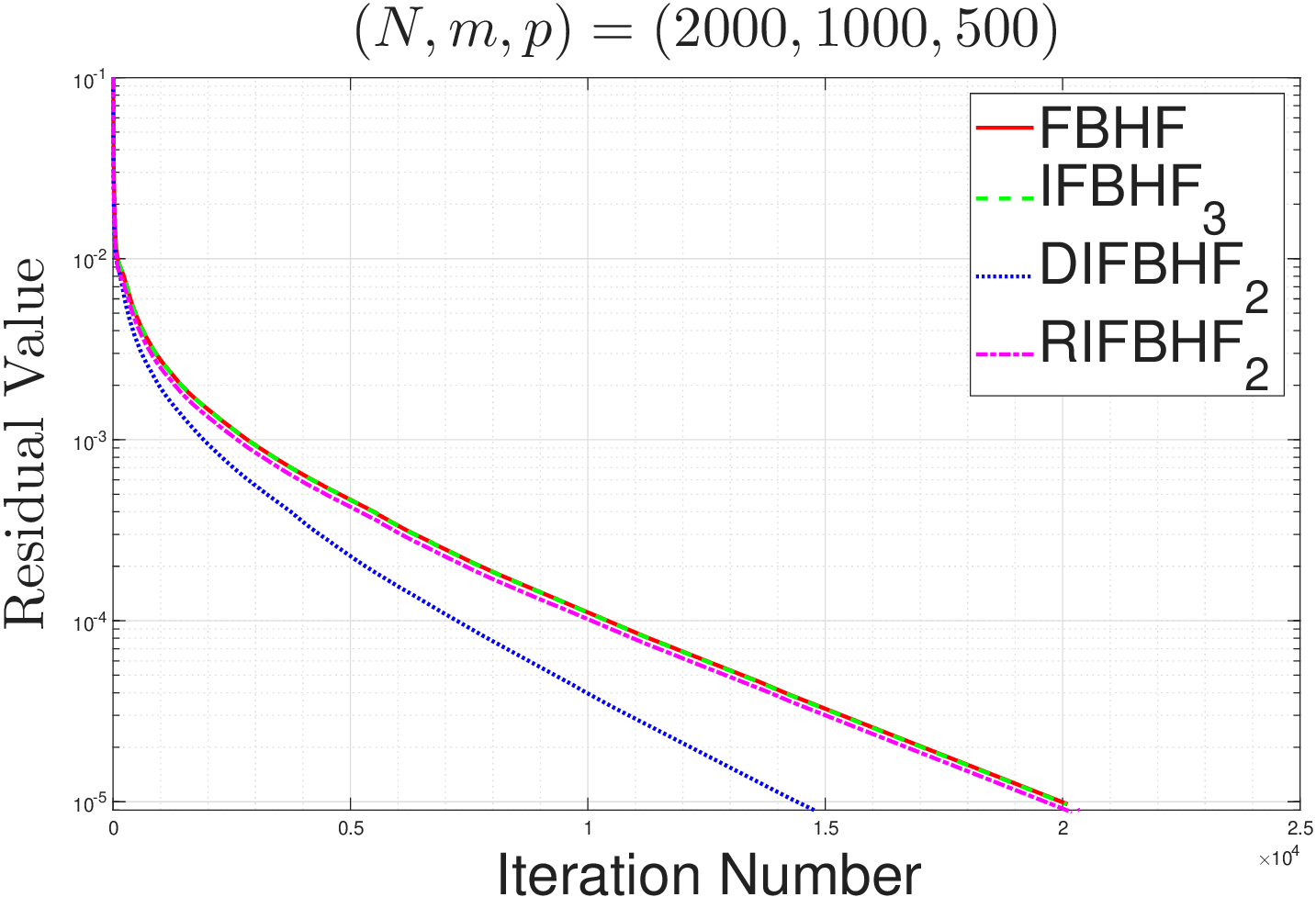}}\,
				\subfloat[$(N,m,k)=(2500,1000,100)$]{\label{fig:resFBHF2}\includegraphics[width=0.32\textwidth]{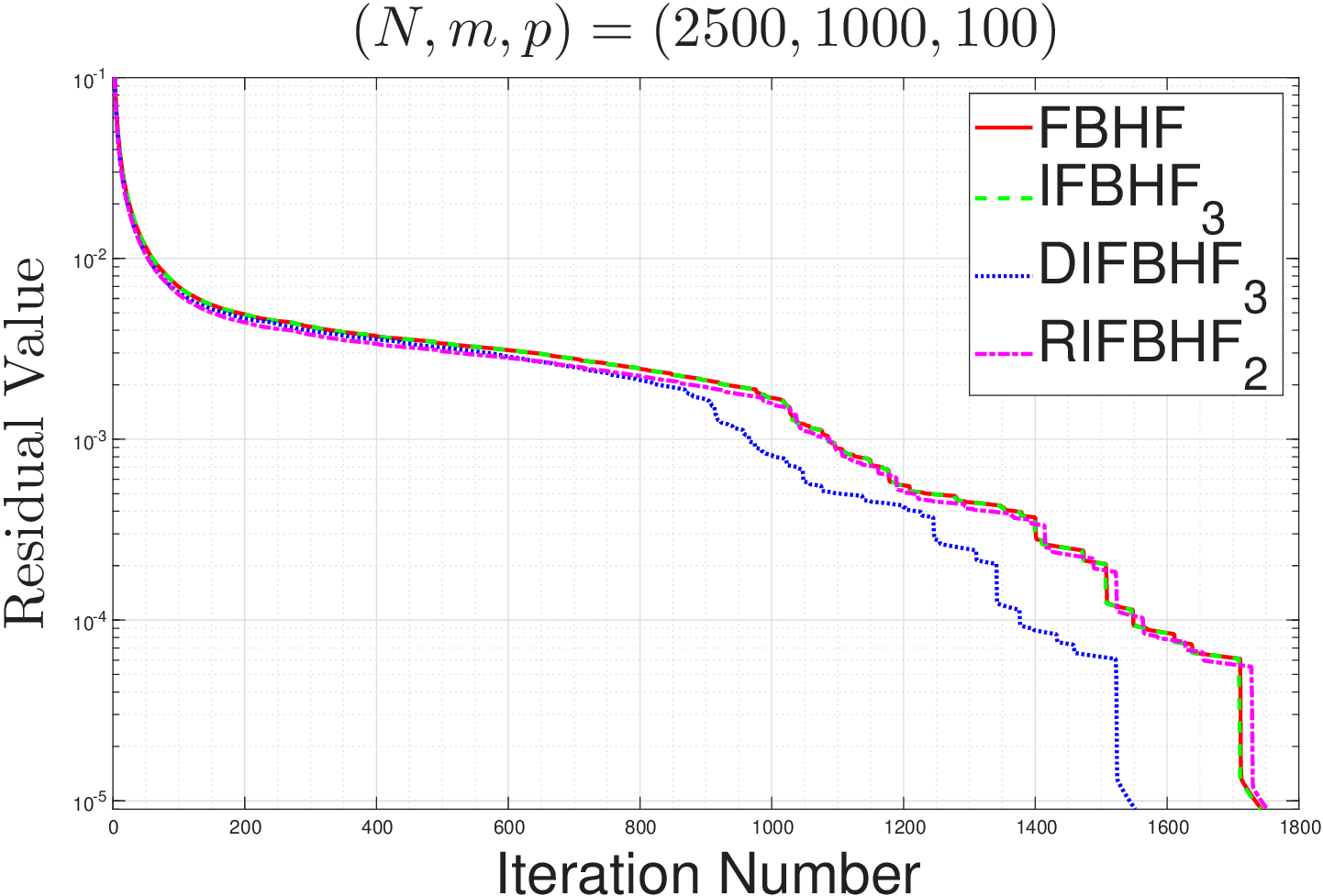}}
				\subfloat[$(N,m,k)=(2000,800,100)$]{\label{fig:resFBHF3}\includegraphics[width=0.32\textwidth]{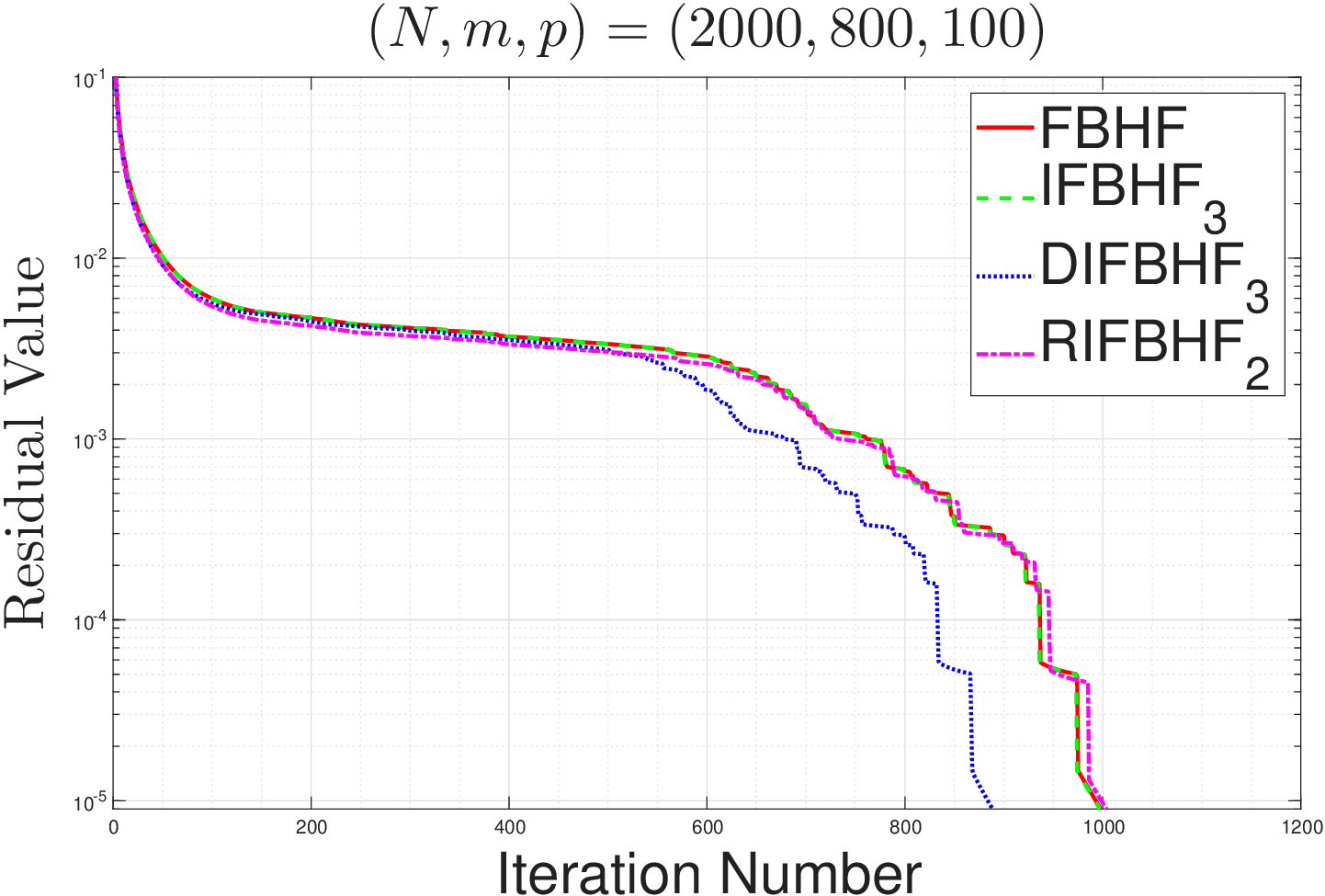}}
				\captionsetup{width=\textwidth} \caption{\cblue{Residual value ($\|p_n-x_n-\tau (Dp_n+Cp_n)+\tau( Dx_n + Cx_n)\|$) for FBHF and for the best performing instances of IFBHF, DIFBHF, and RIFBHF. The graphs display the results for the first of the 20 random realizations of $M$ and $S$.}}
				\label{fig:resFBHF}
			\end{figure}
			\subsection{Image restoration}\label{sec:imagerestoration}
			{To compare the inertial and relaxed version of FPDHF with \cblue{its} classical version, we consider the image restoration problem proposed in \cite[Section~6]{Roldan20254op}. 
				
				Let $(N,M) \in \N^2$ and let $x \in \R^{N\times N}$ be an image to be recovered from 
				an observation
				\begin{equation}\label{eq:modelim}
					z = Tx+\epsilon,
				\end{equation}	
				where { $T\colon \R^{N\times N}\times \R^{M\times M}$} is an operator representing a blur process and $\epsilon$ 
				is an additive noise perturbing the observation. The original image $x$ can be approximated by solving the following optimization problem.
				\begin{equation}\label{eq:opex}
					\min_{x \in [0,1]^{N\times N}} 
					\frac{1}{2}\|Tx-z\|^2+\mu_1\|\nabla 
					x\|_1+\mu_2 H_\delta(Wx),
				\end{equation}
				where $(\mu_1,\mu_2) \in \RPP^2$ are regularization parameters, $\|\cdot\|_1$ is the 
				$\ell^1$ norm, and $H_\delta:\R^{N\times N}\to \R$ is the Huber function defined for $\delta >0$ (see \cite{BricenoPustelnik2023} for details on the Huber function). 
				%
				The linear operator $\nabla \colon \R^{N\times N} \to \R^{N \times N} \times \R^{N \times N}$ is the 
				discrete gradient with Neumann boundary conditions. Its adjoint $\nabla^*$ is 
				the discrete divergence \cite{TV-chambolle}.
				The operator $W:\R^{N\times N}\to \R^{N\times N}$ is an orthogonal basis wavelet transform (we assume that $N$ is a power of 2 and we choose a Haar basis of level 3). For further details on the model see \cite[Section~6]{Roldan20254op}. 
				By setting $f=\iota_{[0,1]^{N\times N}}$, 
				$g=\mu_1\|\cdot\|_1$, $L=\nabla$, $d = \|T(\cdot)-z\|_2^2/2$, and $h =\mu_2 (H_\delta \circ  W)$, the optimization problem in equation \eqref{eq:opex} is a particular instance of the problem in \eqref{eq:probopti} and it can be solved by FPDHF in its variational version \eqref{eq:algo1opti}. For $x \in \R^{N\times N}$, we have $\prox_{\tau f} (x) = ((\min(\max(x_{i,j},0),1))_{1\leq i,j\leq N}$. Closed formulas for $\prox_{\sigma g^*}$ and $\nabla h$, can be found in \cite[Example~24.20]{bauschkebook2017} and \cite[Section~6.1]{ChouzenouxPesquetRoldan2023}, respectively. Note that, $\nabla d$ is 
				$\beta$-cocoercive and $\nabla h$ is 
				$\zeta$-Lipschitz, where $\beta = \|T\|^{-2}$ and $\zeta = \mu_2/\delta$. In the numerical experiments, we consider three scenarios for the operator $T$ being a blur model performed by a $3\times 3$ averaging kernel, $9\times 9$ averaging kernel, and $3\times 3$ with Gaussian kernel. Particularly, $T$ was implemented using the imfilter function in MATLAB with symmetric boundary conditions (we have $\|T\|=1$ in all cases). The observation $z$ is obtained through \eqref{eq:modelim} where $\epsilon$ is an additive zero-mean white Gaussian noise with standard deviation $10^{-3}$ (implemented by the imnoise function in MATLAB). We generate 20 random observations trough $T$ and $\epsilon$ for $N \in \{128,256,512\}$. We set $\mu_1=10^{-2}$, $\mu_2 =10^{-3}$ and $\delta = 10^{-2}$ in our experiments (see \cite[Section~6]{Roldan20254op} for a discussion on these parameters). We consider Initialization~\ref{ini:1} for running our experiments with $\varepsilon = 2t(1+\sqrt{1+\zeta^2\beta^2})^{-1}$, for $t \in ]0,1[$, and $\chi$ defined in \eqref{eq:rtau2}. Furthermore, we first tune up the parameter $\kappa_1$ to obtain the best result of convergence in classical FPDHF. The implemented parameters $\kappa_1$ and $\kappa_2$, for each case, are described in Table~\ref{Tab:FPDHFresult}. From this table we observe that the constant inertial version (IFPDHF) and the relaxed-inertial version (RIFBHF) are comparable in terms of iteration number and CPU time to the classical version of FPDHF. Note that, despite RIFBHF present lower average iteration number, it needs more time to complete the iterations, this is attributable to the additional inertial and relaxation step on its iterations. On the other hand, from Table~\ref{Tab:FPDHFresult}, we observe that the \cblue{decreasing inertial version (DIFPDHF)}, produce numerical advantages in terms of both, iteration number and CPU time. In particular, the sequence $(\alpha_n^2)_{n \in \N}$ has the best performance reducing over $20\%$ of average iteration number and CPU time in comparison to FPDHF, for the $3\times 3$ averaging kernel and $9\times 9$ averaging kernel. On the other hand, the sequence $(\alpha_n^1)_{n \in \N}$ provides better results for the $3 \times 3$ Gaussian Kernel reducing almost $50\%$ of average iteration number and CPU time. However, the lacking of a rule for choosing the inertial parameter is a drawback of this method as we mentioned in previous section.
				
				From Table~\ref{Tab:FPDHFresult}, we observe that the constant inertial version (IFPDHF) and the relaxed-inertial version (RIFPDHF) are comparable to the classical FPDHF in terms of both iteration number and CPU time. Although RIFPDHF achieves slightly lower average iteration numbers, it requires more time to complete the iterations. This increase in computational time can be attributed to the additional inertial and relaxation steps involved in each iteration. On the other hand, Table~\ref{Tab:FPDHFresult} clearly shows that using decreasing inertial sequences provides numerical advantages in both iteration number and CPU time. In particular, the sequence $(\alpha_n^2)_{n \in \N}$ yields the best overall performance, achieving more than a $20\%$ reduction in average iteration number and CPU time compared to FPDHF for both the $3\times 3$ and $9\times 9$ averaging kernels. Moreover, the sequence $(\alpha_n^1)_{n \in \N}$ performs best in the case of the $3\times 3$ Gaussian kernel, with reductions of nearly $50\%$ in both metrics. \cblue{In Figure~\ref{fig:resFBHF}, we plot the residual value, given by $\|(p_n-w_{n+1}-\tau(L^*(q_n-v_{n+1})+Cp_n-Cx_n),q_n-v_{n+1}+\sigma (L(w_{n+1}-p_n))\|$ (see \eqref{eq:algo1PD}),  versus the iteration number for the best cases of IFPDHF, DIFPDHF, and RIFPDHF, considering three different blur operators and dimensions. In each case, the plot corresponds to the first of the 20 random realizations of $z$. From Figures~\ref{fig:resFPDHF1} and \ref{fig:resFPDHF2}, we observe that the residual value decreases faster for the decreasing inertial parameters. In Figure~\ref{fig:resFPDHF3}, we observe that the decreasing inertial parameters produce a slower decay of the residual value during the first iterations; however, the convergence accelerates in the final iterations, reaching the stopping criterion earlier.} As discussed in the previous section, a drawback of using decreasing inertial parameters is the lack of a general rule for choosing the sequence. The original, blurred, and noisy images are shown in Figures~\ref{fig:x128}--\ref{fig:x512}, along with the recovered images for specific instances described in the corresponding captions.
					\begin{table}[htbp]
						\centering
						\captionsetup{width=\linewidth}
						\renewcommand{\arraystretch}{1.2}
						\setlength{\tabcolsep}{6pt}
						\begin{tabular}{cccccccccc}
							&  &  &  & \multicolumn{2}{c}{$(N,\kappa_1,\kappa_2)$} & \multicolumn{2}{c}{$(N,\kappa_1,\kappa_2)$} & \multicolumn{2}{c}{$(N,\kappa_1,\kappa_2)$} \\
							\cmidrule(lr){5-10}
							Algorithm & $\alpha_n$ & $\lambda$ & $t$ & IN & T & IN & T & IN & T \\
							\toprule\\[-0.5em]
							\multicolumn{4}{c}{ $3\times 3$ Averaging Kernel} 
							& \multicolumn{2}{c}{(128,0.17,0.99)} 
							& \multicolumn{2}{c}{(256,0.24,0.99)} 
							& \multicolumn{2}{c}{(512,0.31,0.99)} \\
							\cmidrule(lr){5-10}
							FPDHF & 0 & 1 & 0.999 
							& 865 & 0.87 
							& 481 & 1.51 
							& 494 & 14.81 \\
							\hline
							\multirow{3}{*}{IFPDHF} 
							& \multirow{3}{*}{$\epsilon\cdot\overline{\alpha}$} & 1 & 0.8 
							& 901 & 0.91 
							& 614 & 1.96 
							& 592 & 17.55 \\
							& & 1 & 0.9 
							& 797 & 0.80 
							& 502 & 1.60 
							& 493 & 14.63 \\
							& & 1 & 0.999 
							& 852 & 0.86 
							& 474 & 1.51 
							& 488 & 14.52 \\
							\hline
							\multirow{3}{*}{DIFPDHF} 
							& $\alpha_n^1$ & 1 & 0.999 
							& 818 & 0.82 
							& 469 & 1.50 
							& 507 & 15.08 \\
							& $\alpha_n^2$ & 1 & 0.999 
							& {\bf 691} & {\bf 0.70} 
							& {\bf 368} & {\bf 1.18} 
							& {\bf 385} & {\bf 11.44} \\
							& $\alpha_n^3$ & 1 & 0.999 
							& 808 & 0.81 
							& 445 & 1.43 
							& 459 & 13.64 \\
							\hline
							\multirow{2}{*}{RIFPDHF} 
							& \multirow{2}{*}{$\epsilon\cdot\overline{\alpha}$} & $0.95\cdot \psi$ & 0.95 
							& 855 & 0.89 
							& 476 & 1.57 
							& 490 & 15.43 \\
							& & $\epsilon \cdot \psi$ & 0.999 
							& 852 & 0.88 
							& 474 & 1.56 
							& 488 & 15.36 \\
							\toprule\\[-0.5em]
							\multicolumn{4}{c}{ $9\times 9$ Averaging Kernel} 
							& \multicolumn{2}{c}{(128,0.29,0.99)} 
							& \multicolumn{2}{c}{(256,0.52,0.99)} 
							& \multicolumn{2}{c}{(512,0.59,0.99)} \\
							\cmidrule(lr){5-10}
							FPDHF & 0 & 1 & 0.999 
							& 2400 & 2.58 
							& 1395 & 4.45 
							& 1522 & 45.54 \\
							\hline
							\multirow{3}{*}{IFPDHF} 
							& \multirow{3}{*}{$\epsilon\cdot\overline{\alpha}$} & 1 & 0.8 
							& 2991 & 3.24 
							& 1705 & 5.54 
							& 1834 & 54.41 \\
							& & 1 & 0.9 
							& 2502 & 2.72 
							& 1442 & 4.68 
							& 1558 & 46.18 \\
							& & 1 & 0.999 
							& 2369 & 2.57 
							& 1384 & 4.49 
							& 1511 & 44.82 \\
							\hline
							\multirow{3}{*}{DIFPDHF} 
							& $\alpha_n^1$ & 1 & 0.999 
							& 1956 & 2.12 
							& 1276 & 4.15 
							& 1399 & 41.51 \\
							& $\alpha_n^2$ & 1 & 0.999 
							& {\bf 1871} & {\bf 2.03} 
							& {\bf 1097} & {\bf 3.56} 
							& {\bf 1189} & {\bf 35.26} \\
							& $\alpha_n^3$ & 1 & 0.999 
							& 2237 & 2.43 
							& 1302 & 4.23 
							& 1417 & 42.03 \\
							\hline
							\multirow{2}{*}{RIFPDHF} 
							& \multirow{2}{*}{$\epsilon\cdot\overline{\alpha}$} & $0.95\cdot \psi$ & 0.95 
							& 2375 & 2.68 
							& 1387 & 4.65 
							& 1515 & 47.66 \\
							& & $\epsilon \cdot \psi$ & 0.999 
							& 2368 & 2.67 
							& 1384 & 4.64 
							& 1511 & 47.56
							\\
							\toprule\\[-0.5em]
							\multicolumn{4}{c}{ $3\times 3$ Gaussian Kernel} 
							& \multicolumn{2}{c}{(128,0.05,0.99)} 
							& \multicolumn{2}{c}{(256,0.1,0.99)} 
							& \multicolumn{2}{c}{(512,0.1,0.99)} \\
							\cmidrule(lr){5-10}
							FPDHF & 0 & 1 & 0.999 
							& 1972 & 2.04 
							& 1109 & 3.53 
							& 1265 & 37.75 \\
							\hline
							\multirow{3}{*}{IFPDHF} 
							& \multirow{3}{*}{$\epsilon\cdot\overline{\alpha}$} & 1 & 0.8 
							& 2533 & 2.65 
							& 1435 & 4.64 
							& 1647 & 48.68 \\
							& & 1 & 0.9 
							& 2109 & 2.20 
							& 1192 & 3.85 
							& 1362 & 40.23 \\
							& & 1 & 0.999 
							& 1935 & 2.02 
							& 1090 & 3.52 
							& 1244 & 36.73 \\
							\hline
							\multirow{3}{*}{DIFPDHF} 
							& $\alpha_n^1$ & 1 & 0.999 
							& {\bf 1113} & {\bf  1.16 }
							& {\bf  563} & {\bf  1.82} 
							& {\bf  674} & {\bf  19.89} \\
							& $\alpha_n^2$ & 1 & 0.999 
							& { 1555} & { 1.63} 
							& { 857} & { 2.77} 
							& { 984} & { 29.04} \\
							& $\alpha_n^3$ & 1 & 0.999 
							& 1839 & 1.92 
							& 1030 & 3.33 
							& 1176 & 34.72 \\
							\hline
							\multirow{2}{*}{RIFPDHF} 
							& \multirow{2}{*}{$\epsilon\cdot\overline{\alpha}$} & $0.95\cdot \psi$ & 0.95 
							& 1946 & 2.13 
							& 1094 & 3.65 
							& 1249 & 39.12 \\
							& & $\epsilon \cdot \psi$ & 0.999 
							& 1935 & 2.11 
							& 1089 & 3.63 
							& 1243 & 38.98 \\
							\hline
						\end{tabular}
						\caption{Comparison of FPDHF with and without inertial and relaxation steps. The step-size $(\tau,\sigma)$ is generated by Initialization~\ref{ini:1} for $\varepsilon = t(1+\sqrt{1+16\beta^2/\zeta^2})^{-1}$, where $\beta = \|T\|^{-2}$, $\zeta = \mu_2\delta$, $t \in [0,1[$, the parameter $\chi$ is defined in \eqref{eq:rtau2}, and $\kappa_1$, and $\kappa_2$ are specified on this table. We set $\epsilon =0.9999$. The parameters $\overline{\alpha}=\overline{\alpha}(\tau \|R\|,\varepsilon,\lambda)$ and $\psi=\psi(\tau \|R\|,\varepsilon)$ are defined in \eqref{eq:alphaintervalo} and \eqref{eq:defpsi}, respectively. The decreasing inertial parameters $\alpha_n^1$, $\alpha_n^2$, and $\alpha_n^3$ are defined in Table~\ref{Tab:decinertial}. The best results obtained in terms of average number of iterations \cblue{(IN)} and average CPU time \cblue{(T)} are highlighted in bold.}\label{Tab:FPDHFresult}
					\end{table}
					\begin{figure}
						\centering
						\subfloat[\tiny $3\times 3$ Averaging Kernel, $N=128$]{\label{fig:resFPDHF1}\includegraphics[width=0.32\textwidth]{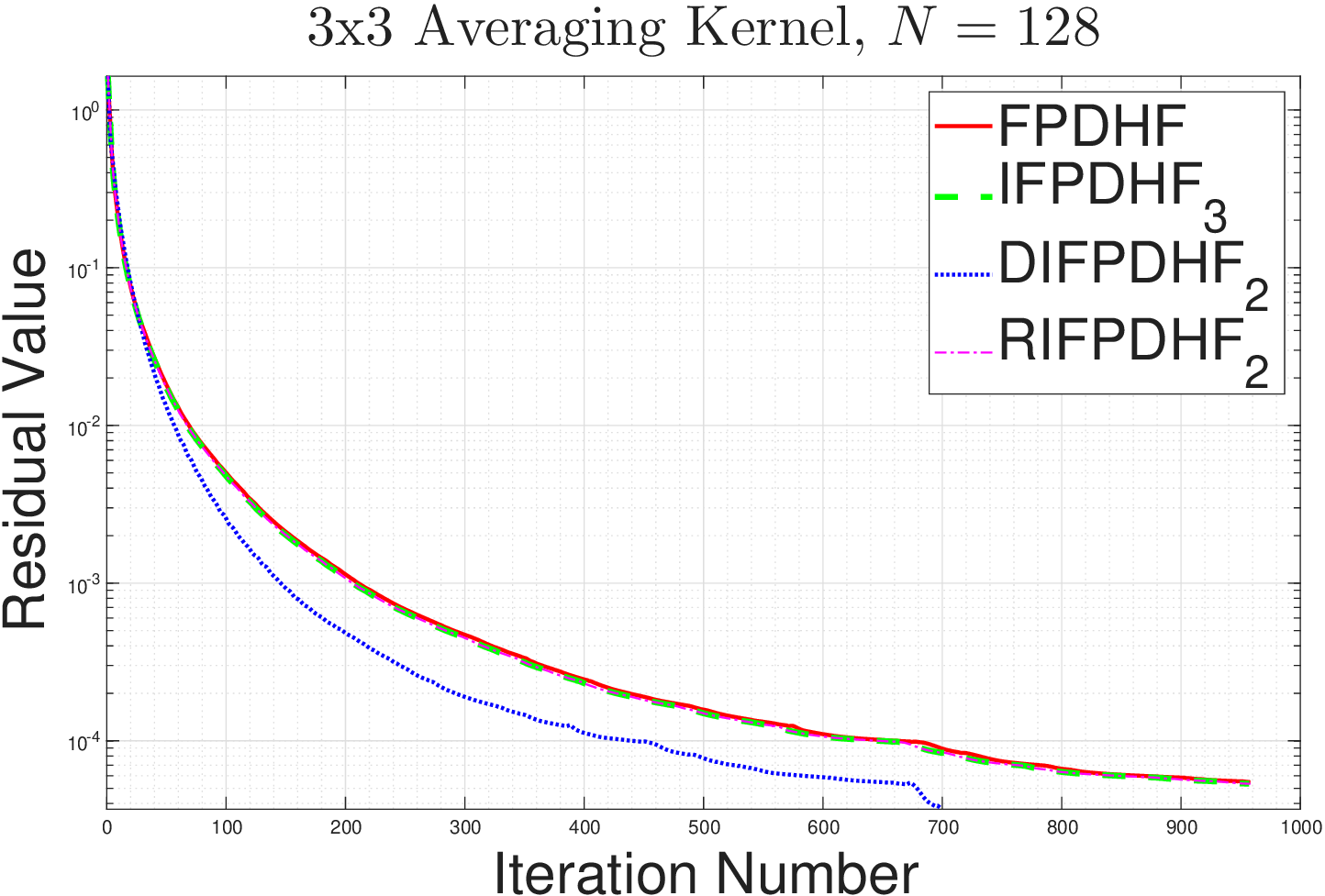}}\,
						\subfloat[\tiny $9\times 9$ Averaging Kernel, $N=256$]{\label{fig:resFPDHF2}\includegraphics[width=0.32\textwidth]{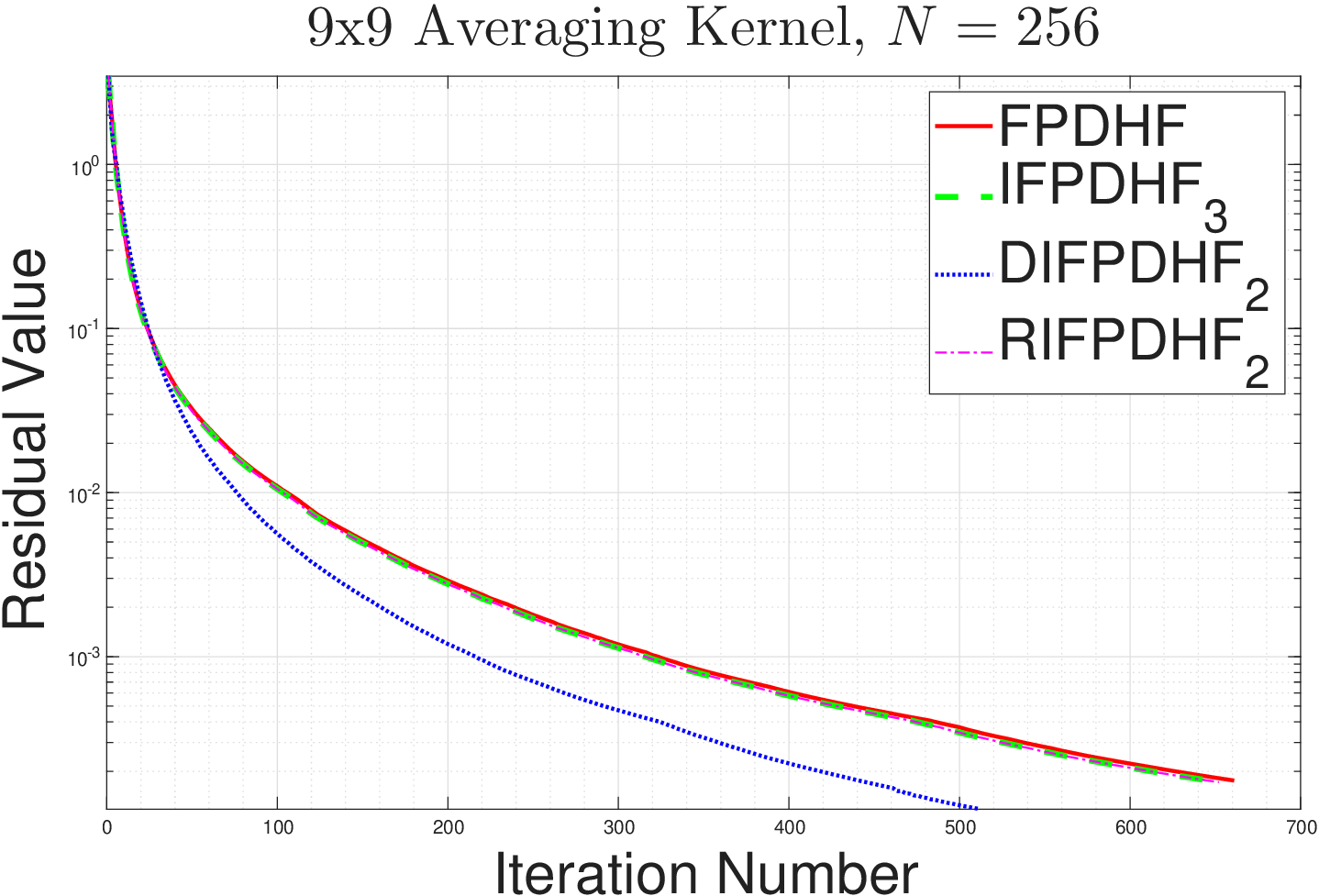}}
						\subfloat[\tiny $3\times 3$ Gaussian Kernel, $N=512$]{\label{fig:resFPDHF3}\includegraphics[width=0.32\textwidth]{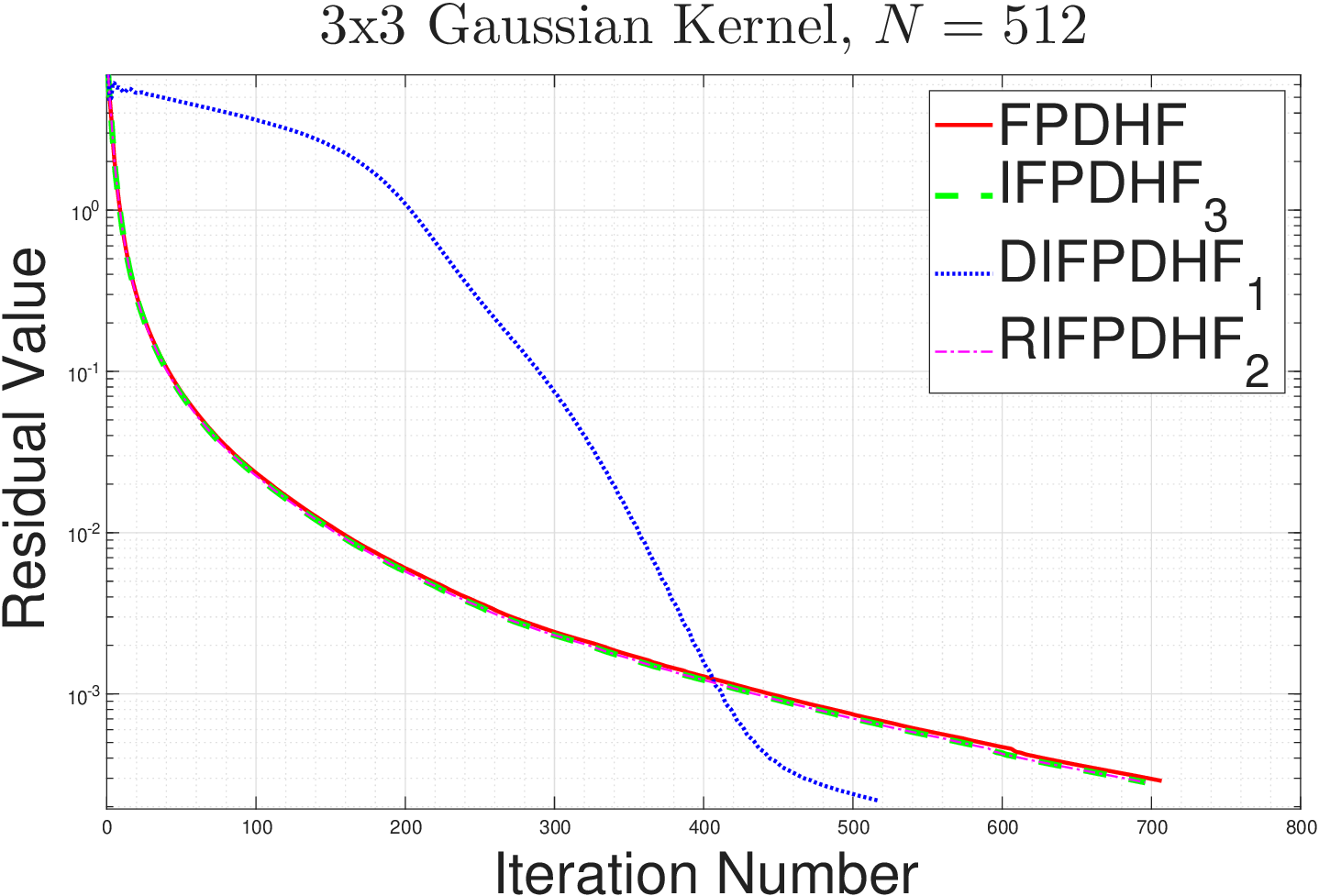}}
						\captionsetup{width=\textwidth} \caption{\cblue{Residual value ($\|(p_n-w_{n+1}-\tau(L^*(q_n-v_{n+1})+Cp_n-Cx_n),q_n-v_{n+1}+\sigma (L(w_{n+1}-p_n))\|$) for FPDHF and for the best performing instances of IFPDHF, DIFPDHF, and RIFPDHF. The graphs display the results for the first of the 20 random realizations of $z$.}}
						\label{fig:resFPDHF}
					\end{figure}
					\begin{figure}
						\centering
						\subfloat[\scriptsize Original Image]{\label{fig:x1281}\includegraphics[width=0.24\textwidth]{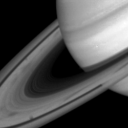}}\,
						\subfloat[\scriptsize $z_{20}$ (PSNR 29.90)]{\label{fig:x1282}\includegraphics[width=0.24\textwidth]{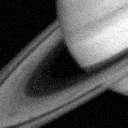}}\,
						\subfloat[\scriptsize FHRB (PSNR 32.24)]{\label{fig:x1283}\includegraphics[width=0.24\textwidth]{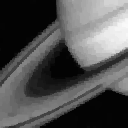}}\,
						\subfloat[\scriptsize FHRBR (PSNR 32.24)]{\label{fig:x1284}\includegraphics[width=0.24\textwidth]{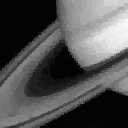}}
						\captionsetup{width=\textwidth} \caption{Original image, blurred and noisy observation $\#$ 20 with a blur of $3\times 3$ averaging kernel, and recovered images for FPDHF and DIFPDHF with $(\alpha_n^2)_{n \in \N}$, with their respective PSNR (dB), for $N=128$.} 
						\label{fig:x128}
					\end{figure}
					\begin{figure}
						\centering
						\subfloat[\scriptsize Original Image]{\label{fig:x2561}\includegraphics[width=0.24\textwidth]{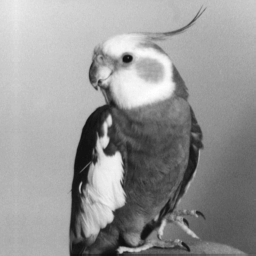}}\,
						\subfloat[\scriptsize $z_{20}$ (PSNR 24.56)]{\label{fig:x2562}\includegraphics[width=0.24\textwidth]{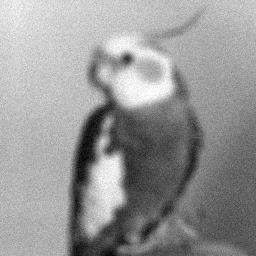}}\,
						\subfloat[\scriptsize FHRB (PSNR 27.94)]{\label{fig:x2563}\includegraphics[width=0.24\textwidth]{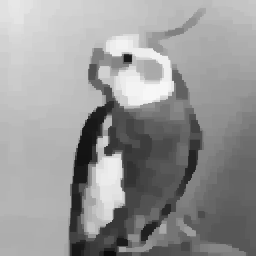}}\,
						\subfloat[\scriptsize FHRBR (PSNR 27.94)]{\label{fig:x2564}\includegraphics[width=0.24\textwidth]{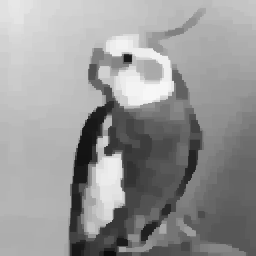}}
						\captionsetup{width=\textwidth} \caption{Original image, blurred and noisy observation $\#$ 20 with a blur of $9\times 9$ averaging kernel, and recovered images for FPDHF and DIFPDHF with $(\alpha_n^2)_{n \in \N}$, with their respective PSNR (dB), for $N=256$.} 
						\label{fig:x256}
					\end{figure}
					\begin{figure}
						\centering
						\subfloat[\scriptsize Original Image]{\label{fig:x5121}\includegraphics[width=0.24\textwidth]{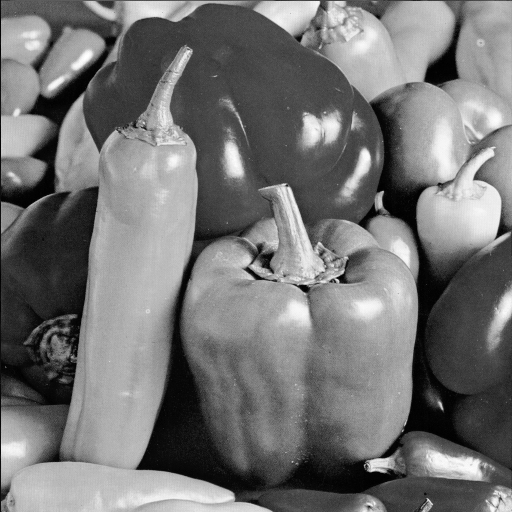}}\,
						\subfloat[\scriptsize $z_{20}$ (PSNR 28.17)]{\label{fig:x5122}\includegraphics[width=0.24\textwidth]{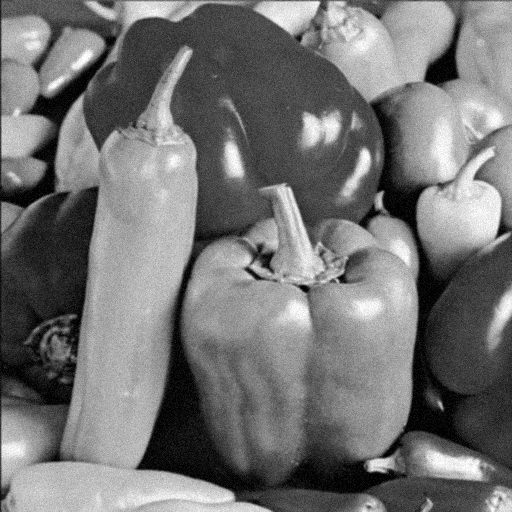}}\,
						\subfloat[\scriptsize FHRB (PSNR 31.43)]{\label{fig:x5123}\includegraphics[width=0.24\textwidth]{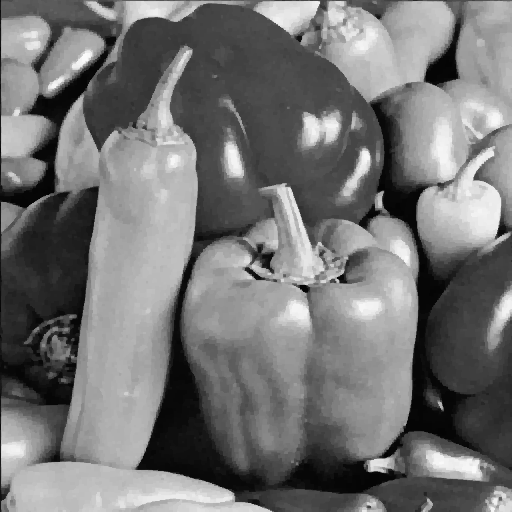}}\,
						\subfloat[\scriptsize FHRBR (PSNR 31.43)]{\label{fig:x5124}\includegraphics[width=0.24\textwidth]{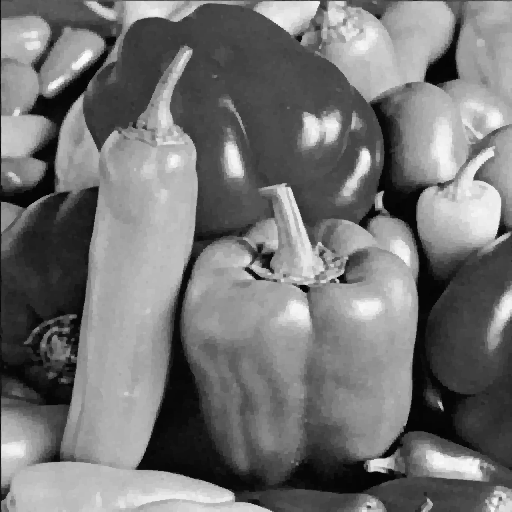}}
						\captionsetup{width=\textwidth} \caption{Original image, blurred and noisy observation $\#$ 20 with a blur of $3\times 3$ averaging kernel, and recovered images for FPDHF and DIFPDHF with $(\alpha_n^2)_{n \in \N}$, with their respective PSNR (dB), for $N=512$.} 
						\label{fig:x512}
					\end{figure}
					\cblue{
						\subsection{Linear convergence in denoising problems}\label{sec:numericallinearconvergence}
						In this section, we analyze the behavior of the FBF method, with and without the inertial step, in the case where the sum of the operators involved in the monotone inclusion is strongly monotone. In this setting, the algorithm converges linearly to the unique solution of the problem. To study this case, we consider the following denoising model. Let $N\in \N$ and let $x \in \R^{N\times N}$ be an image to be recovered from 
						an observation
						\begin{equation}\label{eq:modelim2}
							z = x+\epsilon,
						\end{equation}	
						where $\epsilon$ is an additive noise perturbing the observation. In this case, we aim to recover the original image $x$ by solving the following optimization problem.
						\begin{equation}\label{eq:opSM}
							\min_{x \in ^{N\times N}} 
							\frac{1}{2}\|x-z\|^2+\mu H_\delta(Wx),
						\end{equation}
						where $\mu \in \RPP$ is a regularization parameter, $H_\delta:\R^{N\times N}\to \R$ is the Huber function
						and  $W:\R^{N\times N}\to \R^{N\times N}$ is the orthogonal basis wavelet transform defined in Section~\ref{sec:imagerestoration}. By defining $A = \id-z$ and $D = \mu W\nabla H_\delta W$, this problem can be solved by FBF. Note that $A$ is $1$-strongly monotone and $D$ is $\mu/\delta$-Lipschitz. In our experiments we consider $N \in \{128,256,512\}$, a Gaussian noise with zero mean and variance $0.004$, $\mu = 0.07$, $\delta = 0.01$, and 20 random realizations of the observation $z$, for each dimension. The stopping criterion is a relative error with tolerance $10^{-9}$ and a maximum of $5000$ iterations. In this setting, we test the algorithm in \eqref{eq:algoFBHF} for $C=0$, $\lambda = 1$, $\tau = 0.9\delta/\mu$, and the values of $(\alpha_n)_{n \in \N}$ described in Table~\ref{Tab:FBFresult}. This table shows the results obtained in terms of the average number of iterations (IN) and the average CPU time (T). From this table, we observe the advantage of using a decreasing inertial sequence, which allows reaching the stopping criterion with fewer iterations and less computation time. The inertial sequence $\alpha_n = ({\sqrt{\mu/\delta+1}-1})({\sqrt{\mu/\delta+1}+1+0.0001n})^{-1}$,
						is motivated by \cite[Corollary~4.1]{Briceño2025Fista}. Although the analysis in \cite[Corollary~4.1]{Briceño2025Fista} is carried out for the FB method and strongly convex functions, the numerical experiments show good performance for this particular decreasing inertial sequence in FBF. This motivates further research aimed at finding optimal inertial sequences and convergence rates for NFB with inertia.
						\begin{table}[htbp]
							\centering
							\captionsetup{width=\linewidth}
							\renewcommand{\arraystretch}{1.2}
							\setlength{\tabcolsep}{6pt}\cblue{
								\begin{tabular}{cccccccc}
									&   & \multicolumn{2}{c}{$N=128$} & \multicolumn{2}{c}{$N=256$} & \multicolumn{2}{c}{$N=512$} \\
									\cmidrule(lr){3-8}
									Algorithm & $\alpha_n$  & IN & T & IN & T & IN & T \\
									\toprule
									FBF & 0  
									& 151& 0.11 
									& 148 & 0.36 
									& 149 & 3.21 \\
									\hline
									IFBF
									& $0.99\cdot2\left(\psi-1\right)\left(2\psi-1+\sqrt{8\psi-7}\right)^{-1}$ 
									& 138 & 0.11
									& 134 & 0.32
									& 136 & 3.00 \\
									\hline
									\multirow{2}{*}{DIFBF} 
									& $(9+0.00001\cdot n\cdot(\log(n))^{1.00001})^{-1}$   
									& 133 & 0.10 
									& 129 & 0.31 
									& 131 & 2.90 \\
									& {\footnotesize$(\sqrt{\mu/\delta+1}-1)(\sqrt{\mu/\delta+1}+1+0.0001n)^{-1}$}  
									& {\bf 71} & {\bf 0.05} 
									& {\bf 69} & {\bf 0.17} 
									& {\bf 74} & {\bf 1.65} \\
									\hline
								\end{tabular}
								\caption{\cblue{Comparison of FBF with and without inertial in the strongly monotone case. We consider the recurrence in \eqref{eq:algoFBHF} for $C=0$, $\lambda = 1$, and $\tau = 0.9\delta/\mu$. In this case $\psi={2}/{(1+\tau^2\zeta^2)}$. The best results obtained in terms of average number of iterations (IN) and average CPU time (T) are highlighted in bold.}}\label{Tab:FBFresult}}
						\end{table}
						\begin{figure}
							\centering
							\subfloat[$N=128$]{\label{fig:resFBF1}\includegraphics[width=0.32\textwidth]{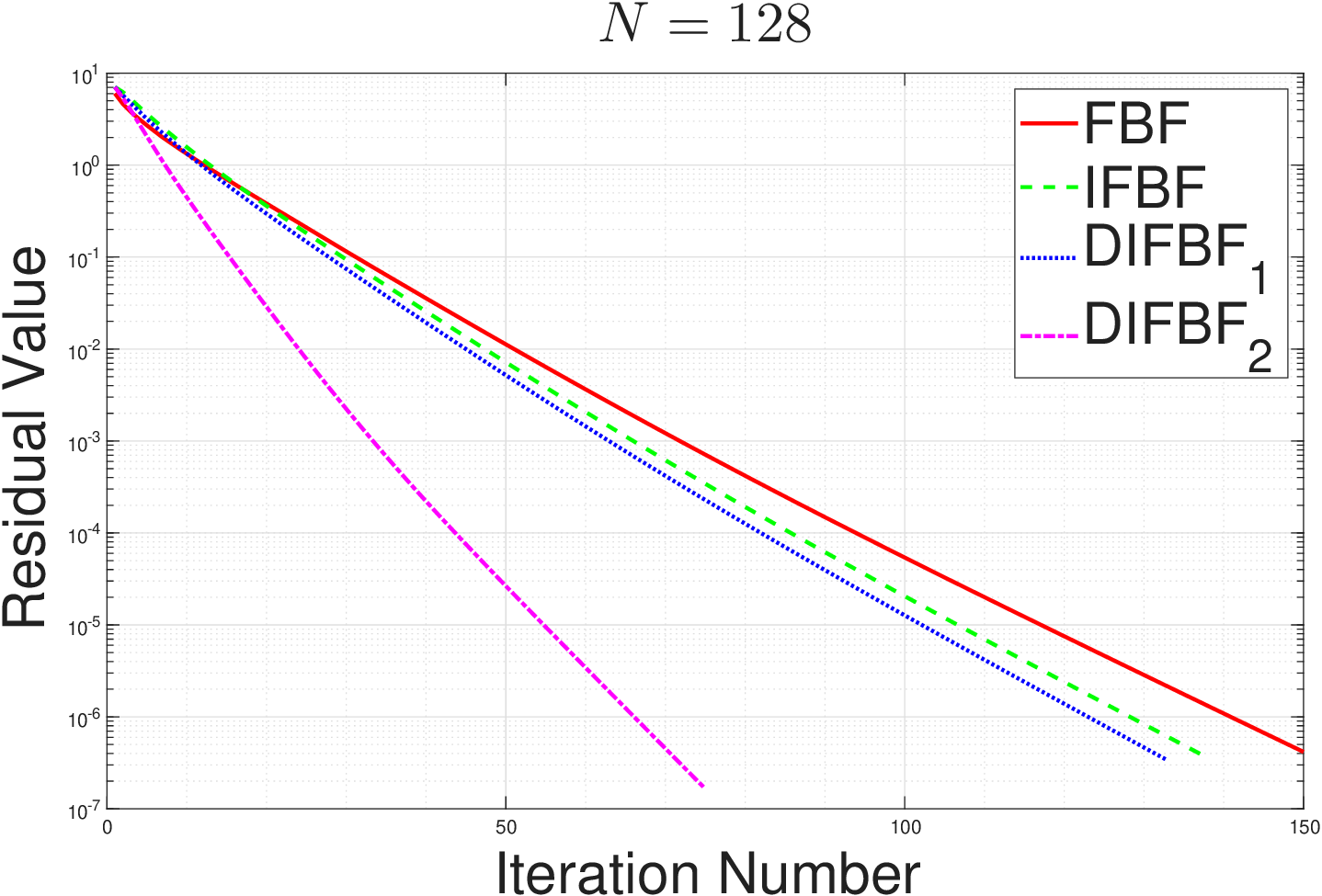}}\,
							\subfloat[$N=256$]{\label{fig:resFBF2}\includegraphics[width=0.32\textwidth]{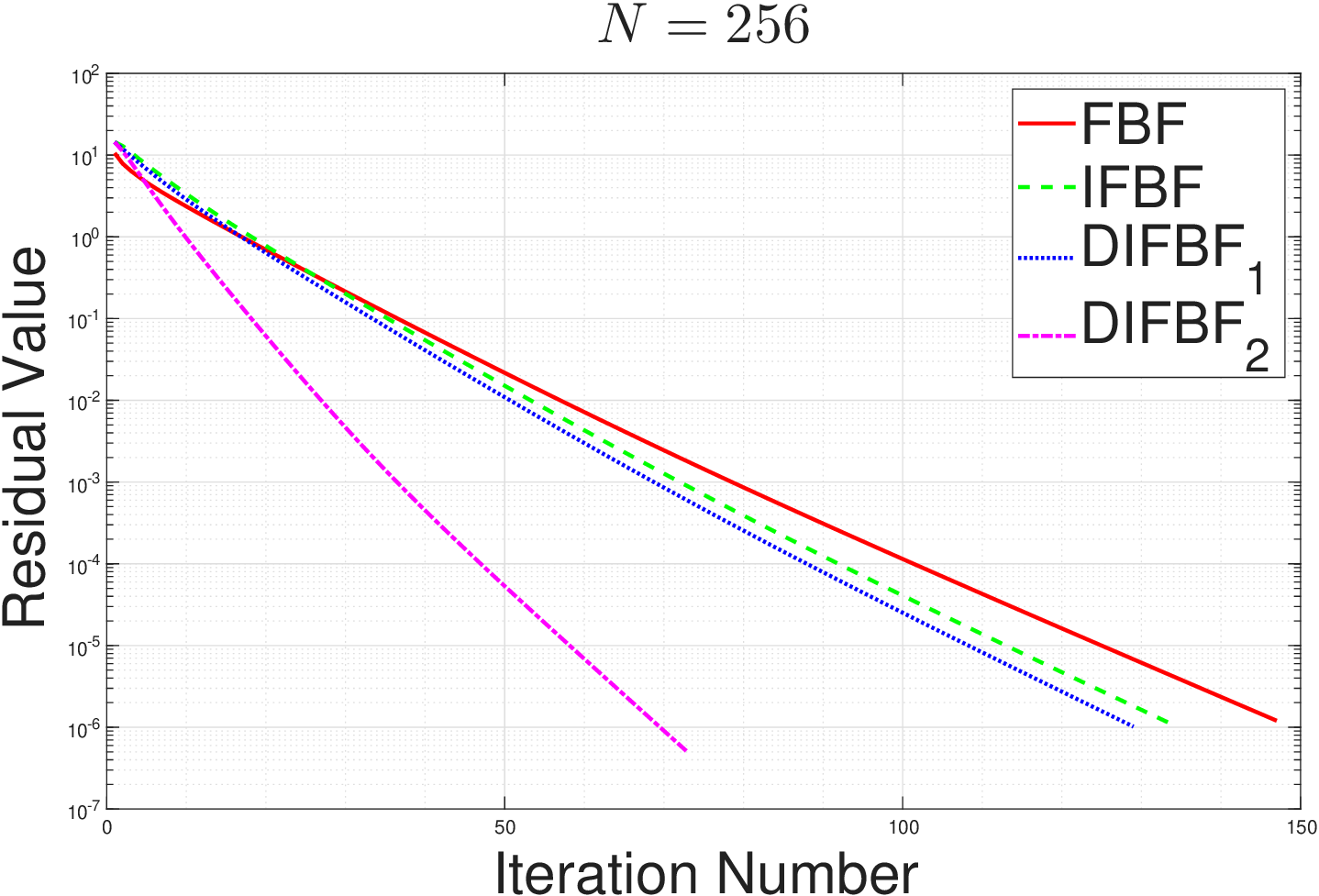}}
							\subfloat[$N=512$]{\label{fig:resFB3}\includegraphics[width=0.32\textwidth]{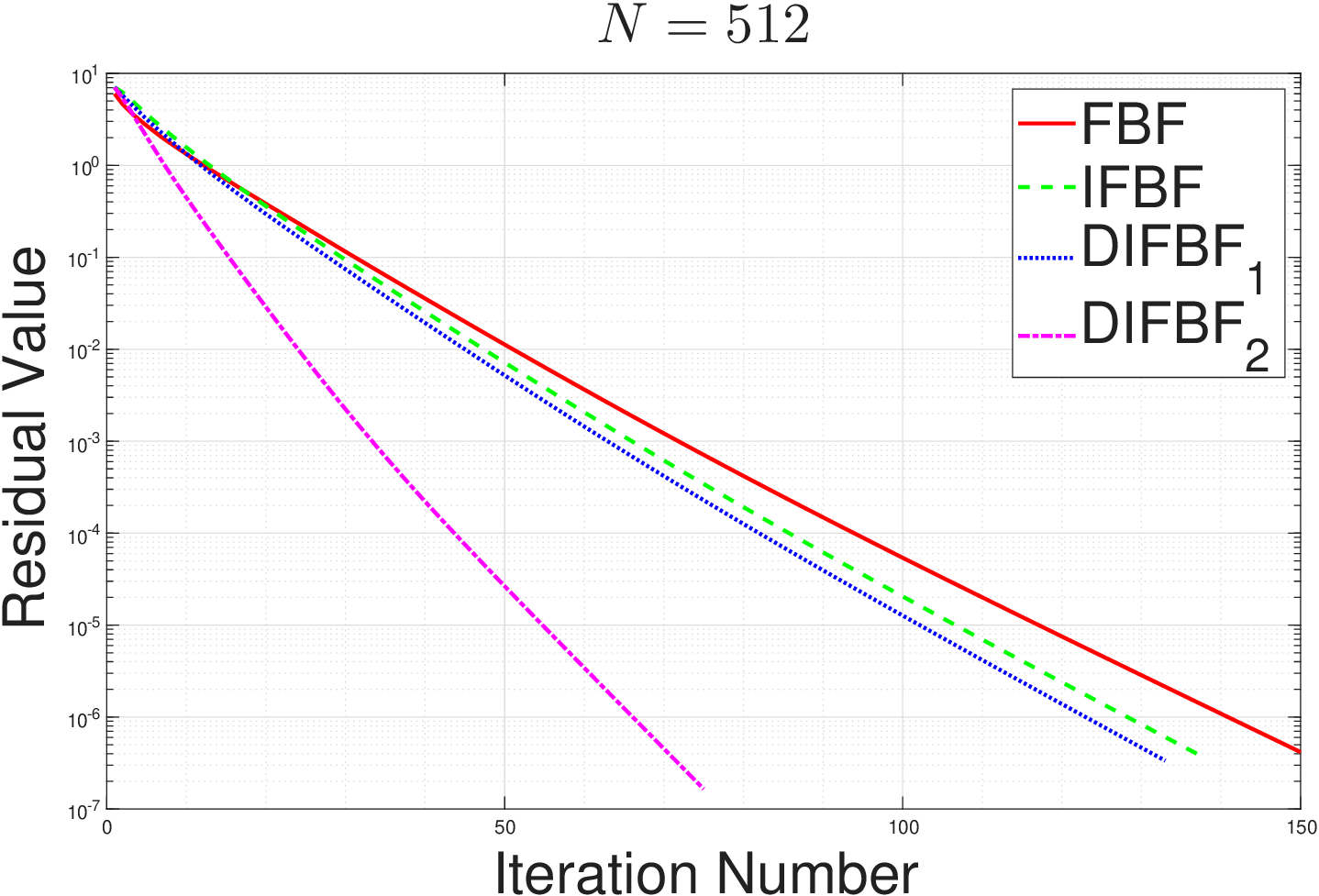}}
							\captionsetup{width=\textwidth} \caption{\cblue{Residual value ($\|x_n-z+D(x_n)\|$) for FBF, IFBF, and DIFBF. The graphs display the results for the first of the 20 random realizations of $z$.}}
							\label{fig:resFBF}
					\end{figure}}
					\section{Conclusions}\label{sec:conclu}
					In this article, we introduced an inertial and relaxed version of the Nonlinear Forward-Backward (NFB) algorithm and established its \cblue{weak and strong} convergence for both nondecreasing and decreasing sequences of inertial parameters. The analysis for decreasing sequences offers a novel perspective in the study of inertial methods. As particular cases, we recovered the convergence of the inertial and relaxed versions of algorithms such as Forward-Backward, Forward-Backward-Forward, Chambolle--Pock, and Condat--V\~u, extending previous results to the decreasing case. Furthermore, we derived new convergence results for the Forward-Backward-Half-Forward and Forward-Primal-Dual-Half-Forward algorithms incorporating inertial and relaxation steps. Finally, numerical experiments on optimization problems with affine constraints and image restoration problems demonstrated that using decreasing inertial sequences can enhance the practical performance of these methods by accelerating convergence. A drawback of this approach is the need to choose an appropriate decreasing inertial sequence for each problem. Nevertheless, this acceleration is particularly useful in scenarios where one needs to solve multiple instances of a problem with fixed operators and parameters but varying observations, as in image reconstruction or restoration. \cblue{Hence, these results, and in particular the \emph{singular} choice of the inertial sequence that yields the favorable numerical results observed in Section~\ref{sec:numericallinearconvergence} for strongly monotone operators, motivate further research aimed at determining the rate of convergence in both the monotone and strongly monotone cases, as well as identifying the optimal inertial and relaxation parameters to enhance convergence.}
					\section*{Acknowledgments}
					The first author was partially supported by ANID-Chile grant Exploración 13220097, Centro de Modelamiento Matemático (CMM) BASAL fund FB210005 for centers of excellence, from ANID-Chile, and FONDECYT Postdoctorado Grant 3250609. The second author was partially supported by ANID under grant FONDECYT Iniciación 11250164. The third author was partially supported by Universidad de Tarapac\'a internal research project Fortalecimientos Grupos de Investigaci\'on C\'odigo 8802-25.

					\bibliographystyle{spmpsci.bst}
					\bibliography{ref}

				\end{document}